\documentclass{article}
\usepackage[T1]{fontenc}
\usepackage[english]{babel}
\usepackage{graphicx}
\usepackage{stmaryrd}
\usepackage{amsmath,amsfonts,amssymb}

\usepackage{ntheorem}
\usepackage{dsfont}
\usepackage{slashed}
\usepackage{hyperref}
\usepackage{tikz}
\usepackage{pgfplots}
\usetikzlibrary{patterns}
\usetikzlibrary{positioning,arrows,arrows.meta}
\usepackage{geometry}
\usepackage{enumitem}
\theoremseparator{.} 
\newtheorem{Th}{Theorem}[section]
\newtheorem{Def}[Th]{Definition}
\newtheorem{Rq}[Th]{Remark}
\newtheorem{Pro}[Th]{Proposition}
\newtheorem{Cor}[Th]{Corollary}
\newtheorem{Lem}[Th]{Lemma}
\newcommand{\R}{\mathbb{R}}
\newcommand{\TT}{\mathbf{T}}
\newcommand{\E}{\mathbb{E}}
\newcommand{\Po}{\mathbb{P}}
\newcommand{\I}{\mathfrak{I}}
\newcommand{\Or}{\mathbb{O}}
\newcommand{\Sp}{\mathbb{S}}
\newcommand{\Si}{\overline{\Sigma}}
\newcommand{\Vv}{\mathbb{V}}
\newcommand{\V}{\mathbf{k}_1}
\newcommand{\K}{\widehat{\mathbb{P}}_0}

\newenvironment{proof}{\noindent\textit{Proof.~}}{\hfill$\square$\bigbreak} 
\newlength{\plarg}
\title{Asymptotic properties of the solutions to the Vlasov-Maxwell system in the exterior of a light cone}

\author{L\'eo Bigorgne\footnote{Laboratoire de Mathématiques, Univ. Paris-Sud, CNRS, Université Paris-Saclay, 91405 Orsay. E-mail address: leo.bigorgne@u-psud.fr.}}
\date{}
\begin{document}

\maketitle
    
\begin{abstract}

This paper is concerned with the asymptotic behavior of small data solutions to the three-dimensional Vlasov-Maxwell system in the exterior of a light cone. The plasma does not have to be neutral and no compact support assumptions are required on the data. In particular, the initial decay in the velocity variable of the particle density is optimal and we only require an $L^2$ bound on the electromagnetic field with no additional weight. We use vector field methods to derive improved decay estimates in null directions for the electromagnetic field, the particle density and their derivatives. In contrast with \cite{dim3}, where we studied the behavior of the solutions in the whole spacetime, the initial data have less decay and we do not need to modify the commutation vector fields of the relativistic transport operator. To control the solutions under these assumptions, we crucially use the strong decay satisfied by the particle density in the exterior of the light cone, null properties of the Vlasov equation and certain hierarchies in the energy norms.
\end{abstract}

    \tableofcontents
\section{Introduction}

In this article, we study the asymptotic properties of small data solutions of the Vlasov-Maxwell (VM) system in the exterior of a light cone $V_b := \{ (t,x) \in \R_+ \times \R^3 \hspace{1mm} / \hspace{1mm} |x| > t-b \}$, where, say, $b \leq -1$. The system, which is of particular importance in plasma physics, is given for one species of particles by\footnote{We choose to lighten the notations by considering only one species since the presence of other ones does not complicate the analysis.}$^{,}$\footnote{We will, throughout this article, use the Einstein summation convention so that $v^{i} \partial_{i} f = \sum_{i=1}^3 v^i \partial_{i} f$. A sum on latin letters starts from $1$ whereas a sum on greek letters starts from $0$.}
\begin{eqnarray}\label{VM1}
\TT_F(f) \hspace{1mm} := \hspace{1mm} v^0\partial_t f+v^i \partial_i f +ev^{\mu} {F_{\mu}}^j \partial_{v^j} f & = & 0, \\ \label{VM2}
\nabla^{\mu} F_{\mu \nu} & = &  J(f)_{\nu} \hspace{1mm} := \hspace{1mm} \int_{v \in \R^3} \frac{v_{\nu}}{v^0} f dv , \\ \label{VM3}
\nabla^{\mu} {}^* \!  F_{\mu \nu} & = & 0,
\end{eqnarray}
where
\begin{itemize}
\item $v^0 = \sqrt{m^2+|v|^2}$, $m>0$ is the mass of the particles and $e \neq 0$ their charge. For the remaining of this paper, we take $m=e=1$ and we denote $\sqrt{1+|v|^2}$ by $v^0$.
\item The function $f(t,x,v)$ is the particle density, the $2$-form $F(t,x)$ is the electromagnetic field and ${}^* \!  F(t,x)$ is its Hodge dual.
\end{itemize}

\subsection{Small data results for the VM system}

The study of the small data solutions of the VM system has been initiated in \cite{GSt} by Glassey-Strauss. Under a compact support assumption in space and in velocity on the initial data, they proved global existence and obtained the optimal decay rate on $\int_v f dv$. The compact support assumption in $v$ is replaced by Schaeffer in \cite{Sc} by a polynomial decay but the data still have to be compactly supported in space. Moreover, the optimal decay rate on $\int_v f dv$ is not obtained by this method. None of these results contain estimates on the derivatives of $\int_v f dv$ nor on the higher order derivatives of the electromagnetic field. 

In \cite{dim4}, we removed all compact support assumptions for the dimensions $d \geq 4$. For this, we used vector field methods, developped in \cite{CK} for the electromagnetic field and \cite{FJS} for relativistic transport equations. We then obtained almost optimal decay on the solutions of the system and their derivatives and we described precisely the behavior each null component of the electromagnetic field. We recently extended these results to the $3d$ case and we also relaxed the assumptions on the initial data, allowing in particular the presence of a non zero total charge. A better understanding of the null structure of the VM system as well as the use of modified vector fields\footnote{Modified Vector fields, which depend on the solution itself, were already used by \cite{FJS2} (respectively \cite{FJS3}) in the context of the Vlasov-Nordstr\"om (respectively the Einstein-Vlasov) system. They are built over the commutation vector fields of the relativistic transport operator $v^{\mu} \partial_{\mu}$ in order to compensate the worst source terms of the commuted Vlasov equation.} were the key for dealing with the slower decay rates of the solutions. We splitted the electromagnetic field into two parts. The chargeless one on which we could then propagate a weighted $L^2$ norm and the pure charge part, given by an explicit formula, which decays as $\epsilon r^{-2}$ despite of its infinite energy.

We also investigate the case where the particles are massless (i.e. $m=0$). First in \cite{dim4} for the high dimensions, where we proved that similar results to the massive case hold provided that the velocity support of the the particle density is bounded away from $0$. These extra hypothesis appears to be necessary since we also proved in \cite{dim4} that the VM system do not admit a local classical solution for certain smooth initial data which do not vanish for small velocities. Secondly, in our recent work \cite{massless}, we proved sharp asymptotics on the small data solutions and their derivatives to the massless VM system in $3d$. Contrary to the massive case, the proof does not require the use of modified vector fields but still necessitates a strong understanding of the null properties of the system.

In this article, we study the asymptotic properties of the solutions to the VM system in the exterior of a light cone under a smallness assumption but weaker decay near infinity. We obtain in particular almost optimal pointwise decay estimates on the velocity average of the Vlasov field as well as its derivatives. The hypotheses on the particle density in the variable $v$ are optimal in the sense that we merely suppose $f$ and its derivatives to be initially integrable in $v$, which is a necessary condition for the source term of the Maxwell equations \eqref{VM2} to be well defined. As $f$ strongly decay in the domain studied, our proof merely requires the boundedness of the $L^2$ norm of the electromagnetic field. This has to be compared with our proof in \cite{dim3}, where we study the same problem in the whole spacetime, which crucially relies on the propagation of a weighted energy norm of $F$. Another remarkable point, still related to the good behavior of $f$ in the region $V_b$, concerns the commutation vector fields used to study the Vlasov equation. Contrary to \cite{dim3}, we do not need to modify the commutation vector fields of the relativistic transport operator $v^{\mu} \partial_{\mu}$ in order to compensate the worst source terms of the commuted Vlasov equations and then close the energy estimates. This leads in particular to a much simpler proof.

Finally, let us mention the recent result \cite{Wang} of Wang concerning the small data solutions of the massive $3d$ VM system. Using both vector field methods and Fourier analysis, he proved optimal pointwise decay estimates on $\int_v f dv$ and its derivatives under strong polynomial decay hypotheses in $(x,v)$ on $f(t=0)$. In particular, the initial data does not have to be compactly supported.

\subsection{Previous works on Vlasov systems using vector field methods}

Results on the asymptotic behavior of solutions of several Vlasov systems were recently derived using vector field methods. Let us mention the pioneer work \cite{FJS} of Fajman-Joudioux-Smulevici on the Vlasov-Norström system (see also \cite{FJS3}) as well as the results of \cite{Poisson} on the Vlasov-Poisson system. The two different proofs, obtained independently by \cite{FJS2} and \cite{Lindblad}, of the stability of the Minkowski spacetime as a solution to the Einstein-Vlasov system constitute a culmination of these vector field methods.

\subsection{Statement of the main result}
In order to present our main theorem, we call initial data set for the VM system any ordered pair $(f_0,F_0)$ where $f_0 : \R^3_x \times \R^3_v \rightarrow \R$ and $F_0$ are both sufficiently regular and satisfy the constraint equations
$$\nabla^i (F_{0})_{i0} =- \int_{v \in \R^3} f_0 dv \hspace{10mm} \text{and} \hspace{10mm} \nabla^i {}^* \! (F_0)_{i0} =0.$$ 
We precise that $\tau_+=\sqrt{1+(t+r)^2}$, $\tau_-=\sqrt{1+(t-r)^2}$ and we refer to Section \ref{sec2} for the notations not yet defined.
\begin{Th}\label{theorem}
Let $N \geq 9$, $b \leq -1$, $0 < \delta < \frac{1}{8N}$, $\epsilon >0$, $(f_0,F_0)$ an initial data set for the Vlasov-Maxwell equations \eqref{VM1}-\eqref{VM3} satisfying\footnote{We could assume $N\geq 8$ and save almost $\frac{15}{4}$ powers of $x$ in the condition on the initial norm of $f_0$ with easy but cumbersome modifications of our proof (see Remarks \ref{rqdecayforintro}). Note also that following the strategy used in Subsection $17.2$ of \cite{FJS3} to derive $L^2$ estimates on the Vlasov field, we could avoid any hypotheses on the derivatives of order $N+1$ and $N+2$ of $F_0$.}
$$ \sum_{ |\beta|+|\kappa| \leq N+3} \int_{|x| \geq b} \int_{v \in \R^3} (1+|x|)^{\frac{N+14+|\beta|}{2}}(1+|v|)^{|\kappa|} \left| \partial_x^{\beta} \partial_v^{\kappa} f_0 \right| dv dx +\sum_{  |\gamma| \leq N+2} \int_{|x| \geq b} (1+|x|)^{2|\gamma|} \left| \nabla_{\partial^{\gamma}_x} F_0 \right|^2 dx \leq \epsilon$$
and $(f,F)$ be the unique classical solution of the system which satisfies $f(t=0)=f_0$ and $F(t=0)=F_0$. Then, there exists $C>0$ and $\epsilon_0>0$, depending only on $N$ and $\delta$, such that, if $0 \leq \epsilon \leq \epsilon_0$, $(f,F)$ is well defined in $V_b = \{ (t,x) \in \R_+ \times \R^3 \hspace{1mm} / \hspace{1mm} r > t-b \}$ and verifies the following estimates.
\begin{itemize}
\item Energy bound for the electromagnetic field $F$: $\forall$ $t \in \R_+$, 
$$ \sum_{0 \leq k \leq N} \sum_{ Z^{\gamma} \in \mathbb{K}^{k} } \int_{|x| \geq t-b}   \left| \mathcal{L}_{ Z^{\gamma}}(F) \right|^2dx \leq C\epsilon ,$$
\item Pointwise decay estimates for the null components of $\mathcal{L}_{Z^{\gamma}}(F)$: $\forall$ $|\gamma| \leq N-2$, $(t,x) \in V_b$,
$$ \left| \underline{\alpha} \left( \mathcal{L}_{Z^{\gamma}}(F) \right) \right|(t,x) \lesssim  \frac{\sqrt{\epsilon}}{\tau_+ \tau_-^{\frac{1}{2}} }, \hspace{12mm} \left| \alpha \left( \mathcal{L}_{Z^{\gamma}}(F) \right) \right|(t,x)+\left| \rho \left( \mathcal{L}_{Z^{\gamma}}(F) \right) \right|(t,x)+ \left| \sigma \left( \mathcal{L}_{Z^{\gamma}}(F) \right) \right|(t,x) \lesssim \frac{\sqrt{\epsilon}}{\tau_+^{\frac{3}{2}}}.$$
\item Energy bound for the particle density: $\forall$ $t \in \R_+$,
$$ \sum_{0 \leq k \leq N} \sum_{ \widehat{Z}^{\beta} \in \K^{k} } \int_{|x| \geq t-b} \int_{v \in \R^3} \left| \widehat{Z}^{\beta} f \right| dv dx \leq C\epsilon (1+t)^{\delta}.$$

\vspace{-5mm}

\item Pointwise decay estimates for the velocity averages of $\widehat{Z}^{\beta} f$: $\forall$ $|\beta| \leq N-4$, $(t,x) \in V_b$,
$$\int_{ v \in \R^3} \left| \widehat{Z}^{\beta} f \right| dv \lesssim \frac{\epsilon}{\tau_+^{2-\delta} \tau_-} \hspace{8mm} \text{and} \hspace{8mm} \forall \hspace{0.5mm} a \in \left[ 1,4 \right], \hspace{5mm} \int_{ v \in \R^3} \left| \widehat{Z}^{\beta} f \right| \frac{dv}{(v^0)^{2a}} \lesssim  \frac{\epsilon}{\tau_+^{3+a-\delta} }.$$
\end{itemize}
\end{Th}

\vspace{-4mm}

\begin{Rq}
Note that we can study the solutions to the Vlasov-Maxwell equations in the exterior of a light cone, without any information on their behavior in the remaining part of the Minkowski space, by finite speed of propagation. Every inextendible past causal curves of such a region intersect the hypersurface $t=0$ once and only once, i.e. the region is globally hyperbolic. 
\end{Rq}
\begin{Rq}
By a time translation, one can prove a similar result for $b \in \R$ ($\epsilon_0$ would then also depends on $b$).
\end{Rq}
\begin{Rq}
The pointwise decay estimates on the electromagntic field are sharp. For the Vlasov field, we could derive the almost sharp decay estimate $\int_{\R^3_v} |f| \frac{dv}{|v^0|^{N+10}} \lesssim \epsilon \tau_+^{-3-\frac{N+8}{2}+\delta}$ (see Proposition \ref{Cordecayopti}) but we would have to propagate a less convenient energy norm.
\end{Rq}
\begin{Rq}
Assuming more decay on the electromagnetic field at $t=0$, one could propagate a stronger energy norm as in \cite{dim3} or \cite{massless}. We then could assume less decay decay in $x$ on $f_0$ and improve the decay rate of the null components of the electromagnetic field. Note however that if the total electromagnetic charge

\vspace{-1mm}

$$ Q(F)(t) \hspace{1mm} := \hspace{1mm} \lim_{r \rightarrow + \infty} \int_{\mathbb{S}_{t,r}} \frac{x^i}{r} F_{0i} d\mathbb{S}_{t,r} \hspace{1mm} = \hspace{1mm} - \lim_{r \rightarrow + \infty} \int_{\mathbb{S}_{t,r}} \rho (F) d\mathbb{S}_{t,r} \hspace{1mm} = \hspace{1mm} \int_{x \in \R^3} \int_{v \in \R^3} f dx dv, $$

\vspace{-1mm}

\noindent which is a conserved quantity in $t$, is non zero, we cannot obtain a better decay rate than $r^{-2}$ on $\rho (F)$ and assume that $\int_{\R^3} r |\rho(F)|^2 dx$ is initially finite. We point out that our hypotheses on the electromagnetic field are compatible with the presence of a non zero total charge.
\end{Rq}
\begin{Rq}
The results of \cite{FJS} and \cite{FJS2} are obtained using a hyperboloidal foliation and then require compactly supported initial data in space. These compact restrictions on the data could be removed by adapting the method used in this article to the Vlasov-Nordström system, provided that we could prove that a certain $L^1$ norm on a hyperboloidal slice of the particle density is small. For that purpose, the estimate on the energy norm of the Vlasov field given by Theorem \ref{theorem} is not good enough since it is not uniform in $t$. This issue can be solved by assuming a slightly stronger decay hypothesis on either the electromagnetic field or on the Vlasov field (see Subsection \ref{subsecuniformbound}).
\end{Rq}
Theorem \ref{theorem} immediately implies the following result, concerning solutions arising from large data.
\begin{Cor}
Let $N \geq 8$ and $(f_0,F_0)$ an initial data set for the Vlasov-Maxwell equations \eqref{VM1}-\eqref{VM3} satisfying
$$ \sum_{|\beta|+|\kappa| \leq N+3} \int_{x \in \R^3} \int_{v \in \R^3} (1+|x|)^{\frac{N+14+|\beta|}{2}}(1+|v|)^{|\kappa|} \left| \partial_x^{\beta} \partial^{\kappa}_v f_0 \right| dv dx + \sum_{|\gamma| \leq N} \int_{x \in \R} (1+|x|)^{2|\gamma|} |\nabla_{\partial^{\gamma}_x} F_0|^2 dx < + \infty$$
and $(f,F)$ be the unique local classical solution to the system which satisfies $f(t=0)=f_0$ and $F(t=0)=F_0$. Then, there exists $b \leq -1$ such that $(f,F)$ is well defined in $V_b$ and verifies similar estimates as those presented in Theorem \ref{theorem}.
\end{Cor}
\begin{proof}
One only has to notice that there exists $b \leq -1$ such that $(f_0,F_0)$ satisfies the hypotheses of Theorem \ref{theorem}.
\end{proof}

\noindent Although global existence in the whole Minkowski spacetime for classical solutions to the VM system with large data still remains an open problem, several continuation critera are known (see for instance the recent results of \cite{LukStrain}, \cite{Kunze} and \cite{Patel}). For the weak solutions, the problem was solved in \cite{DL} and revisited in \cite{Rein}.

\subsection{Main ingredients of the proof}

We present here the strategy of the proof and the main differences with our study in \cite{dim3} of small data solutions to the VM system in the whole Minkowski spacetime. The proof of the main result of this paper is based on vector field methods and then essentially relies on bounding sufficiently well the spacetime integrals of the source terms of the commuted equations. For instance, at first order, if $Z$ is a Killing vector field and $\widehat{Z}$ its complete lift,
\begin{equation}\label{comutationinto}
[\TT_F,\widehat{Z}](f) \hspace{1mm} = \hspace{1mm} - v^{\mu} {\mathcal{L}_Z(F)_{\mu}}^i \partial_{v^i} f.
\end{equation}
Since the electromagnetic field decay initially as $r^{-\frac{3}{2}}$, one cannot expect a better estimate than $|\mathcal{L}_Z(F)|(t,x) \lesssim (1+t+r)^{-1}(1+|t-r|)^{-\frac{1}{2}}$. Moreover, since $v^0\partial_{v^i}f$ behaves as $(t+r)\partial_{t,x}f+\widehat{\Omega}f$, where $\widehat{\Omega}$ is part of the commutation vector fields $\widehat{\mathbb{P}}_0$ of the transport equation, we have the following naive inequality
\begin{equation}\label{esticomutationintro}
\left| [\TT_F,\widehat{Z}](f)\right| \hspace{1mm} \lesssim \hspace{1mm} \frac{1}{(1+|t-r|)^{\frac{1}{2}}}|\partial_{t,x}f |+\frac{1}{(1+t+r)(1+|t-r|)^{\frac{1}{2}}}\sum_{\widehat{\Omega} \in \widehat{\mathbb{P}}_0 }|\widehat{\Omega}f |.
\end{equation}
We observe that there are two obstructions for closing the energy estimates by performing, say, a Grönwall inequality from \eqref{esticomutationintro}.
\begin{enumerate}
\item The decay rate $(1+|t-r|)^{-\frac{1}{2}}$ is non integrable.
\item The decay rates of both terms in the right hand side of \eqref{esticomutationintro} degenerate near the light cone $t-r=0$. We will deal with this issue by taking advantage of the null structure of the system.
\end{enumerate}

\subsubsection{Dealing with non integrable decay rates}

The problem of using the decay in $t-r$ or transforming it into a decay in $t+r$ will be adressed later and we only focus here on the first problem. In \cite{dim3}, our proof required to propagate a stronger energy norm on the electromagnetic field, leading, by neglecting a logarithmal growth, to the estimate $|\mathcal{L}_Z(F)|(t,x) \lesssim (1+t+r)^{-1}(1+|t-r|)^{-1}$. Because of the presence of a charge\footnote{Note that even in the case of a neutral plasma, the slow decay rate of the Vlasov field in the interior region would prevent us to improve this pointwise estimate with our method.}, this inequality is optimal for large $r$. This gave us a better, but still not satisfying, naive estimate than \eqref{esticomutationintro}. To deal with this issue, we considered modifications of the commutation vector fields of the form $Y=\widehat{Z}+\Phi^{\nu} \partial_{\nu}$. The coefficients $\Phi$ depend on the solution itself, growth logarithmically and are constructed in order to compensate the worst error terms in \eqref{comutationinto}. In the exterior of the light cone $V_0$, the solutions to the Vlasov equation behave better than in the interior region. One can already see that with the following estimate (see Lemma \ref{weights1}), for $g$ a solution to the free transport equation $v^{\mu} \partial_{\mu} g=0 $,

\begin{equation}\label{extradecayintro}
 \forall \hspace{0.5mm} |x| \geq t, \quad \int_{v \in \R^3} |g|(t,x,v) dv \hspace{2mm} \lesssim \hspace{2mm} \sum_{|\beta|+|\kappa| \leq 3} \frac{\| (v^0)^{|\kappa|+2k} (1+r)^{|\beta|+k+q} \partial^{\beta}_{t,x} \partial_v^{\kappa} g \|_{L^1_{x,v}}(t=0)}{(1+t+r)^{2+k}(1+|t-r|)^{1+q}},
\end{equation}
whereas
\begin{equation}\label{extradecayintro2}
 \forall \hspace{0.5mm} |x| \leq t, \quad \int_{v \in \R^3} |g|(t,x,v) dv \hspace{2mm} \lesssim \hspace{2mm} \frac{(1+|t-r|)^{k-1}}{(1+t+r)^{2+k}} \sum_{|\beta|+|\kappa| \leq 3} \| (v^0)^{|\kappa|+2k} (1+r)^{|\beta|+k} \partial_{t,x}^{\beta} \partial_{v}^{\kappa} g \|_{L^1_{x,v}}(t=0).
\end{equation}
One can then observe in \eqref{extradecayintro2} that, in the interior of the light cone, $ \int_{v \in \R^3} |g|(t,x,v) dv$ has a stronger decay rate along null rays, where $t-r=constant$, than along timelike curves. In, say, the region $|x| \leq \frac{t}{2}$, one cannot expect a better decay rate than $(1+t)^{-3}$. The situation is completely different in the exterior of the light cone since one can obtain arbitrary decay through \eqref{extradecayintro} provided that the initial data decay sufficiently fast. The idea is to capture this phenomenon by considering and controlling weighted $L^1$ norms adapted to our problem and by identifying several hierarchies in the commuted equations.
\begin{Rq}
Contrary to \cite{dim3}, this strong decay allows us to avoid the use of modified vector fields. This is also what allows us to assume less decay on the electromagnetic field and then to avoid any difficulty due to the presence of a non zero total charge.
\end{Rq}
Let us illustrate how we will deal with the weak decay rate of $F$ and which kind of hierarchies we will use.
\begin{itemize}
\item The translations provide extra decay in $t-r$. More precisely, $|\mathcal{L}_{\partial_{x^{\mu}}}(F)| \lesssim (1+t+r)^{-1}(1+|t-r|)^{-\frac{3}{2}}$ and the estimate \eqref{esticomutationintro} can be improved if $Z=\partial_{x^{\mu}}$. As all the error terms have an integrable decay rate in that case, one can expect to prove that $\| \partial_{t,x} f \|_{L^1_{x,v}}$ is uniformly bounded in time.
\item If $Z \neq \partial_{x^{\mu}}$, the worst error terms only contain standard derivatives $\partial_{t,x} f$ of the Vlasov field. This makes in some sens the system of the commuted equations triangular. However, at this point, we cannot expect to prove a better bound than $\| \widehat{Z} f \|_{L^1_{x,v}}(t) \lesssim \sqrt{1+t}$.
\item We know that there exists $z$, which is a combination of weights preserved by $v^{\mu} \partial_{\mu}$, such that $1+|t-r| \lesssim z$ (see Subsection \ref{sectionweights}). Hence, if $g$ is a sufficiently regular solution to $v^{\mu} \partial_{\mu} g=0$, $\| z g \|_{L^1_{x,v}}$ is constant. This indicates that we should be able to bound sufficiently well $\| z \partial_{t,x} f \|_{L^1_{x,v}}$ and then $\| \widehat{Z} f \|_{L^1_{x,v}}$ since
$$ \left| [\TT_F,\widehat{Z}](f)\right| \hspace{1mm} \lesssim \hspace{1mm} \frac{1}{(1+|t-r|)^{\frac{3}{2}}}|z\partial_{t,x}f |+\frac{1}{(1+t+r)(1+|t-r|)^{\frac{1}{2}}}|\widehat{\Omega}f |.$$
\item We will then consider energy norms of the form $\| z^{Q-\beta_P} \widehat{Z}^{\beta} f \|_{L^1_{x,v}}$ where $\beta_P$ is the number of homogeneous vector fields composing $\widehat{Z}^{\beta}$.
\end{itemize}
\begin{Rq}\label{rqforintroduction}
Note that
\begin{itemize}
\item a gain of decay such as $(1+|t-r|)^{\frac{1}{2}+\eta}$, $\eta >0$ is sufficient.
\item Certain a priori good error terms become problematic. For instance, in order to control $\| z^{Q} \partial_{t,x} f \|_{L^1_{x,v}} $, we will have to deal with $|\mathcal{L}_{\partial_{t,x}}(F)| z^Q|\widehat{\Omega} f|$. Since we merely control $z^{Q-1}|\widehat{\Omega} f|$, we will have to estimate $z^Q$ by $(1+t+r)z^{Q-1}$. In our case, these terms could be handled but in different settings, they could prevent us to close the energy estimates (see Remark \ref{Rqintronull} for further explanations).
\end{itemize} 
We will then rather propagate the norms $\| z^{Q-(\frac{1}{2}+\eta)\beta_P} \widehat{Z}^{\beta} f \|_{L^1_{x,v}}$, which also have the advantage to require less decay in $x$ on the Vlasov field.
\end{Rq}
\begin{Rq}
The inequality $1+|t-r| \lesssim z$, proved in Lemma \ref{weights1}, does not hold in the interior of the light cone. However, our method can be used in order to deal with non integrable decay rates for solutions to massless Vlasov systems in the whole spacetime since there exists $z_0$ which is preserved by the flow of $|v|\partial_t+v^i \partial_{i}$ and such that $1+|t-r| \lesssim z_0$.
\end{Rq}
\subsubsection{Null structure of the VM system}

As we start with optimal decay in $v$, we cannot fully use \eqref{extradecayintro}. In particular, no extra decay in the $t+r$ direction can be obtained in that way. Moreover, since the initial data are not compactly supported in $v$, a problem arises from large velocities, for which $v^0 \sim |v|$, so that the characteristics of the transport equation ultimately approach those of the Maxwell equations. The consequence is that, in a product such as $\mathcal{L}_{Z}(F) \cdot \partial_{t,x}f$, one cannot, in view of support considerations, transform a $|t-r|$ decay in a $t+r$ one anymore. To circumvent this difficulty, we take advantage of the null structure of the non linearities such as $v^{\mu} {\mathcal{L}_{Z}(F)_{\mu}}^{ i} \partial_{v^i}  f$ by expanding them with respect to a null frame $(L:= \partial_t+\partial_r, \underline{L}:= \partial_t-\partial_r, e_1,e_2)$, where $(e_1,e_2)$ is an orthonormal basis on the $2$-spheres. 

In order to simplify the presentation, we introduce $G$ a smooth solution to the free Maxwell equations, $g$ a smooth solution to the relativistic transport equation $v^{\mu} \partial_{\mu} g=0$ and we study $v^{\mu}{G_{\mu}}^i \partial_{v^i} g$ as a model of the error terms arising in the commuted Vlasov equation \eqref{comutationinto}. 
\begin{enumerate}
\item We know from \cite{dim4} that $|\frac{x^i}{r} \partial_{v^i} g|\lesssim |t-r| |\partial_{t,x} g|+\sum_{\widehat{\Omega}} |\widehat{\Omega} g|$ behaves better than $|\partial_{v^j} g|$. We also observed that the component $v^{\underline{L}}$ allows us to exploit the decay in $t-r$. Indeed, denoting by $C_u$ the null cone $t-|x|=u$, there holds
$$ \sup_{u \in \R} \int_{C_u} \int_{\R^3_v} \frac{v^{\underline{L}}}{v^0} |g|dvdC_u  \lesssim \| g(0 , \cdot , \cdot ) \|_{L^1_{x,v}}, \qquad \text{so that} \qquad \int_{t=0}^{+\infty} \int_{\R^3_x} \int_{\R^3_v}  \frac{v^{\underline{L}}}{v^0} \frac{|g| \hspace{2mm} dv dx dt}{(1+|t-r|)^{1+\eta}}  < +\infty.$$
The good behavior of the angular component $\slashed{v}$ of the velocity vector $v$ is reflected by $|\slashed{v}| \lesssim \sqrt{v^0 v^{\underline{L}}}$. Additional useful properties concerning $v^{\underline{L}}$ and $\slashed{v}$ are proved in Lemma \ref{weights1}.
\item We know from \cite{CK} that if $|G| \lesssim r^{-\frac{5}{2}}$ holds initially, then
\begin{equation}\label{decayCK}
|\alpha_i| := |G_{e_i L}| \lesssim (1+t+r)^{-\frac{5}{2}}, \quad |\rho|+|\sigma| := \frac{1}{2}|G_{L \underline{L}}|+|G_{e_1e_2}| \lesssim (1+t+r)^{-2}(1+|t-r|)^{-\frac{1}{2}},
\end{equation}
whereas $ |\underline{\alpha}_i| := |G_{e_i L}| \lesssim (1+t+r)^{-1}(1+|t-r|)^{-\frac{3}{2}}$.
\end{enumerate} 
Using the notation $R_1 \prec R_2$ if $R_2$ is expected to behave better than $R_2$ and denoting by $\nabla_v g$ the quantity $(0,\partial_{v^1}g, \partial_{v^2} g , \partial_{v^3} g)$, we have
\begin{itemize}
\item $\underline{\alpha} \prec \rho \sim \sigma \prec \alpha$,
\item $v^L \prec \slashed{v} \prec v^{\underline{L}}$,
\item $\left( \nabla_v g \right)^{e_i} \prec \left( \nabla_v g \right)^L=\left( \nabla_v g \right)^{\underline{L}}$.
\end{itemize}
Our non-linearity $v^{\mu}{G_{\mu}}^i \partial_{v^i} g$ is bounded by
\begin{equation}\label{calculGintro}
 \left( |\rho|+\frac{|\slashed{v}|}{v^0} |\underline{\alpha}|  \right)  \left( |t-r| \left| \partial_{t,x} g\right|+\sum_{\widehat{Z} \in \K} \left| \widehat{Z} g \right| \right) + \left(  |\alpha|+\frac{|\slashed{v}|}{v^0}|\sigma|+\frac{v^{\underline{L}}}{v^0} |\underline{\alpha}|  \right)  \left( (t+r) \left| \partial_{t,x} g\right|+\sum_{\widehat{Z} \in \K} \left| \widehat{Z} g \right| \right) .
 \end{equation}
Note now that each term contains either two good factors, $\alpha$, which is the best null component of $G$, or $v^{\underline{L}}$, the best null component of $v$. Another important element, that we have already discussed and which is crucial in the proof of Theorem \ref{theorem}, is that the translations $\partial_{t,x}$ are the only commutation vector fields hit by the weights $t-r$ and $t+r$.

Apart from the last feature mentionned, we had to use the full null structure of the massless VM system in \cite{massless} in order to prove that the norms $\| \widehat{Z}^{\beta} f \|_{L^1_{x,v}}$ are uniformly bounded in time. In \cite{dim3}, where we studied small data solutions to the massive VM system in the whole Minkowski spacetime, the electromagnetic field merely decay as $r^{-2}$ initially, so that the component $\alpha$, $\rho$ and $\sigma$ decay as $(1+t+r)^{-2}$ and we could not exploit all the null properties of the equations. Note however that $\alpha$ still have in this setting a better behavior when $F$ is estimated in $L^2$ since we control
\begin{equation}\label{Cucone}
\sup_{u \in \R}\int_{C_u} (1+t+r) |\alpha|^2+(1+|t-r|)(|\rho|^2+|\sigma|^2)dC_u .
\end{equation}
In this paper, we start with weak decay assumptions on the electromagnetic field, $|F| \lesssim r^{-\frac{3}{2}}$, and the good components $\alpha$, $\rho$ and $\sigma$ merely decay as\footnote{Note however that the null components of  $\mathcal{L}_{\partial_{t,x}} (F)$ satisfy \eqref{decayCK} (c.f. Proposition \ref{extradecayderiv} and Remark \ref{rqforintroo}).} $(1+t+r)^{-\frac{3}{2}}$. Moreover, no weighted $L^2$ bound such as \eqref{Cucone} can be propagated. As a consequence, we will unify the analysis of the terms containg one of the good component $\alpha$, $\rho$ and $\sigma$ as a factor, i.e. we will bound $|t-r| |\rho|$ by $(t+r)|\rho|$ and $\frac{|\slashed{v}|}{v^0}|\sigma|$ by $|\sigma|$.

\begin{Rq}\label{Rqintronull}
We can understand now why it could be necessary to consider energy norms of the form $\| z^{Q-\kappa \beta_P} \widehat{Z}^{\beta} f \|_{L^1_{x,v}}$, with $\kappa <1$. Indeed, as we can see in Remark \ref{rqforintroduction}, we would have to deal with terms of the form $(|\rho|+\frac{|\slashed{v}|}{v^0} |\underline{\alpha}|) (1+t+r)^{\kappa}z^{Q-\kappa (\xi_P-1)} \widehat{Z}^{\xi} f$. However, in view of \eqref{calculGintro}, the bad terms involving $|\rho|$ and $\frac{|\slashed{v}|}{v^0}|\underline{\alpha}|$ carry a $1+|t-r|$ weight. In some sens, taking $\kappa=1$ implies to forget about one part of the null structure of the system, which could be problematic in a different situation or for an other Vlasov system (e.g. when the solution to the wave equation merely decay in the $t+r$ direction as $(1+t+r)^{-1+\delta}$, $\delta >0$).
\end{Rq}

We will also make fundamental use of the fact that the linear transport operator $v^{\mu} \partial_{\mu}$ is of degree $1$ in $v$ whereas the non-linearity $v^{\mu} {F_{\mu}}^i \partial_{v^i} $ is of degree $0$. In the energy inequalities, this is reflected through the schematic estimate
$$ \| \widehat{Z} f (t, \cdot , \cdot ) \|_{L^1_{x_v}} \hspace{1mm} \lesssim \hspace{1mm} \| \widehat{Z} f (0, \cdot , \cdot ) \|_{L^1_{x_v}} + \int_0^t \int_{_x} \int_{\R^3_v} |\mathcal{L}_Z(F)| \frac{|\widehat{Z} f|}{v^0} dv dx ds.$$
For massless particles, $v^0=|v|$ and this leads to additional difficulties (see \cite{dim4} and \cite{massless}). In our case, this allows us to use the inequality $1 \lesssim \sqrt{v^0 v^{\underline{L}}}$ and then to take advantage of the good properties of the component $v^{\underline{L}}$. We point out that without this property, we would not be able to control sufficiently well the error term $(t+r)|\alpha| \lesssim (1+t+r)^{-\frac{1}{2}}$. Indeed, without any of the components $\slashed{v}$ or $v^{\underline{L}}$, one cannot use the extra decay in $t-r$ provided by our hierarchised energy norms.

Finally, in order to prove that the $L^2$ norm of the electromagnetic field and its derivatives is uniformly bounded in time, which is crucial in our proof, we will have use the full null structure of the Maxwell equations. Indeed, 
$$ \int_{|x| \geq t-b} |F(t,x)|^2 dx \hspace{1mm} \lesssim \hspace{1mm} \int_{|x| \geq -b} |F(0,x)|^2 dx + \int_0^t \int_{|x| \geq s-b} \left| F_{\mu 0} \int_{\R^3_v} \frac{v^{\mu}}{v^0} f dv \right| dx ds$$
and with our weak assumptions in $v$ on the Vlasov field, one cannot expect $\| \int_v f dv \|_{L^2_x}$ to decay faster than $(1+t)^{-1}$. In order to avoid a small growth on $\int_{\R^3} |F|dx$, we will then take advantage of the fact that 
$$\left| F_{\mu 0} \int_{\R^3_v} \frac{v^{\mu}}{v^0} f dv \right| \hspace{1mm} \lesssim \hspace{1mm} |\rho ( F)| \int_{\R^3_v} |f| dv +(|\alpha (F)|+|\underline{\alpha} (F) | ) \int_{\R^3_v} \frac{|\slashed{v}|}{v^0} |f| dv $$
so that each term contain at least either a good component of the Maxwell field or $\slashed{v}$.

\subsubsection{Sharp pointwise decay estimates in the exterior of the light cone}\label{subsubsec}

The problem of the estimates \eqref{extradecayintro}-\eqref{extradecayintro2}, which are given by a Klainerman-Sobolev inequality (see Proposition \ref{KS1}), is that they do not provide a sharp pointwise decay estimate on $\int_{\R^3_v} |g| \frac{dv}{(v^0)^2}$ near the light cone, for $g$ a smooth solution to the free relativistic massive transport equation $v^{\mu} \partial_{\mu} (g)=0$. Indeed, if one merely assume that $\sum_{|\beta|+|\kappa| \leq 3} \||v^0|^{|\kappa|}(1+r)^{|\beta|} \partial_{t,x}^{\beta} \partial_v^{\kappa} g \|_{L^1_{x,v}}$ is initially finite, we cannot fully use \eqref{extradecayintro}, which requires the finiteness of a stronger energy norm, so that, for $u \in \R$,
\begin{equation*}
\int_{\R^3_v} |g|(t,t-u,v) \frac{dv}{|v^0|^2}  \hspace{2mm} \lesssim \hspace{2mm} \int_{\R^3_v} |g|(t,t-u,v) dv  \hspace{2mm} \lesssim \hspace{2mm} \frac{1}{(1+t)^2} \sum_{|\beta|+|\kappa| \leq 3} \| (v^0)^{|\kappa|} (1+r)^{|\beta|} \partial^{\beta}_{t,x} \partial_v^{\kappa} g \|_{L^1_{x,v}}(t=0)
\end{equation*}
In \cite{dim4}, we proved that 
\begin{equation}\label{eq:improvementdecayestiintro} \int_{\R^3_v} |g|(t,t-u,v) \frac{dv}{|v^0|^2}  \hspace{2mm} \lesssim \hspace{2mm}  \frac{1}{(1+t)^3} \sum_{|\beta|+|\kappa| \leq n} \| (v^0)^{|\kappa|} (1+r)^{|\beta|} \partial^{\beta}_{t,x} \partial_v^{\kappa} g \|_{L^1_{x,v}}(t=0)
\end{equation}
holds in the interior region, i.e. for all $u \geq 0$, and with $n=3$. In this article, we prove a functional inequality which implies that \eqref{eq:improvementdecayestiintro} also holds in the exterior of the light cone $u < 0$ and with $n=4$.
\begin{Rq}
The improved decay estimates on the velocity average of the Vlasov field given in Theorem \ref{theorem} are obtained by applying this functional inequality. However, we do not need it in order to prove that the solution of the system is global in time.
\end{Rq}

\subsection{Structure of the paper}

Section \ref{sec2} contains most of the notations used in this article. The vector fields used in this paper and the commuted equations are presented in Subsection \ref{subsecvector}. In Subsection \ref{sectionweights}, fundamental properties of the null components of the velocity vector are proved. The energy norms used to study the Vlasov-Maxwell system are introduced in Section \ref{sec4}. During this section, we also prove approximate conservation laws as well as Klainerman-Sobolev type inequalities in order to control these norms and derive pointwise decay estimates from them. Our new decay estimate for velocity averages is proved in Subsection \ref{subsecnewdecay}. In section \ref{sec6}, we set up the bootstrap assumptions, present their immediate consequences and describe the strategy of the proof of our main result. Sections \ref{sec7} (respectively \ref{sec8}) concerns the improvement of the energy bounds on the particle density (respectively the electromagnetic field). Finally, we prove in Section \ref{sec9} $L^2$ estimates for the velocity averages of the higher order derivatives of the Vlasov field.

\subsection{Acknowledgements}

I would like to express my gratitude to my advisor Jacques Smulevici for suggesting me to study this problem and for his valuable comments. Part of this work was funded by the European Research Council under the European Union's Horizon 2020 research and innovation program (project GEOWAKI, grant agreement 714408).

\section{Preliminaries}\label{sec2}

\subsection{Basic notations}

In this article we work on the $3+1$ dimensional Minkowski spacetime $(\R^{3+1},\widetilde{\eta})$ and we will use two sets of coordinates. The Cartesian coordinates $(x^0=t,x^1,x^2,x^3)$ and null coordinates $(\underline{u},u,\omega_1,\omega_2)$, where
$$\underline{u}=t+r, \hspace{10mm} u=t-r$$
and $(\omega_1,\omega_2)$ are spherical variables, which are spherical coordinates on the spheres $(t,r)=constant$. Apart from $r=0$ and the usual degeneration of spherical coordinates, these coordinates are defined globally on the Minkowski space. We will also use the following classical weights,
$$\tau_+:= \sqrt{1+\underline{u}^2} \hspace{10mm} \text{and} \hspace{10mm} \tau_-:= \sqrt{1+u^2}.$$

\begin{Rq}
In this paper, we exclusively work in regions where $ 1+t \leq \tau_+(t,x) \lesssim |x|$.
\end{Rq}
We denote by $\slashed{\nabla}$ the intrinsic covariant differentiation on the spheres $(t,r)=constant$ and by $(e_1,e_2)$ an orthonormal basis on them. Capital Roman indices such as $A$ or $B$ will always correspond to spherical variables. The null derivatives are defined by
$$L=\partial_t+\partial_r \hspace{4mm} \text{and} \hspace{4mm} \underline{L}=\partial_t-\partial_r, \hspace{4mm} \text{so that} \hspace{4mm} L(\underline{u})=2, \hspace{3mm} L(u)=0, \hspace{3mm} \underline{L}( \underline{u})=0 \hspace{3mm} \text{and} \hspace{3mm} \underline{L}(u)=2.$$
The velocity vector $(v^{\mu})_{0 \leq \mu \leq 3}$ is parametrized by $(v^i)_{1 \leq i \leq 3}$ and $v^0=\sqrt{1+|v|^2}$ since we normalize the mass of the particles to $m=1$. Let $\TT$ be the operator defined by
$$\TT : f \mapsto v^{\mu} \partial_{\mu} f,$$
for all sufficiently regular function $f : [0,T[ \times \R^3_x \times \R^3_v$. We will raise and lower indices using the metric $\widetilde{\eta}$. For instance, $v_0 = v^{\mu} \widetilde{\eta}_{\mu 0} = - v^0$ and $x^1 = x_{\mu} \widetilde{\eta}^{ \mu 1}=x_1$. Finally, we will use the notation $Q \lesssim R$ for an inequality of the form $Q \leq C R$, where $C>0$ is a constant independent of the solutions but which could depend on $N$, the maximal number of derivatives, or on fixed parameters ($\delta$ and $\eta$).

\subsection{A null foliation}\label{subsets}

We start by presenting various subsets of the Minkowski space which depends on $t \in \mathbb{R}_+$, $r \in \R_+$, $u \in \mathbb{R}$ or $b \in \R$. Let $\Sp_{t,r}$, $\Si^b_t$, $C_u(t)$ and $V_u(t)$, be the sets defined as
\begin{flalign*}
& \hspace{0.3cm} \Sp_{t,r}:= \{ (s,y) \in \R_+ \times \R^3 \hspace{1mm} / \hspace{1mm} (s,|y|)=(t,r) \}, \hspace{1.8cm} C_u(t):= \{(s,y)  \in \mathbb{R}_+ \times \mathbb{R}^3 / \hspace{1mm} s < t, \hspace{1mm} s-|y|=u \}, & \\
& \hspace{3mm} \Si^b_t := \{ (s,y) \in \R_+ \times \R^3 \hspace{1mm} / \hspace{1mm} s=t, \hspace{1mm} |y| > s-b \}, \hspace{1.5cm} V_u(t) := \{ (s,y) \in \mathbb{R}_+ \times \mathbb{R}^3 / \hspace{1mm} s < t, \hspace{1mm} s-|y| < u \}. &
\end{flalign*}
The volume form on $C_u(t)$ is given by $dC_u(t)=\sqrt{2}^{-1}r^{2}d\underline{u}d \mathbb{S}^{2}$, where $ d \mathbb{S}^{2}$ is the standard metric on the $2$ dimensional unit sphere.
\vspace{5mm}

\begin{tikzpicture}
\draw [-{Straight Barb[angle'=60,scale=3.5]}] (0,-0.3)--(0,5);
\fill[color=gray!35] (2,0)--(5,3)--(9.8,3)--(9.8,0)--(1,0);
\node[align=center,font=\bfseries, yshift=-2em] (title) 
    at (current bounding box.south)
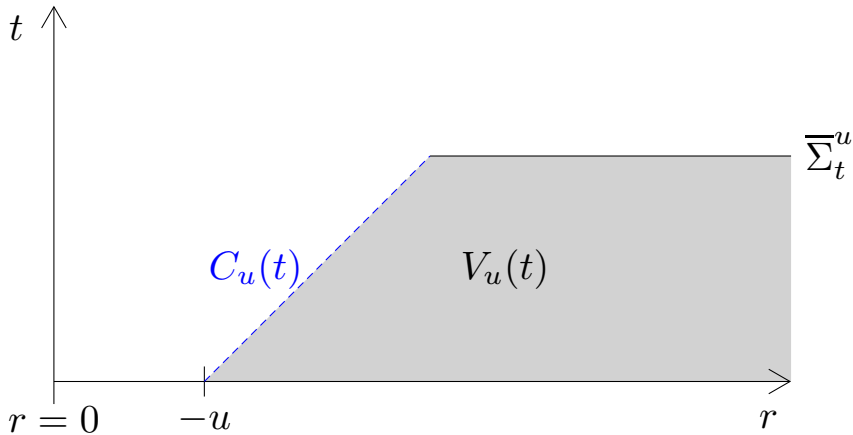
    {The sets $\Si^{u}_t$, $C_u(t)$ and $V_u(t)$};
\draw (5,3)--(9.8,3) node[scale=1.5,right]{$\Si^u_t$};
\draw (2,0.2)--(2,-0.2);
\draw [-{Straight Barb[angle'=60,scale=3.5]}] (0,0)--(9.8,0);
\draw[densely dashed, blue] (2,0)--(5,3) node[scale=1.5,left, midway] {$C_u(t)$};
\draw (6,1.5) node[ color=black!100, scale=1.5] {$V_u(t)$}; 
\draw (0,-0.5) node[scale=1.5]{$r=0$};
\draw (2,-0.5) node[scale=1.5]{$-u$};
\draw (-0.5,4.7) node[scale=1.5]{$t$};
\draw (9.5,-0.5) node[scale=1.5]{$r$};   
\end{tikzpicture}

The following lemma illustrates that we can foliate $V_b(T)$ by $(\Si^b_s)_{0 \leq s < T}$ or $(C_u(T))_{u< b}$ and will be used several times during this article.
\begin{Lem}\label{foliationexpli}
Let $T>0$, $b \in \R$ and $g \in L^1(V_b(T))$. Then
$$ \int_{V_b(T)} g dV_b(T) \hspace{2mm} = \hspace{2mm} \int_0^T \int_{\Si^b_s} g dx ds \hspace{2mm} = \hspace{2mm} \int_{u=-\infty}^b \int_{C_u(T)} g dC_u(T) \frac{du}{\sqrt{2}} .$$
\end{Lem}
We will use the second foliation in order to take advantage of decay in the $t-r$ direction as $\| \tau_-^{-1} \|_{L^{\infty}(C_u(t))} = (1+u^2)^{-\frac{1}{2}}$ whereas $\| \tau_-^{-1} \|_{L^{\infty}(\Si^b_s)} = 1$.

\subsection{The commutation vector fields}\label{subsecvector}

The aim of this subsection is to introduce the commutation vector fields for the Maxwell equations, those for the relativistic transport operator and certain of their basic properties. Let $\Po$ be the generators of the Poincaré algebra, i.e. the set containing
\begin{itemize}
\item the translations \hspace{12mm} $\partial_{\mu}:= \partial_{x^{\mu}}$, \hspace{2mm} $0 \leq \mu \leq 3$.
\item the rotations \hspace{16.5mm} $\Omega_{ij}=x^i\partial_{j}-x^j \partial_i$, \hspace{2mm} $1 \leq i < j \leq 3$.
\item the Lorentz boosts \hspace{7.5mm} $\Omega_{0k}=t\partial_{k}+x^k \partial_t$, \hspace{2mm} $1 \leq k \leq 3$.
\end{itemize}
Let also $\Or := \{ \Omega_{12}, \hspace{1mm} \Omega_{13}, \hspace{1mm} \Omega_{23} \}$ be the set of the rotational vector fields and $\mathbb{K}:= \Po \cup \{ S \}$, where $S=x^{\mu} \partial_{\mu}$ is the scaling vector field. We will use the vector fields of $\mathbb{K}$ for commuting the Maxwell equations. To commute the operator $\TT=v^{\mu} \partial_{\mu}$, we will rather use the complete lifts of the vector fields of $\Po$.

\begin{Def}
Let $V$ be a vector field of the form $V^{\beta} \partial_{\beta}$. Then, the complete lift $\widehat{V}$ of $V$ is defined by
$$\widehat{V}=V^{\beta} \partial_{\beta}+v^{\gamma} \frac{\partial V^i}{\partial x^{\gamma}} \partial_{v^i}.$$
\end{Def}
Consequently, for all $ \mu \in \llbracket 0, 3 \rrbracket$, $1 \leq i < j \leq 3$ and $k \in \llbracket 1, 3 \rrbracket$, $$ \widehat{\partial_{\mu}}= \partial_{\mu},  \hspace{12mm} \widehat{\Omega}_{ij}=x^i \partial_j-x^j \partial_i+v^i \partial_{v^j}-v^j \partial_{v^i} \hspace{12mm} \text{and} \hspace{12mm} \widehat{\Omega}_{0k} = t\partial_k+x^k \partial_t+v^0 \partial_{v^k}.$$
Since $[\TT,\widehat{Z}]=0$ for all $Z \in \Po$ and $[\TT,S]=\TT$, we consider, as \cite{FJS},  the following set
$$\K := \{ \widehat{Z} \hspace{1mm} / \hspace{1mm} Z \in \Po \} \cup \{ S \}.$$
For simplicity, we denote by $\widehat{Z}$ an arbitrary vector field of $\K$, even if $S$ is not a complete lift. Note that the vectorial space generated by each of these sets is an algebra. More precisely, if $\mathbb{L}$ is either $\mathbb{K}$, $\Po$ or $\Or$, then for all $(Z_1,Z_2) \in \mathbb{L}^2$, $[Z_1,Z_2]$ is a linear combination of vector fields of $\mathbb{L}$. We also consider an ordering on each of the sets $\mathbb{O}$, $\mathbb{P}$, $\mathbb{K}$ and $\widehat{\mathbb{P}}_0$, such that, if $\mathbb{P}= \{ Z^i / \hspace{2mm} 1 \leq i \leq |\mathbb{P}| \}$, then $\mathbb{K}= \{ Z^i / \hspace{2mm} 1 \leq i \leq |\mathbb{K}| \}$, with $Z^{|\mathbb{K}|}=S$, and
$$ \K= \left\{ \widehat{Z}^i / \hspace{2mm} 1 \leq i \leq |\K| \right\}, \hspace{2mm} \text{with} \hspace{2mm} \left( \widehat{Z}^i \right)_{ 1 \leq i \leq |\Po|}=\left( \widehat{Z^i} \right)_{ 1 \leq i \leq |\Po|} \hspace{2mm} \text{and} \hspace{2mm} \widehat{Z}^{|\K|}=S  .$$
If $\mathbb{L}$ denotes $\mathbb{O}$, $\mathbb{P}$, $\mathbb{K}$ or $\widehat{\mathbb{P}}_0$, and  $\beta \in \{1, ..., |\mathbb{L}| \}^q$, with $q \in \mathbb{N}^*$, we will denote the differential operator $V^{\beta_1}...V^{\beta_r} \in \mathbb{L}^{|\beta|}$ by $V^{\beta}$. For a vector field $X$, we denote by $\mathcal{L}_X$ the Lie derivative with respect to $X$ and if $Z^{\gamma} \in \mathbb{K}^{q}$, we will write $\mathcal{L}_{Z^{\gamma}}$ for $\mathcal{L}_{Z^{\gamma_1}}...\mathcal{L}_{Z^{\gamma_q}}$. We denote moreover the number of translations composing $V^{\beta}$ by $\beta_T$ and the number of homogeneous vector fields by $\beta_P$, so that $|\beta|= \beta_T+\beta_P$.

Let us recall, by the following classical result, that the derivatives tangential to the cone behave better than others.

\begin{Lem}\label{goodderiv}
The following relations hold,
$$(t-r)\underline{L}=S-\frac{x^i}{r}\Omega_{0i}, \hspace{6mm} (t+r)L=S+\frac{x^i}{r}\Omega_{0i} \hspace{6mm} \text{and} \hspace{6mm} re_A=\sum_{1 \leq i < j \leq 3} C^{i,j}_A \Omega_{ij},$$
where the $C^{i,j}_A$ are uniformly bounded and depend only on spherical variables. Similarly, we have
$$(t-r)\partial_t =\frac{t}{t+r}S-\frac{x^i}{t+r}\Omega_{0i} \hspace{6mm} \text{and} \hspace{6mm} (t-r) \partial_i = \frac{t}{t+r} \Omega_{0i}- \frac{x^i}{t+r}S- \frac{x^j}{t+r} \Omega_{ij}.$$
\end{Lem}
We introduce now the notation $\nabla_v g := (0,\partial_{v^1}g, \partial_{v^2}g,\partial_{v^3}g)$, so that \eqref{VM1} can be rewritten $$\TT_F(f) := v^{\mu} \partial_{\mu} f +F \left( v, \nabla_v f \right) =0.$$
In order to commute the Vlasov-Maxwell system, we recall the following result (see Lemma $2.8$ of \cite{massless} for a proof) where the Kronecker symbol is extended to vector fields, i.e. $\delta_{X, Y}=1$ if $X=Y$ and $\delta_{X, Y}=0$ otherwise.
\begin{Lem}\label{comumax1}
Let $G$ be a $2$-form and $g$ a function both sufficiently regular. Then, for all $\widehat{Z} \in \K$,
$$ \widehat{Z} \left( G \left( v, \nabla_v g \right) \right) = \mathcal{L}_Z(G) \left( v, \nabla_v g \right)+G \left( v, \nabla_v \widehat{Z} g \right)-2 \delta_{\widehat{Z},S} G \left( v, \nabla_v g \right).$$
If $G$ and $g$ satisfy $\nabla^{\mu} G_{\mu \nu} = J(g)_{\nu}$ and $\nabla^{\mu} {}^* \! G_{\mu \nu} = 0$, then
$$\forall \hspace{0.5mm} Z \in \mathbb{K}, \hspace{6mm} \nabla^{\mu} \mathcal{L}_Z(G)_{\mu \nu} = J(\widehat{Z}g)_{\nu}+3 \delta_{Z,S} J(g)_{\nu} \hspace{15mm} \text{and} \hspace{15mm} \nabla^{\mu} {}^* \! \mathcal{L}_Z(G)_{\mu \nu} = 0.$$
\end{Lem}
We then deduce the form of the source terms of the commuted Vlasov-Maxwell equations.
\begin{Pro}\label{Maxcom}
Let $(f,F)$ be a sufficiently regular solution to the VM system \eqref{VM1}-\eqref{VM3} and $Z^{\kappa} \in \mathbb{K}^{|\kappa|}$. There exists integers $n^{\kappa}_{\gamma, \beta}$ and $m^{\kappa}_{\xi}$ such that
\begin{eqnarray}
\nonumber [\TT_F, \widehat{Z}^{\kappa} ](f) \hspace{2mm} = \hspace{2mm} \TT_F \left( \widehat{Z}^{\kappa} f \right) & = & \sum_{\begin{subarray}{} |\gamma|+|\beta| \leq |\kappa| \\ \hspace{1mm} |\beta| \leq |\kappa|-1 \end{subarray}} n^{\kappa}_{\gamma , \beta} \mathcal{L}_{Z^{\gamma}}(F) \left( v , \nabla_v \widehat{Z}^{\beta} (f) \right) , \\ \nonumber
\nabla^{\mu} \mathcal{L}_{Z^{\kappa}}(F)_{\mu \nu} & = & \sum_{|\xi| \leq |\kappa|} m^{\kappa}_{\xi} J \left( \widehat{Z}^{\xi} f \right)_{\nu}, \\ \nonumber
\nabla^{\mu} {}^* \! \mathcal{L}_{Z^{\kappa}}(F)_{\mu \nu} & = & 0 .
\end{eqnarray}
Moreover, the number of homogeneous vector fields $\beta_P$ of $\widehat{Z}^{\beta}$ satisfies the following condition.
\begin{itemize}
\item Either $\beta_P < \kappa_P$
\item or $\beta_P = \kappa_P$ and $\gamma_T \geq 1$.
\end{itemize}
\end{Pro}
Note that the structure of the non-linearity $F(v,\nabla_v f)$ as well as the one of $J(f)$ is preserved by commutation, which reflects the null properties of the system. This is crucial for us since, as mentioned earlier, if the source terms of the Vlasov equation (respectively the Maxwell equations) behaved as $v^0 |F| |\partial_v f|$ (respectively $\int_v |f| dv$), we would not be able to close the energy estimates for the Vlasov field (respectively the electromagnetic field). 
\begin{Rq}\label{hierar}
Let us explain why we count the number of the homogeneous vector fields in the source terms of the Vlasov equation. As $\partial_v f \sim \tau_+ \partial_{t,x} f+\widehat{Z} f$, the decay rate of the solutions will not be strong enough for us to close the energy estimates without using a hierarchy on the derivatives of $f$. If $\gamma_T \geq 1$, Lemma \ref{goodderiv} will give us an extra decay in the $u$ direction. Otherwise, the worst source terms to control in order to bound $\| \widehat{Z}^{\kappa} f \|_{L^1_{x,v}}$ will only involve $\widehat{Z}^{\beta} f$, with $\beta_P < \kappa_P$, making, in some sens, the system constituted by the commuted Vlasov equations triangular.
\end{Rq}

\subsection{The null components of the velocity vector and the weights preserved by $\TT$}\label{sectionweights}

We denote by $(v^L,v^{\underline{L}},v^{e_1},v^{e_2})$ the null components of the velocity vector $v$, so that
$$v=v^L L+ v^{\underline{L}} \underline{L}+v^{e_A}e_A, \hspace{6mm} v^L=\frac{v^0+v^r}{2} \hspace{6mm} \text{and} \hspace{6mm} v^{\underline{L}}=\frac{v^0-v^r}{2}.$$
If there is no ambiguity, we will write $v^{A}$ for $v^{e_A}$. Let $\V$ and $z$ be defined as
$$ \V := \left\{\frac{v^{\mu}}{v^0} \hspace{1mm} \Big/ \hspace{1mm} 0 \leq \mu \leq 3 \right\} \cup \left\{ z_{\mu \nu} \hspace{1mm} \Big/ \hspace{1mm} \mu \neq \nu \right\}, \hspace{4mm} \text{where} \hspace{4mm} z_{\mu \nu} := x^{\mu}\frac{v^{\nu}}{v^0}-x^{\nu}\frac{v^{\mu}}{v^0}, \hspace{0.9cm} \text{and} \hspace{0.9cm} z^2 :=  \sum_{w \in \V} w^2 .$$
Because of regularity issues, we will rather work with $z$ than with the elements of $\V$. We point out that, since $1 \in \V$, $\frac{|v^{\mu}|}{v^0} \leq 1$ and $|z_{\mu\nu}| \leq 2(t+r)$, we have
\begin{equation}\label{boundz}
1 \hspace{1mm} \lesssim \hspace{1mm} z \hspace{1mm} \lesssim \hspace{1mm} \tau_+.
\end{equation} 
Two fundamental properties of these weights is that they are preserved by the flow of $\mathbf{T}$ and by the action of $\K$.
\begin{Lem}\label{weights}
For all $\widehat{Z} \in \K$ and $a \in \R$, we have
$$ \mathbf{T}(z)=0 \hspace{10mm} \text{and} \hspace{10mm} \left| \widehat{Z} (z^a) \right| \hspace{0.5mm} \lesssim \hspace{0.5mm} |a|z^a.$$
\end{Lem}
\begin{proof}
Let $w \in \V$. By straightforward computations, one can prove that
$$ \mathbf{T}(w)=0 \hspace{1cm} \text{and} \hspace{1cm} \widehat{Z}(v^0 w) \in v^0 \V \cup \{0 \}, \hspace{1cm} \text{so that} \hspace{1cm} \left| \widehat{Z} (w) \right| \lesssim \sum_{w_0 \in \V} |w_0|.$$ Indeed, considering for instance $tv^1-x^1v^0$, $x^1v^2-x^2v^1$, $\widehat{\Omega}_{12}$ and $S$, we have
\begin{eqnarray}
\nonumber \widehat{\Omega}_{12}(tv^1-x^1v^0) & = & -tv^2+x^2v^0, \hspace{24mm} \widehat{\Omega}_{12}(x^1v^2-x^2v^1 ) \hspace{2mm} = \hspace{2mm} 0,\\ \nonumber
S(tv^1-x^1v^0) & = & tv^1-x^1v^0 \hspace{12.9mm} \text{and} \hspace{12.9mm} S(x^1v^2-x^2v^1 ) \hspace{2mm} = \hspace{2mm} x^1v^2-x^2v^1.
\end{eqnarray}
Then, 
$$\mathbf{T}(z)= \sum_{w \in \V} \frac{w}{z} \mathbf{T}(w)=0 \hspace{1cm} \text{and} \hspace{1cm} \left|\widehat{Z}(z^a)\right| = \left| az^{a-1}\sum_{w \in \V} \frac{w}{z} \widehat{Z}(w) \right| \lesssim |a|z^{a-1} \sum_{w_0 \in \V} |w_0|\lesssim |a| z^a .$$
\end{proof}
Recall that if $\mathbf{k}_0 := \mathbf{k}_1 \cup \{ x^{\mu} v_{\mu} \}$, then $\tau_- v^L+\tau_+v^{\underline{L}} \lesssim \sum_{w \in \mathbf{k}_0} |w|$. Unfortunately, the weight $x^{\mu} v_{\mu}$ is not preserved by\footnote{Note however that $x^{\mu} v_{\mu}$ is preserved by the massless relativistic transport operator $|v| \partial_t+v^i \partial_i$.} $\mathbf{T}$ so we will not be able to take advantage of this inequality during this paper. In the following lemma, which reflects the good behavior of the components $v^{\underline{L}}$ and $v^A$ of the velocity vector, we prove a similar inequality specific to the exterior of the light cone and adapted to the study of massive particles.
\begin{Lem}\label{weights1}
We have, for all $|x| \geq t$,
$$ 1 \leq 4v^0v^{\underline{L}}, \hspace{8mm} |v^A| \lesssim \sqrt{v^0v^{\underline{L}}} \hspace{8mm} \text{and} \hspace{8mm} \tau_-+(1+r)\frac{v^{\underline{L}}}{v^0}+(1+r)\frac{|v^A|}{v^0} \lesssim z .$$
This implies in particular $|v^A| \lesssim v^0 v^{\underline{L}}$.
\end{Lem} 
\begin{proof}
Note first that $4r^2v^Lv^{\underline{L}} \geq r^2 +\sum_{k < l} |v^0z_{kl}|^2 $. Indeed, as we study massive particles, we have $v^0 =\sqrt{1+ |v|^2}$, so that
\begin{eqnarray}
\nonumber 4r^2v^Lv^{\underline{L}} & = & (rv^0)^2-\left(x^i v_i \right)^2 \hspace{2mm} = \hspace{2mm} r^2+ \sum_{i=1}^3(r^2 -|x^i|^2)|v_i|^2-2\sum_{1 \leq k < l \leq 3} x^kx^lv_kv_l, \\ \nonumber 
 \sum_{1 \leq k < l \leq 3} |v^0 z_{kl}|^2 & = & \sum_{1 \leq k < l \leq 3} |x^k|^2 |v_l|^2+|x^l|^2 |v_k|^2-2x^kx^lv_kv_l \hspace{2mm} = \hspace{2mm} \sum_{i=1}^3 \sum_{j \neq i} |x^j|^2 |v_i|^2 -2\sum_{1 \leq k < l \leq n} x^kx^lv_kv_l.
\end{eqnarray}
The first inequality then comes from $v^L \leq v^0$. The second one and $(1+r)\frac{|v^A|}{v^0} \lesssim z$ then ensue from $rv_A = v^0 C_A^{i,j} z_{ij}$, where $C^{i,j}_A$ are bounded functions depending only on the spherical variables such as $re_A= C^{i,j}_A \Omega_{ij}$. The last part of the third inequality is specific to the exterior of the light cone. Recall that $x^i-t\frac{v^i}{v^0}  \in \V$. Then, $\tau_- \lesssim z$ follows from $1 \leq z$ (see \eqref{boundz}) and
\begin{align}
(r-t) \leq r-t\frac{|v|}{v^0} \leq \left| x-t\frac{v}{v^0} \right| \leq \sum_{i=1}^3 \left| x^i-t\frac{v^i}{v^0} \right|=\sum_{i=1}^3 |z_{0i}| & \leq z. \label{tauminus}
\end{align}
Finally, remark first that $v^{\underline{L}} \leq v^0$, which treats the case $|x| \leq 1$. If $|x| \geq \max(t,1)$, note that
$$2r v^{\underline{L}} = r v^0-r\frac{x^i}{r}v_i=rv^0+(t-r)\frac{x^i}{r}v_i-\frac{x^i}{r}v^0 \left( t \frac{v_i}{v^0}-x_i \right)-rv^0 =(t-r)v_i\frac{x^i}{r}-v^0 \frac{x^i}{r}z_{0i} \hspace{5mm} \text{and use \eqref{tauminus}}.$$
\end{proof}
\begin{Rq}\label{hierar2}
In the cases where no extra decay can be obtained from the electromagnetic field, i.e. when $\gamma_T =0$ (see Remark \ref{hierar}), the hierarchies in the commuted Vlasov equation and a well-choosen energy norm will allow us to gain decay from Lemma \ref{weights1}.
\end{Rq}
\subsection{The null decomposition of the electromagnetic field}\label{basicelec}

In order to capture its geometric properties, the electromagnetic field will be represented all along this paper by a $2$-form. Let $G$ be a $2$-form defined on $[0,T[ \times \R^3_x$. Its Hodge dual ${}^* \! G$ is the $2$-form given by
$${}^* \! G_{\mu \nu} = \frac{1}{2} G^{\lambda \sigma} \varepsilon_{ \lambda \sigma \mu \nu},$$
where $\varepsilon$ is the Levi-Civita symbol, and its energy-momentum tensor is
$$T[G]_{\mu \nu} :=   G_{\mu \beta} {G_{\nu}}^{\beta}- \frac{1}{4}\eta_{\mu \nu} G_{\rho \sigma} G^{\rho \sigma}.$$
Note that $T[G]_{\mu \nu}$ is symmetric, i.e. $T[G]_{\mu \nu}=T[G]_{\nu \mu}$. The null decomposition of $G$, $(\alpha(G), \underline{\alpha}(G), \rho(G), \sigma (G))$, introduced by \cite{CK}, is defined by
$$\alpha_A(G) = G_{AL}, \hspace{5mm} \underline{\alpha}_A(G)= G_{A \underline{L}}, \hspace{5mm} \rho(G)= \frac{1}{2} G_{L \underline{L} } \hspace{5mm} \text{and} \hspace{5mm} \sigma(G) =G_{e_1 e_2},$$
so that the null components of $T[G]$ are then given by
\begin{equation}\label{tensorcompo}
T[G]_{L L}=|\alpha(G)|^2, \hspace{6mm} T[G]_{\underline{L} \hspace{0.5mm} \underline{L}}=|\underline{\alpha}(G)|^2 \hspace{6mm} \text{and} \hspace{6mm} T[G]_{L \underline{L}}=|\rho(G)|^2+|\sigma(G)|^2.
\end{equation}
For a proof of the following classical results, we refer to Section $3$ of \cite{CK} or to Subsection $2.3$ of \cite{massless}.
\begin{Lem}\label{maxwellbis}
Let $G$ be a $2$-form and $J$ be a $1$-form both sufficiently regular and such that
\begin{eqnarray}
\nonumber \nabla^{\mu} G_{\mu \nu} & = & J_{\nu} \\ \nonumber
\nabla^{\mu} {}^* \! G_{\mu \nu} & = & 0.
\end{eqnarray}
Then, $\nabla^{\mu} T[G]_{\mu \nu}=G_{\nu \lambda} J^{\lambda}$ and, denoting by $(\alpha, \underline{\alpha}, \rho, \sigma)$ the null decomposition of $G$,
\begin{eqnarray}
\nabla_{\underline{L}} \rho - \frac{2}{r} \rho + \slashed{\nabla}^A (\underline{\alpha})_A & = & J_{\underline{L}}, \label{nullrho} \\
\nabla_{\underline{L}} \sigma - \frac{2}{r} \sigma + \varepsilon^{AB} \slashed{\nabla}_A (\underline{\alpha})_B & = & 0, \label{nullsigma} \\
\nabla_{\underline{L}} (\alpha)_A-\nabla_L ( \underline{\alpha})_A +2\slashed{\nabla}_{e_A} \rho - \frac{1}{r}\alpha_A-\frac{1}{r} \underline{\alpha}_A & = & 0 . \label{nullalpha}
\end{eqnarray}
\end{Lem}
\begin{proof}
As the equation \eqref{nullalpha} is less classical, we give its proof here. Recall that $\nabla^{\mu} {}^* \! G_{\mu \nu} = 0$ implies the tensorial equation $\nabla_{[\lambda} G_{\mu \nu ] }=0$ (see for instance Subsection $2.3$ of \cite{massless} for a proof). Taking the $(L,\underline{L},A)$ component of this equation, we get, as $\nabla_L \underline{L}=\nabla_{\underline{L}} L=0$ and $\nabla_{e_A} L=-\nabla_{e_A} \underline{L} = \frac{1}{r}e_A$,
\begin{eqnarray}
\nonumber \nabla_{[L} G_{\underline{L} A ] } =0 & \Leftrightarrow & \nabla_L (G)(\underline{L} , e_A)+\nabla_{\underline{L}} (G)(e_A,L)+\nabla_{e_A} (G)(L,\underline{L}) =0 \\ \nonumber
& \Leftrightarrow & -\nabla_L (\underline{\alpha})( e_A)+\nabla_{\underline{L}} (\alpha)(e_A)+2 \nabla_{e_A} \rho-G( \nabla_{e_A} L, \underline{L} )-G(L, \nabla_{e_A} \underline{L} ) =0 \\ \nonumber
& \Leftrightarrow & -\nabla_L (\underline{\alpha})( e_A)+\nabla_{\underline{L}} (\alpha)(e_A)+2 \slashed{\nabla}_{e_A} \rho-\frac{1}{r}\underline{\alpha}( e_A )-\frac{1}{r} \alpha (e_A) =0,
\end{eqnarray}
which implies \eqref{nullalpha}.
\end{proof}

\section{Energy and pointwise decay estimates}\label{sec4}

We recall here classical energy estimates for both the Vlasov field and the electromagnetic field and how obtain pointwise decay estimates from them. For all this section, we define $T>0$ and $b \leq -1$. The energies defined below are adapted to the study of the Vlasov-Maxwell system in the exterior of the light cone $u \geq b$.

\subsection{Estimates for velocity averages}\label{energy}

For the Vlasov field, we will use the following approximate conservation law.
\begin{Pro}\label{energyf}
Let $H : V_b(T) \times \R^3_v \rightarrow \R $ and $g_0 : \Si^b_0 \times \R^3_v \rightarrow \R$ be two sufficiently regular functions and $F$ a sufficiently regular $2$-form. Then, $g$, the unique classical solution of 
\begin{eqnarray}
\nonumber \TT_F(g)&=&H \\ \nonumber
g(0,.,.)&=&g_0,
\end{eqnarray}
satisfies, for all $t \in [0,T[$, the following estimates,
$$ \left\| \int_{v \in \R^3} |g|dv \right\|_{L^1(\overline{\Sigma}^b_t)}+ \sup_{u < b} \left\| \int_{v \in \R^3} \frac{v^{\underline{L}}}{v^0} |g| dv \right\|_{L^1(C_u(t))} \hspace{1mm} \leq \hspace{1mm}   2 \left\| \int_{v \in \R^3} |g_0| dv \right\|_{L^1(\Si^b_0)} + 2\int_0^t \int_{\Si^b_s}  \int_{v \in \R^3}  |H|\frac{dv}{v^0} dxds.$$
\end{Pro}

\begin{proof}
As $\TT(|g|)=\frac{g}{|g|}H-\frac{g}{|g|} F ( v, \nabla_v g)$ and since $F$ is a $2$-form, integration by parts in $v$ gives us
\begin{equation}\label{div}
 \partial_{\mu} \int_v |g| \frac{v^{\mu}}{v^0}dv = \int_v \left( \frac{g}{|g|}\frac{H}{v^0}-\frac{g}{|g|} F\left( \frac{v}{v^0}, \nabla_v g \right) \right) dv = \int_v \left( \frac{g}{|g|}\frac{H}{v^0}- \frac{v^{j}v^i}{(v^0)^3}F_{j i}|g| \right) dv= \int_v  \frac{g}{|g|}H\frac{dv}{v^0} .
 \end{equation}
Apply now the divergence theorem to $\int_v |g| \frac{v^{\mu}}{v^0}dv$ in the region $V_u(t)$, for $u < b$, in order to get
$$ \int_{\Si^u_t} \int_v  |g| dvdx+\sqrt{2} \int_{C_u(t)} \int_v \frac{v^{\underline{L}}}{v^0} |g| dv dC_u(t)  \hspace{2mm} = \hspace{2mm}  \int_{\Si^u_0} \int_v  |g| dvdx+\int_0^t \int_{\Si^u_s}  \int_v  \frac{g}{|g|}H \frac{dv}{v^0} dx ds .$$
We then deduce that
 \begin{eqnarray}
\nonumber \int_{\Si^b_t}  \int_v  |g| dv dx \hspace{2mm} = \hspace{2mm} \sup_{u < b} \int_{\Si^{u}_t}  \int_v  |g| dv dx  & \leq & \int_{\Si^{b}_0}  \int_v  |g| dv dx+\int_0^t \int_{\Si^{b}_s }   \int_v  |H| \frac{dv}{v^0} dx ds, \\
\nonumber \sup_{u < b} \int_{C_{u}(t)} \int_{v } \frac{v^{\underline{L}}}{v^0} |g| dv dC_{u}(t) & \leq & \int_{\Si^{b}_0}  \int_{v } |g| dv dx+\int_0^t \int_{\Si^{b}_s }   \int_v  |H| \frac{dv}{v^0} dx ds, 
 \end{eqnarray}
which allows us to conclude the proof.
\end{proof}

In view of Remarks \ref{hierar}-\ref{hierar2} and the previous proposition, we then define hierarchised energy norms. For $(Q,\lambda) \in \mathbb{N}  \times [0,\frac{1}{2}]$ and $q \in [Q,+\infty[$, let
\begin{eqnarray}\label{defE}
 \E_b[g](t) & := &\left\| \int_{v \in \R^3} |g|dv \right\|_{L^1(\overline{\Sigma}^b_t)}+ \sup_{u < b} \left\| \int_{v \in \R^3} \frac{v^{\underline{L}}}{v^0} |g| dv \right\|_{L^1(C_u(t))}, \\ \label{defE2}
  \E^{q,\eta}_{Q,b}[g](t) & := & \sum_{0 \leq j \leq Q} \hspace{0.5mm} \sum_{ \widehat{Z}^{\beta} \in \K^j } \hspace{0.5mm}   \E_b[ \sqrt{z}^{ \hspace{0.5mm} q-(1+2\eta) \beta_P} \widehat{Z}^{\beta} g ](t).
\end{eqnarray}

\begin{Rq}\label{Rq32}
As $z \geq 1$ by \eqref{boundz}, we have $\E[\sqrt{z}^a \widehat{Z}^{\beta} g ](t) \leq  \E^{q,\eta}_{Q,b}[g](t)$ for all $0 \leq a \leq q-(1+2\eta) \beta_P$.
\end{Rq}

The remainder of this subsection is devoted to the proof of a Klainerman-Sobolev type inequality. The constants hidden by $\lesssim$ will here depend on $|a|$. We start with a commutation property between the vector fields of $\mathbb{K}$ and the averaging in $v$.
\begin{Lem}\label{lift}
Let $g : V_b(T) \times \R^3_v \rightarrow \mathbb{R} $ be a sufficiently regular function and $a \in \R$. We have, almost everywhere,
$$\forall \hspace{0.5mm} Z \in \mathbb{K}, \hspace{10mm} \left|Z\left( \int_{v \in \R^3 } z^a|g| dv \right) \right| \hspace{1.5mm} \lesssim \hspace{1.5mm} \int_{v \in \R^3 } z^a |g | dv+\int_{v \in \R^3 } z^a | \widehat{Z} g | dv .$$
\end{Lem}
\begin{proof}
Consider for instance the case where $Z=\Omega_{01}=t \partial_1+x^1 \partial_t$. We have, almost everywhere,
\begin{eqnarray}
\nonumber \left| Z\left( \int_{v \in \R^3 } |z^ag| dv \right) \right| & = & \left| \int_{v \in \R^3 } \widehat{\Omega}_{01} \left( |z^ag|  \right)  dv -\int_{v \in \R^3 } v^0 \partial_{v^1} \left( |z^ag|  \right)  dv \right| \\ \nonumber
& = & \left| \int_{v \in \R^3 } \frac{z^a g}{|z^a g|} \widehat{\Omega}_{01} \left( z^ag \right)  dv +\int_{v \in \R^3 } \frac{v^1}{v^0} |z^ag|   dv \right| \\ \nonumber 
& \leq & \int_{v \in \R^3 } \left| \widehat{\Omega}_{01} \left( z^a g \right) \right| dv+\int_v \left|  z^a g \right| dv.
\end{eqnarray}
It then remains to use $|\widehat{\Omega}_{01} \left( z^a \right) | \lesssim |a| z^a$.
\end{proof}
Before presenting the Klainerman-Sobolev inequality used in this article, we prove the following estimate.
\begin{Lem}\label{Sobshere}
Let $h : \mathbb{S}^2 \times \R^3_v \rightarrow \R$ be a sufficiently regular function and $a \in \R$. Then,
$$\forall \hspace{0.5mm} \omega \in \mathbb{S}^2, \hspace{8mm} \int_{v \in \R^3} z^a |h|(\omega,v) dv \hspace{2mm} \lesssim \hspace{2mm} \sum_{0 \leq j \leq 2} \hspace{0.5mm} \sum_{ \Omega^{\beta} \in \Or^j } \hspace{0.5mm} \left\| \int_{v \in \R^3}  z^a \left| \widehat{\Omega}^{\beta} h \right| dv \right\|_{L^1(\mathbb{S}^2)} .$$
\end{Lem}
\begin{proof}
Let $\omega \in \mathbb{S}^2$ and $(\theta, \varphi )$ a local coordinate map in a neighborhood of $w$. By the symmetry of the sphere we can suppose that $\theta$ and $\varphi$ take their values in an interval of a size independent of $\omega$. Using a one dimensional Sobolev inequality, that $|\partial_{\theta} u| \lesssim \sum_{\Omega \in \Or} |\Omega u |$ and Lemma \ref{lift}, we have,
\begin{eqnarray}
\nonumber \int_v z^a |h|(\omega_1,\omega_2,v) dv & \lesssim & \int_{\theta} \left| \int_v z^a |h|(\theta,\omega_2,v) dv \right|+\left| \partial_{\theta} \int_v z^a |h|(\theta,\omega_2,v) dv \right| d \theta \\ \nonumber
& \lesssim & \int_{\theta} \int_v z^a | h|(\theta,\omega_2,v) dv  d \theta+ \sum_{  \Omega \in \Or } \int_{\theta} \int_v z^a | \widehat{\Omega} h|(\theta,\omega_2,v) dv  d \theta.
\end{eqnarray}
Doing the same for the second spherical variable of $\int_v z^a | \widehat{\Omega}^{\beta} h|(\theta,\omega_2,v) dv$, we obtain the result.
\end{proof}
\begin{Pro}\label{KS1}
Let $g : V_b(T) \times \R^3_v \rightarrow \R$ be a sufficiently regular function and $a \in \R$. Then,
$$\forall \hspace{0.5mm} (t,x) \in V_b(T), \hspace{1cm} \int_{v \in \R^3} z^a |g|(t,x,v) dv \hspace{2mm} \lesssim \hspace{2mm} \frac{1}{\tau_+^2 \tau_-} \sum_{0 \leq j \leq 3} \hspace{0.5mm} \sum_{ \widehat{Z}^{\beta} \in \K^j } \hspace{0.5mm} \int_{|y| \geq |x| } \int_{v \in \R^3} z^a \left| \widehat{Z}^{\beta} g \right| dv dx.$$
\end{Pro}
\begin{proof}
Let $(t,x)=(t,|x| \omega) \in V_b(T)$. One has, using successively Lemmas \ref{goodderiv} and \ref{lift},
\begin{align}
\nonumber |x|^2 \tau_- \int_v z^a |g|(t,|x| \omega,v) dv \hspace{1mm} & = \hspace{1mm}  -|x|^2 \int_{r=|x|}^{+ \infty} \partial_r \left( \tau_- \int_v z^a |g|(t,r \omega,v) dv \right) dr\\ \nonumber
& \lesssim \hspace{1mm} |x|^2 \sum_{Z \in \mathbb{K}} \int_{r=|x|}^{+ \infty} \left( \int_v z^a |g|(t,r \omega,v) dv + Z \left( \int_v z^a |g|(t,r \omega ,v) dv \right) \right) dr \\ 
& \lesssim \hspace{1mm} \int_{r=|x|}^{+ \infty}  \int_v z^a |g|(t,r \omega ,v) dv r^2 dr+ \sum_{ \widehat{Z} \in \K} \int_{r=|x|}^{+ \infty}  \int_v z^a |\widehat{Z}^{\kappa} g|(t,r \omega ,v) dv r^2 dr. \label{eq:tmoinsr}
\end{align}
It then remains to apply the previous lemma and to recall that $\tau_+ \lesssim r$ in $V_b(T)$.
\end{proof}
We can improve the decay rate in the $\underline{u}$ direction if we pay the price in terms of weights in $v^0$ and $z$. The decay in $u$ can be improved without any loss in $v^0$. More precisely, by Lemma \ref{weights1}, we have $|v^0|^{-2} \lesssim \frac{v^{\underline{L}}}{v^0} \lesssim \frac{ z}{1+r}$ and $1 \lesssim \frac{z}{\tau_-}$, so that, for any $(k,q) \in \R_+^2$,
\begin{eqnarray}
 \forall \hspace{0.5mm} (t,x) \in V_b(T), \hspace{1cm} \int_{v \in \R^3} |g|(t,x,v) \frac{dv}{(v^0)^{2k}} & \lesssim &  \frac{1}{\tau_+^k\tau_-^q}\int_{v \in \R^3} z^{k+q} |g|(t,x,v) dv \label{decaytr33} \\ & \lesssim & \frac{1}{\tau_+^{2+k} \tau_-^{q+1}} \sum_{0 \leq j \leq 3} \hspace{0.5mm}  \sum_{ \widehat{Z}^{\beta} \in \K^j } \hspace{0.5mm} \left\| \int_{v \in \R^3} z^{k+q} \left|  \widehat{Z}^{\beta} g \right| dv \right\|_{L^1(\Si^b_t)}. \label{improvementdecay}
\end{eqnarray}
Note that one can deduce the estimate \eqref{extradecayintro} from this inequality and Proposition \ref{energyf}.

\subsection{Improved decay estimates for velocity averages in the exterior of the light cone}\label{subsecnewdecay}

As mentionned during the introduction, the problem of the Klainerman-Sobolev inequality of Proposition \ref{KS1}, applied with $a=0$, is that it does not provide a sharp pointwise decay estimate on $\int_{\R^3_v} |g| \frac{dv}{(v^0)^2}$ near the light cone, for $g$ a smooth solution to the free relativistic massive transport equation $v^{\mu} \partial_{\mu} (g)=0$. The aim of this section is to prove the following functional inequalities.

\begin{Pro}\label{Propdecayopti}
Let $u \leq 0$ and $f : V_u(T) \times \R^3_v \rightarrow \R$ be a sufficiently regular function such that
$$ \sum_{|\beta| \leq 3} \int_{\overline{\Sigma}^u_0} \int_{\R^3_v} \left| \widehat{Z}^{\beta} f \right|  dv dx \hspace{2mm} < \hspace{2mm} + \infty.$$
Then, for all $(t,x) \in V_u(T)$,
\begin{equation}\label{kevatalenn:opti}
 \int_{\R^3_v} |f|(t,x,v) \frac{dv}{(v^0)^2} \hspace{1mm} \lesssim \hspace{1mm} \frac{1}{(1+t+r)^3}\sum_{|\beta| \leq 4} \left(  \int_{\overline{\Sigma}^u_0} \int_{\R^3_v} \left| \widehat{Z}^{\beta} f \right| dv dx+ \int_0^t\int_{\overline{\Sigma}^u_s} \int_{\R^3_v} \left| \TT_F ( \widehat{Z}^{\beta} f ) \right| \frac{dv}{v^0} dx ds \right) \hspace{-1mm}.
 \end{equation}
Moreover, if $f$ is also defined and regular on $[0,T[ \times \R^3_x \times \R^3_v$, we have, for all $(t,x) \in [0,T[ \times \R^3_x$,
$$ \int_{\R^3_v} |f|(t,x,v)  \frac{dv}{(v^0)^2} \hspace{0.5mm} \lesssim \hspace{0.5mm} \frac{1}{(1+t+r)^3}\sum_{|\beta| \leq 4} \left(  \int_{\R^3_x} \int_{\R^3_v} \left| \widehat{Z}^{\beta} f \right| (0,x,v) dv dx+ \int_0^t\int_{\{s \} \times \R^3_x} \int_{\R^3_v} \left| \TT_F ( \widehat{Z}^{\beta} f ) \right| \frac{dv}{v^0} dx ds \right) \hspace{-1.1mm}.$$
\end{Pro}
This improves in particular Theorem $5.1$ of \cite{dim4}, where $\int_{\R^3_v} |f|(t,x,v)  \frac{dv}{(v^0)^2}$ is estimated by norms with an additional weight $z$. 
\begin{Rq}
Our estimate requires more regularity on $f$ than a standard $L^{\infty}-L^1$ Sobolev-type inequality. To understand why, we refer to Remark \ref{diffi} below. Note that Georgiev proved in \cite{Georgiev} a similar estimate for the solutions of the Klein-Gordon equation $-\partial^2_t \phi+\Delta \phi =G$ which also requires to loose one derivative. Indeed, the solution $\phi$ is controled pointwise by certain $L^2$ norms of derivatives of $G$ up to third order.
\end{Rq}
We point out that, for the interior region $|x| \leq t$, the second estimate of the proposition is given by the proof of Theorem $5.1$ (the weight $z$ was only used in order to deal with the exterior of the light cone $|x| > t$). The remainder of this subsection is then devoted to the proof of \eqref{kevatalenn:opti}. In order to lighten the notation, we do it for $u=0$ but one can check that the proof could be adapted, without any additional argument, to the cases $u<0$. The strategy consists in splitting $V_0(T)$ into three parts.
\begin{enumerate}
\item The bounded subset $t < |x| \leq 2$. Here, we use a standard Sobolev inequality for velocity averages but since we want the domain of integration to be included in $V_0(T)$, we need to be careful. For this, we denote by $\varepsilon^i$ the sign of $x^i$, for any $i \in \llbracket 1 , 3 \rrbracket$ and we remark that $$\{t \} \times \prod_{1 \leq i \leq 3} ] \min ( \varepsilon^i \infty , x^i), \min (x^i, \varepsilon^i \infty) [ \hspace{2mm} \subset \hspace{2mm} V_0(T).$$ 
Hence, applying successively three times the fundamental theorem of calculus and the inequality $|\partial_i (|h|)| \leq |\partial_i h| $, which holds for all sufficiently regular function $h$, we have
\begin{align*}
\int_{\R^3_v} |f| (t,x,v) dv  & =  -\int_{y^1=x^1}^{\varepsilon^1 \infty} \partial_{y_1} \int_{\R^3_v} |f| (t,y^1,x^2,x^3,v) dv d y^1 \leq \left| \int_{y^1=x^1}^{\varepsilon^1 \infty} \int_{\R^3_v} |\partial_1 f| (t,y^1,x^2,x^3,v) dv d y^1 \right| \\ & \leq  \left| \int_{y^2=x^2}^{\varepsilon^2 \infty}\int_{y^1=x^1}^{\varepsilon^1 \infty} \int_{\R^3_v} |\partial_2\partial_1 f| (t,y^1,y^2,x^3,v) dv d y^1 d y^2 \right| \\ & \leq \left| \int_{y^3=x^3}^{\varepsilon^3 \infty} \int_{y^2=x^2}^{\varepsilon^2 \infty}\int_{y^1=x^1}^{\varepsilon^1 \infty} \int_{\R^3_v} |\partial_3\partial_2\partial_1 f| (t,y,v) dv d y^1 d y^2 d y^3 \right| \leq \int_{\overline{\Sigma}^0_t} \int_{\R^3_v} |\partial_3\partial_2\partial_1 f| dv dx.
\end{align*}
It then remains to apply the energy inequality of Proposition \ref{energyf} to $\partial_3\partial_2\partial_1 f$.
\item The region $|x| \geq \max (2, 2t)$, on which $\tau_+ \leq 10 \tau_-$, so that the result is implied by the Klainerman-Sobolev inequality of Proposition \ref{KS1} and Proposition \ref{energyf}.
\item The subset of $V_0(T)$ where $\min(2,t) < |x| \leq 2t$, on which a further analysis is required. The decay rate $|x|^{-3}$ will be obtained through a Sobolev inequality on $\mathcal{H}_a(t)$, defined in \eqref{boundary} below and which is a piece of the hyperboloid $r^2-s^2=a^2$. Hence, the second step consists in controlling the $L^1$ norm of $\int_{\R^3_v} |\widehat{Z}^{\beta} f | dv$ on $\mathcal{H}_a(t)$. To that end, we apply the divergence theorem to $\int_{\R^3_v} |\widehat{Z}^{\beta} f | \frac{v^{\mu}}{v^0}dv$ in a well-choosen domain $\mathcal{D}_a(t) \subset V_0(T)$. More precisely, with $a:=\sqrt{|x|^2-t^2}$ and $q :=\frac{1}{2}t- \sqrt{a^2+\frac{1}{4}t^2}$, we define
\begin{equation}\label{Da}
\mathcal{D}_a(t) := \left\{ (s,x) \in V_0(T) \hspace{1mm} / \hspace{1mm} s \leq \frac{t}{2}, \hspace{2mm} s-|x| \leq q \right\} \cup \left\{ (s,x) \in V_0(T) \hspace{1mm} / \hspace{1mm} \frac{t}{2} \leq s \leq t, \hspace{2mm} |x|^2-s^2 \geq a^2 \right\} 
\end{equation}
and we observe that the boundary of $\mathcal{D}_a(t)$ is constituted by
\begin{equation}\label{boundary}
\overline{\Sigma}^q_0, \quad \overline{\Sigma}^{t-|x|}_t, \quad C_q \left( \frac{t}{2} \right), \quad \mathcal{H}_a(t):= \left\{ (s,x) \in V_0(T) \hspace{1mm} / \hspace{1mm} \frac{t}{2} \leq s \leq t, \hspace{2mm} |x|^2-s^2 = a^2 \right\}.
\end{equation}
\begin{tikzpicture}

\fill [color=gray!35] 
(0,0)
--(1.605,0)--({sqrt(13)},2)
-- plot [domain=4:5] (\x, {sqrt(\x*\x-9)})
--(9,4)--(9,0)
--cycle;
\draw[color=black!100] (6,2) node[scale=1.5]{$\mathcal{D}_a(t)$};
\draw[color=black!100] (1.5,1) node[scale=1.3]{$C_q(\frac{t}{2})$};
\draw[color=black!100] (3.4,3.1) node[scale=1.3]{$\mathcal{H}_a(t)$};
\draw[color=black!100] (10,4) node[scale=1.5]{$\overline{\Sigma}^{t-|x|}_t$};
\draw (0,4) node[scale=1.5,left]{$t$};
\draw (3.3,-0.5) node[scale=1.5,left]{$a$};
\draw (0,2) node[scale=1.5,left]{$\frac{t}{2}$};
\draw (0,0)--(0,{sqrt(7)});
\draw [-{Straight Barb[angle'=60,scale=3.5]}] (0,{sqrt(7)})--(0,5);
\draw (5,4)--(9,4);
\draw (1.605,0)--({sqrt(13)},2);
\draw (0,0)--(9,0) node[scale=1.5,right]{$t=0$};
\draw [domain=sqrt(13):5] plot (\x, {sqrt(\x*\x-9)});
\draw [dashed,domain=3:sqrt(13)] plot (\x, {sqrt(\x*\x-9)});
\draw (-0.1,4)--(0,4);
\draw (-0.1,2)--(0,2);
\draw (3,-0.1)--(3,0);
\draw (0,-0.5) node[scale=1.5]{$r=0$};
\node[align=center,font=\bfseries, yshift=-1em] (title) 
    at (current bounding box.south)
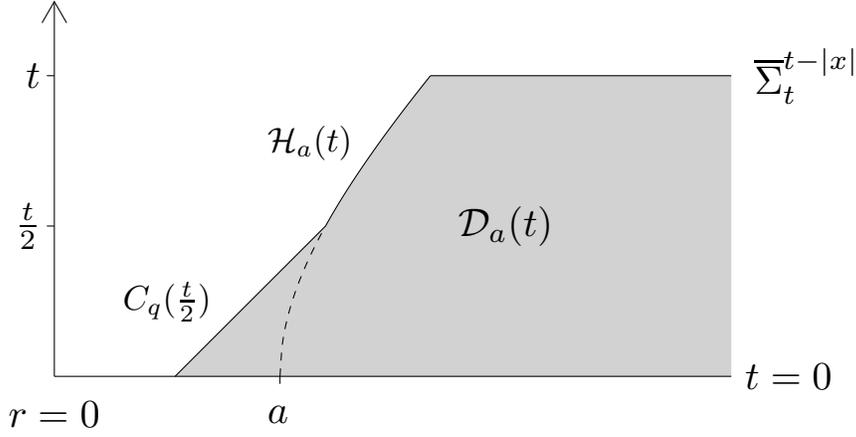
    {The set $\mathcal{D}_a(T)$ and its boundary};
\end{tikzpicture}

\begin{Rq}\label{diffi}
 Note that the reason why we do not include in $\mathcal{H}_a(t)$ the part of the hyperboloid contained in $\{(s,x) \in V_0(T) \hspace{1mm} / \hspace{1mm} s \leq \frac{t}{2} \}$ is technical and will appear in \eqref{conditiondomaine} below. It will permit us to circumvent a difficulty, related to the fact that the outward unit normal vector to $\mathcal{H}_a(t)$ is spacelike, by loosing regularity. Let us mention that we do not have to deal with such a problem for the region $\frac{|x|}{2} \leq t \leq |x|$ (see Section $5$ of \cite{dim4} and in particular Lemma $5.7$).
\end{Rq}
\end{enumerate}

We will obtain the pointwise decay estimate in the region $\min(2,t) < |x| \leq 2t$ from the following Sobolev inequality on a hyperboloid.
\begin{Lem}\label{Sobhyp}
Let $g : V_0(T) \times \R^3_v \rightarrow \R$ be a sufficiently regular function. Then, there exists an absolute constant $C>0$ such that, for all $(t,x) \in V_0(T)$ satisfying $|x| \leq 2t$,
$$ \int_{\R^3_v} |g|(t,x,v) dv \hspace{2mm} \leq \hspace{2mm} \frac{C}{|x|^3}\sum_{|\beta| \leq 3} \int_{\sqrt{a^2+\frac{1}{4}t^2 } \leq |y| \leq |x|} \int_{\R^3_v} \left| \widehat{Z}^{\beta} g \right| \left( \sqrt{|y|^2-a^2},y ,v \right) dvdy, \quad a:= \sqrt{|x|^2-t^2}.$$
We point out that since $(t,x) \in V_0(T)$, $|x| \neq 0$. 
\end{Lem}
\begin{proof}
Let $(t,x)=(t,|x| \omega) \in V_0(T)$ such that $|x| \leq 2t$ and introduce
$$  h : r \mapsto  \int_{\R^3_v} |g| \left( \sqrt{r^2-a^2},r \omega ,v \right) dv, \qquad a^2=|x|^2-t^2. $$
It will be convenient to consider $c:= \left| \frac{15}{16} \right|^{\frac{1}{2}} < 1$, which satisfies, as $|x| \leq 2t$,
\begin{equation}\label{condidomain}
c^2|x|^2 =  a^2+t^2-\frac{|x|^2}{16} \geq a^2 + \frac{t^2}{4},
\end{equation}
so that $h$ is well defined on $[c|x|,+\infty[$. Applying a one-dimensional Sobolev inequality to the function $s \mapsto h(|x|s)$ and then making the change of variable $r=|x|s$, we obtain
\begin{eqnarray*}
\int_{\R^3_v} |g|(t,x,v) dv & = & h(|x|)  \hspace{2mm} \lesssim \hspace{2mm}  \int_{s=c}^1 |h|(|x|s) ds+\int_{s=c}^1 |x| |\partial_r(h)|(|x|s) ds \\ & = &\frac{1}{|x|} \int_{r=c|x|}^{|x|} |h|(r) dr+ \frac{1}{|x|} \int_{c|x|}^{|x|} |x| |\partial_r(h)|(r) dr . 
\end{eqnarray*}
Since $|x| \leq \frac{r}{c}$ on the domain of integration of the two integrals on the right hand side of the previous inequality, we get
\begin{equation}\label{eq:partsob0}
 \int_{\R^3_v} |g|(t,x,v) dv \hspace{2mm} \lesssim \hspace{2mm} \frac{1}{|x|^3} \int_{r=c|x|}^{|x|} |h|(r) r^2 dr+ \frac{1}{|x|^3} \int_{r=c|x|}^{|x|}  |r\partial_r(h)|(r) r^2dr.
\end{equation}
Note now that, with $\Omega_{0r}= r \partial_t+t\partial_r$,
\begin{eqnarray*}
 r\partial_r(h)(r) & = &  \int_{\R^3_v} \frac{r^2}{\sqrt{r^2-a^2}} \partial_t (|g|) \left( \sqrt{r^2-a^2},r \omega ,v \right) dv +\int_{\R^3_v}r \partial_r (|g|) \left( \sqrt{r^2-a^2},r \omega ,v \right) dv \\ & = & \frac{r}{\sqrt{r^2-a^2}} \int_{\R^3_v} \Omega_{0r} (|g|) \left( \sqrt{r^2-a^2},r \omega ,v \right) dv. 
\end{eqnarray*}
Consequently, using that $\Omega_{0r}=\frac{x^i}{r}\Omega_{0i}$ and applying Lemma \ref{lift},
\begin{eqnarray}
\nonumber \left| r\partial_r(h)(r)\right| & \lesssim  & \frac{r}{\sqrt{r^2-a^2}} \left| \frac{x^i}{r} \int_{\R^3_v} \Omega_{0i} (|g|) \left( \sqrt{r^2-a^2},r \omega ,v \right) dv \right| \\ & \lesssim & \frac{r}{\sqrt{r^2-a^2}} \sum_{|\beta| \leq 1} \int_{\R^3_v} | \widehat{Z}^{\beta}g| \left( \sqrt{r^2-a^2},r \omega ,v \right) dv. \label{eq:partsob1}
\end{eqnarray}
Since $|x| \leq 2t$, one has, using \eqref{condidomain},
\begin{equation}\label{eq:partsob2}
\forall \hspace{0.5mm} r \in [c |x|, |x| ], \qquad \frac{r}{\sqrt{r^2-a^2}} \leq \frac{2t}{\sqrt{c^2|x|^2-a^2}} \leq 4.
\end{equation}
Combining first \eqref{eq:partsob0} with \eqref{eq:partsob1}-\eqref{eq:partsob2} and then using again \eqref{condidomain}, we obtain
\begin{eqnarray*}
\int_{\R^3_v} |g|(t,x,v) dv & \lesssim & \frac{1}{|x|^3} \sum_{|\beta| \leq 1} \int_{r=c|x|}^{|x|} \int_{\R^3_v} \left| \widehat{Z}^{\beta} g \right|  \left( \sqrt{r^2-a^2},r \omega ,v \right) dv r^2 dr \\
& \leq & \frac{1}{|x|^3} \sum_{|\beta| \leq 1} \int_{r=\sqrt{a^2+\frac{1}{4}t^2}}^{|x|} \int_{\R^3_v} \left| \widehat{Z}^{\beta} g \right|  \left( \sqrt{r^2-a^2},r \omega ,v \right) dv r^2 dr.
\end{eqnarray*}
The result then follows from this inequality and from
$$ \int_{\R^3_v} \left| \widehat{Z}^{\beta} g \right|  \left( \sqrt{r^2-a^2},r \omega ,v \right) dv \hspace{2mm} \lesssim \hspace{2mm} \sum_{|\xi| \leq 2} \int_{\theta \in \mathbb{S}^2}  \left| \widehat{\Omega}^{\xi} \widehat{Z}^{\beta} g \right|  \left( \sqrt{r^2-a^2},r \theta ,v \right) dv d \mathbb{S}^2,$$
where $\widehat{\Omega}^{\xi} \in \{ \widehat{\Omega}_{12}, \widehat{\Omega}_{13}, \widehat{\Omega}_{23} \}^{|\xi|}$. To prove this last estimate, apply Lemma \ref{Sobshere} to the function $j:(\theta, v) \mapsto k(r\theta , v)$, where $k(r \theta,v) :=  \widehat{Z}^{\beta} f  \left( \sqrt{r^2-a^2},r \theta ,v \right)$ and remark that, as $\widehat{\Omega}_{ij} (r)=0$ and since $\widehat{\Omega}_{ij}$ are homogeneous vector field,
$$\widehat{\Omega}^{\xi} (j)(\theta,v) \hspace{2mm} = \hspace{2mm} \widehat{\Omega}^{\xi} (k)(r \theta, v) \hspace{2mm} = \hspace{2mm} \widehat{\Omega}^{\xi}  \left( \widehat{Z}^{\beta} f  \right)  \left( \sqrt{r^2-a^2},r \theta ,v \right).$$
\end{proof}

We are led to prove  the following energy inequality.
\begin{Lem}\label{energyhyp}
Let $g : V_0(T) \times \R^3_v \rightarrow \R $ be a sufficiently regular function and $(t,x) \in V_0(t)$ such that $|x| \leq 2t$. Then, with $a = \sqrt{|x|^2-t^2}$, we have
\begin{eqnarray*}
\int_{ \sqrt{a^2+\frac{1}{4}t^2} \leq |y| \leq |x|} \int_{\R^3_v} |g| \left( \sqrt{|y|^2-a^2}, y , v \right)  \frac{dv}{(v^0)^2} dy & \lesssim & \sum_{|\beta| \leq 1} \int_{\overline{\Sigma}^0_0} \int_{\R^3_v} |\widehat{Z}^{\beta} g | dx dv \\ & &+ \sum_{|\beta| \leq 1} \int_0^t \int_{\overline{\Sigma}^0_s} \int_{\R^3_v} \left| \TT_F\left( \widehat{Z}^{\beta} g \right) \right| \frac{dv}{v^0} dx ds .
\end{eqnarray*}
\end{Lem}
\begin{proof}
Recall from \eqref{div} that $\partial_{\mu} \int_v |g| \frac{v^{\mu}}{v^0}dv = \int_v \TT_F(g)\frac{g}{|g|} dv$. Applying the euclidian divergence theorem to $ \int_v |g| \frac{v^{\mu}}{v^0}dv$ in the domain $\mathcal{D}_a(t)$ defined in \eqref{Da} and which boundary is given in \eqref{boundary}, we get
\begin{multline*}
 \int_{C_q(\frac{t}{2})} \int_{\R^3_v}\frac{v^{\underline{L}}}{v^0} |g| dv d C_u(\frac{t}{2}) + \int_{\mathcal{H}_a(t)} \int_{\R^3_v} |g| \frac{v\cdot n}{v^0} dv d \mathcal{H}_a(t)+\int_{\overline{\Sigma}^{t-|x|}_t} \int_{\R^3_v} |g| dv dx - \int_{\overline{\Sigma}^q_0} \int_{\R^3_v} |g| dv dx \\
  = \int_{\mathcal{D}_a(t)} \int_{\R^3_v} \frac{g}{|g|} \TT_F(g) dv d \mathcal{D}_a(t) ,
  \end{multline*}
where $n$ is the outward pointing normal unit vector field to $\mathcal{H}_a(t)$ and $d \mathcal{H}_a(t)$ (respectively $d \mathcal{D}_a(t)$) are the volume form on $\mathcal{H}_a(t)$ (respectively $\mathcal{D}_a(t)$). If $\mathcal{H}_a(t)$ is parmeterized by $y \mapsto ( \sqrt{|y|^2-a^2},y)$, we have, with $I_3=diag(1,1,1)$ and denoting by ${}^ty$ the transpose of $y$,
$$ n = \frac{1}{\sqrt{2|y|^2-a^2}}( \sqrt{|y|^2-a^2},-y), \qquad d \mathcal{H}_a(t) = \sqrt{ \det \left( I_3+ \frac{1}{\sqrt{|y|^2-a^2}} {}^t y y \right)} = \frac{\sqrt{2|y|^2-a^2}}{\sqrt{|y|^2-a^2}}.$$
Since
$$\int_{C_q(\frac{t}{2})} \int_{\R^3_v}\frac{v^{\underline{L}}}{v^0} |g| dv \geq 0, \qquad  \int_{\overline{\Sigma}^{t-|x|}_t} \int_{\R^3_v} |g| dv dx \geq 0,$$
we obtain
\begin{multline*}
 \int_{\sqrt{a^2+\frac{1}{4}t^2} \leq |y| \leq |x|} \int_{\R^3_v} |g| \left( \sqrt{ |y|^2-a^2},y,v \right) \frac{\sqrt{|y|^2-a^2}v^0-y^i v_i}{\sqrt{|y|^2-a^2}v^0} dv dy \\ \leq \hspace{2mm} \int_{\overline{\Sigma}^q_0} \int_{\R^3_v} |g| dv dx+\int_{\mathcal{D}_a(t)} \int_{\R^3_v} | \TT_F(g)| dv d \mathcal{D}_a(t).
 \end{multline*}
The problem here is that the integrand in the left hand side of the previous inequality is not nonnegative. Then, observe that, for $s= \sqrt{|y|^2-a^2}$ and since $v^0=\sqrt{1+|v|^2}$,
$$ sv^0-y^i v_i \geq sv^0-|y| |v| = (s-|y|)v^0 +|y|(v^0-|v|)= (s-|y|)v^0 +|y|\frac{(v^0)^2-|v|^2}{v^0+|v|} \geq (s-|y|)v^0 +\frac{|y|}{2v^0}.$$
Consequently, using also $\mathcal{D}_a(t) \subset V_0(t)$, $\Si^q_0 \subset \Si^0_0$ and introducing the notation $s(|y|) :=  \sqrt{ |y|^2-a^2}$, we get
\begin{align}
\nonumber \int_{\sqrt{a^2+\frac{1}{4}t^2} \leq |y| \leq |x|} \int_{\R^3_v} |g| \left(s(|y|),y,v \right) \frac{|y|}{s(|y|)} \frac{dv}{(v^0)^2} dy \hspace{1mm} & \leq \hspace{1mm} 2\int_{\sqrt{a^2+\frac{1}{4}t^2} \leq |y| \leq |x|} \frac{|y|-s(|y|)}{s(|y|)} \int_{\R^3_v} |g| \left(s(|y|), y,v \right)  dv dy \\ &  \quad \hspace{1mm} 2\int_{\overline{\Sigma}^0_0} \int_{\R^3_v} |g| dv dx + 2\int_{0}^t \int_{\overline{\Sigma}^0_s} \int_{\R^3_v} | \TT_F(g)| dv d xds. \label{eq:partie1}
\end{align}
Since 
$$ \sqrt{a^2+\frac{1}{4}t^2} \leq |y| \leq |x| \quad \Rightarrow \quad \frac{|y|}{s(|y|)} =\frac{|y|}{\sqrt{|y|^2-a^2}} \geq \frac{\sqrt{a^2+\frac{1}{4}t^2}}{\sqrt{|x|^2-a^2}} \geq \frac{1}{2},$$
we have
\begin{equation}\label{eq:partie2}
\int_{\sqrt{a^2+\frac{1}{4}t^2} \leq |y| \leq |x|} \int_{\R^3_v} |g| \left(s(|y|), y,v \right) \frac{|y|}{s(|y|)} \frac{dv}{(v^0)^2} dy \hspace{1mm} \geq \hspace{1mm} \frac{1}{2}\int_{\sqrt{a^2+\frac{1}{4}t^2} \leq |y| \leq |x|} \int_{\R^3_v} |g| \left(\sqrt{|y|^2-a^2},y,v \right)  \frac{dv}{(v^0)^2} dy.
\end{equation}
Now, as $\frac{1}{s(|y|)} \leq \frac{2}{t}$ if $|y| \geq \sqrt{a^2+\frac{1}{4}t^2}$, one gets
\begin{align}\label{conditiondomaine}
\nonumber & \int_{\sqrt{a^2+\frac{1}{4}t^2} \leq |y| \leq |x|} \frac{|y|-s(|y|)}{s(|y|)} \int_{\R^3_v} |g| \left(s(|y|), y,v \right)  dv dy \\ & \qquad \leq \frac{2}{t} \int_{\sqrt{a^2+\frac{1}{4}t^2} \leq r \leq |x|} r^2 \left( r-s(r) \right) \int_{\omega \in \mathbb{S}^2} \int_{\R^3_v} |g| \left(s(r), r \omega,v \right)  dv  d \mathbb{S}^2  dr.
\end{align}
Applying the inequality \eqref{eq:tmoinsr} to the function $(s,y,v) \mapsto  g \left(s, |y| \omega,v \right)  $, we then obtain
$$ \int_{\sqrt{a^2+\frac{1}{4}t^2} \leq |y| \leq |x|} \frac{|y|-s(|y|)}{s(|y|)} \int_{\R^3_v} |g| \left(s(|y|), y,v \right)  dv dy \lesssim \frac{1}{t} \sum_{|\beta| \leq 1} \int_{\sqrt{a^2+\frac{1}{4}t^2} \leq r \leq |x|} \left\| \int_{\R^3_v} \left| \widehat{Z}^{\beta} g \right| dv \right\|_{L^1(\overline{\Sigma}^0_{s(r)})} dr.$$
Note now that, using $|x| \leq 2t$ and the energy inequality of Proposition \ref{energyf},
\begin{eqnarray*}
 \frac{1}{t} \int_{\sqrt{a^2+\frac{1}{4}t^2} \leq r \leq |x|} \left\| \int_{\R^3_v} \left| \widehat{Z}^{\beta} g \right| dv \right\|_{L^1(\overline{\Sigma}^0_{s(r)})} dr & \leq & \frac{|x|}{t} \sup_{a \leq r \leq |x|}  \left\| \int_{\R^3_v} \left| \widehat{Z}^{\beta} g \right| dv \right\|_{L^1(\overline{\Sigma}^0_{s(r)})} \\ & \lesssim & \left\| \int_{\R^3_v} \left| \widehat{Z}^{\beta} g \right| dv \right\|_{L^1(\overline{\Sigma}^0_0)}+\int_0 \int_{\overline{\Sigma}^0_s} \int_{\R^3_v} \left| \TT_F \left( \widehat{Z}^{\beta} g \right) \right| \frac{dv}{v^0} dy ds.
 \end{eqnarray*}
 The result then follows from \eqref{eq:partie1}, \eqref{eq:partie2} and the last two inequalities.
\end{proof}
The estimate \eqref{kevatalenn:opti} in the region $\min (2,t) < |x| \leq 2t$ then ensues from Lemmas \ref{Sobhyp} and \ref{energyhyp}, applied respectively to $f$ and to $\widehat{Z}^{\xi}f$ for all $|\xi| \leq 3$. This concludes the proof of Proposition \ref{Propdecayopti}.

Adapting Lemmas \ref{Sobhyp} and \ref{energyhyp} to norms carrying weights in\footnote{For this, note that since $\TT_0(z)=0$, $\TT_F(z)=F(v,\nabla_v z)$. The other additional arguments are similar to those used in Lemmas \ref{lift}-\ref{Sobshere} and Proposition \ref{KS1} in order to deal with the weight $z^a$.} $z$ and applying \eqref{decaytr33} to $\frac{1}{|v^0|^2}f$, one can prove the following estimate.
\begin{Pro}\label{Cordecayopti}
Let $u \leq 0$, $(k,q) \in \R_+^2$ and $f : V_u(T) \times \R^3_v \rightarrow \R$ be a sufficiently regular function. For all $(t,x) \in V_u(T)$,
\begin{multline*}
 \int_{\R^3_v} |f|(t,x,v) \frac{dv}{(v^0)^{2+2k}} \hspace{1mm}  \lesssim \hspace{1mm}  \frac{1}{\tau_+^{3+k}\tau_-^q}\sum_{|\beta| \leq 4} \Bigg(  \int_{\overline{\Sigma}^u_0} \int_{\R^3_v} z^{k+q}\left| \widehat{Z}^{\beta} f \right| dv dx \\
 + \int_0^t\int_{\overline{\Sigma}^u_s} \int_{\R^3_v} z^{k+q}\left| \TT_F ( \widehat{Z}^{\beta} f ) \right| \frac{dv}{v^0} dx ds+\int_0^t\int_{\overline{\Sigma}^u_s} \int_{\R^3_v} z^{k+q-1}\left| F(v,\nabla_v z) \right| |\widehat{Z}^{\beta} f | \frac{dv}{v^0} dx ds \Bigg).
\end{multline*}
\end{Pro}
\subsection{Estimates for the electromagnetic field}

In this subsection, we introduce first the energy norm used in this paper to study the electromagnetic field and, secondly, we derive pointwise decay estimates from it through Klainerman-Sobolev inequalities. We consider, for the remaining of this section, $G$ a sufficiently regular $2$-form defined on $V_b(T)$ and we denote by $(\alpha,\underline{\alpha},\rho,\sigma)$ its null decomposition. We suppose that $G$ satisfies
\begin{eqnarray}
\nonumber \nabla^{\mu} G_{\mu \nu} & = & J_{\nu} \\ \nonumber
\nabla^{\mu} {}^* \! G_{\mu \nu} & = & 0,
\end{eqnarray}
with $J$ a sufficiently regular $1$-form defined on $V_b(T)$.
\begin{Def}\label{defMax1}
Let $N \in \mathbb{N}$. We define, for $t \in [0,T[$,
\begin{eqnarray}
\nonumber \mathcal{E}^{b}[G](t) & := & \int_{\Si^b_t}\left( |\alpha|^2+|\underline{\alpha}|^2+2|\rho|^2+2|\sigma|^2 \right)dx +\sup_{u < b} \int_{C_u(t)} \left( |\alpha|^2+|\rho|^2+|\sigma|^2 \right) dC_u(t), \\ \nonumber
 \mathcal{E}^b_N[G](t) & := & \sum_{0 \leq k \leq N} \hspace{0.5mm} \sum_{  Z^{\gamma} \in \mathbb{K}^k } \hspace{0.5mm} \mathcal{E}^{b}[\mathcal{L}_{ Z^{\gamma}}(G)](t).
\end{eqnarray}
\end{Def}

\begin{Pro}\label{energyMax1}
We have, for all $ t \in [0,T[$,
$$ \mathcal{E}^{b}[G](t) \hspace{2mm} \leq \hspace{2mm} 2\mathcal{E}^{b}[G](0) + 8\int_0^t \int_{\Si^b_s} |G_{\mu 0} J^{\mu}| dx ds.$$
\end{Pro}
\begin{proof}
Recall from Lemma \ref{maxwellbis} that $\nabla^{\mu} T[G]_{\mu 0}=G_{0 \nu} J^{\nu}$. Hence, applying the divergence theorem in $V_{u}(t)$, for $u < b $, we get
\begin{equation}\label{eq:1}
 \int_{\Si^{u}_t}  T[G]_{00}dx + \frac{1}{\sqrt{2}} \int_{C_{u}(t)} T[G]_{L0}dC_{u}(t) = \int_{\Si^{u}_0}  T[G]_{00}dx-\int_0^t \int_{\Si^{u}_s }   G_{0 \nu} J^{\nu} dx ds.
 \end{equation}
We then obtain
 \begin{eqnarray}
 \nonumber \sup_{u < b} \int_{C_{u}(t)} T[G]_{L0}dC_{u}(t) & \leq & \int_{\Si^{b}_0}  |T[G]_{00}|dx+\int_0^t \int_{\Si^{b}_s }   \left| G_{0 \nu} J^{\nu} \right| dx ds, \\ \nonumber
\int_{\Si^b_t} T[G]_{00} dx \hspace{2mm} = \hspace{2mm} \sup_{u < b} \int_{\Si^{u}_t}  T[G]_{00}dx  & \leq & \int_{\Si^{b}_0}  |T[G]_{00}|dx+\int_0^t \int_{\Si^{b}_s }   \left| G_{0 \nu} J^{\nu} \right| dx ds.
 \end{eqnarray}
It then remains to add the previous two inequalities and to notice, using \eqref{tensorcompo}, that
$$ 4T[G]_{00} = |\alpha|^2+|\underline{\alpha}|^2+2|\rho|^2+2|\sigma|^2  \hspace{1cm} \text{and} \hspace{1cm} 2 T[G]_{L0}=  |\alpha|^2+|\rho|^2+|\sigma|^2 .$$
\end{proof}

In order to prove pointwise decay estimates on $G$, we will use the following three Lemmas. The first one, which is proved in Appendix $D$ of \cite{massless}, extends the results of Lemma \ref{goodderiv} for the null components of a $2$-form.

\begin{Lem}\label{null}
We have, denoting by $\zeta$ any of the null component $\alpha$, $\underline{\alpha}$, $\rho$ or $\sigma$,
$$\tau_- \left| \nabla_{\underline{L}} \zeta (G)\right|+\tau_+ \left| \nabla_{L} \zeta (G) \right| \lesssim \sum_{|\gamma| \leq 1} \left|  \zeta \left( \mathcal{L}_{Z^{\gamma}}(G) \right) \right|, \hspace{15mm} (1+r)\left| \slashed{\nabla} \zeta (G) \right| \lesssim |\zeta(G)|+\sum_{ \Omega \in \Or} \left|  \zeta \left( \mathcal{L}_{\Omega}(G) \right) \right|. $$
\end{Lem}
The following result, also proved in Appendix $D$ of \cite{massless}, presents commutation properties between $\mathcal{L}_{\Omega}$, $\nabla_{\partial_r}$, $\nabla_{L}$ or $\nabla_{\underline{L}}$ and the null decomposition of $G$.
\begin{Lem}\label{randrotcom}
Let $\Omega \in \Or$. Then, denoting by $\zeta$ any of the null component $\alpha$, $\underline{\alpha}$, $\rho$ or $\sigma$,
$$ [\mathcal{L}_{\Omega}, \nabla_{\partial_r}]( G)=0, \hspace{1.2cm} \mathcal{L}_{\Omega}(\zeta(G))= \zeta ( \mathcal{L}_{\Omega}(G) ) \hspace{1.2cm} \text{and} \hspace{1.2cm} \nabla_{\partial_r}(\zeta(G))= \zeta ( \nabla_{\partial_r}(G) ) .$$
Similar results hold for $\mathcal{L}_{\Omega}$ and $\nabla_{\partial_t}$, $\nabla_L$ or $\nabla_{\underline{L}}$. For instance, $\nabla_{L}(\zeta(G))= \zeta ( \nabla_{L}(G) )$.
\end{Lem}
We now recall the Sobolev inequalities which will be used to prove the pointwise decay estimates on the null components of the electromagnetic field. For this, we introduce $|U(y)|^2_{\mathbb{O},k} := \sum_{|\beta| \leq k} \left| \mathcal{L}_{\Omega^{\beta}}(U) \right|^2$, where $\Omega^{\beta} \in \Or^{|\beta|}$.
\begin{Lem}\label{Sob}
Let $U$ be a sufficiently regular tensor field defined on $\R^3$. Then,
$$\forall \hspace{0.5mm} x \neq 0, \hspace{8mm} |U(x)| \lesssim \frac{1}{|x|^{\frac{3}{2}}} \left( \int_{|y| \geq |x|} |U(y)|^2_{\mathbb{O},2}+|y|^2|\nabla_{\partial_r} U(y) |^2_{\mathbb{O},1} dy \right)^{\frac{1}{2}}.$$
If $t \in \R_+$ and $|x| \geq t-b$, we have
$$\forall \hspace{0.5mm} x \neq 0, \hspace{8mm} |U(x)| \lesssim \frac{1}{|x|\tau_-^{\frac{1}{2}}} \left( \int_{|y| \geq t-b} |U(y)|^2_{\mathbb{O},2}+\tau_-^2|\nabla_{\partial_r} U(y) |^2_{\mathbb{O},1} dy \right)^{\frac{1}{2}}.$$
\end{Lem}
\begin{proof}
The first inequality is proved in Lemma $2.3$ of \cite{CK} and the second one can be proved similarly as inequality $(ii)$ of Lemma $2.3$ of \cite{CK}.
\end{proof}

We now prove the pointwise decay estimates used in this article.
\begin{Pro}\label{decayMaxellext}
For all $(t,x) \in V_b(T)$, we have
$$|\underline{\alpha}|(t,x)+|\rho|(t,x)   \hspace{1mm}  \lesssim \hspace{1mm} \frac{ \sqrt{\mathcal{E}^b_2[G](t)}}{\tau_+\tau_-^{\frac{1}{2}}}, \qquad |\alpha|(t,x) + |\sigma|(t,x)  \hspace{1mm}  \lesssim \hspace{1mm} \frac{ \sqrt{\mathcal{E}^b_2[G](t)}}{\tau_+^{\frac{3}{2}}}.$$
\end{Pro}
\begin{Rq}\label{rhodecayagain}
Following the proof of the estimates on $\alpha$ or $\sigma$, we could also prove
$$|\rho|(t,x)  \hspace{1mm}  \lesssim \hspace{1mm} \frac{ \sqrt{\mathcal{E}^b_2[G](t)}+\sum_{|\beta| \leq 1}\|r \mathcal{L}_{Z^{\beta}}(J)^L \|_{L^2(\overline{\Sigma}^b_t)}}{\tau_+^{\frac{3}{2}}}.$$
During the proof of Theorem \ref{theorem}, $\|r \mathcal{L}_{Z^{\beta}}(J)^L \|_{L^2(\overline{\Sigma}^b_t)}$ will not be uniformly bounded in time so we will rather use the result of Proposition \ref{decayrhoalter}.
\end{Rq}
\begin{proof}
Let $(t,x) \in V_b(T)$. In this proof, $\Omega^{\beta}$ will always be in $\Or^{|\beta|}$ and $Z^{\gamma}$ in $\mathbb{K}^{|\gamma|}$. Let $\theta$ be any of the null components $\alpha$, $\rho$, $\sigma$ or $\underline{\alpha}$ and $\zeta$ be either $\alpha$ or $ \sigma$. As $\nabla_{\partial_r}$ and $\mathcal{L}_{\Omega}$ commute with the null decomposition (see Lemma \ref{randrotcom}), Lemma \ref{Sob} gives us
\begin{align}
\nonumber r^2 \tau_- |\theta|^2 & \lesssim \int_{ |y| \geq t-b} | \theta |^2_{\mathbb{O},2}+\tau_-^2|\nabla_{\partial_r} ( \theta) |_{\mathbb{O},1}^2 dy \lesssim \sum_{ |\beta| \leq 1 } \sum_{|\gamma| \leq 2} \int_{ |y| \geq t-b} | \theta ( \mathcal{L}_{Z^{\gamma}} (G) |^2+\tau_-^2| \theta ( \mathcal{L}_{\Omega^{\beta}} (\nabla_{\partial_r} G)) |^2 dy, \\
r^3 |\zeta|^2 & \lesssim \int_{ |y| \geq t-b} | \zeta |^2_{\mathbb{O},2}+r^2|\nabla_{\partial_r} ( \zeta) |_{\mathbb{O},1}^2 dy \lesssim \sum_{ |\beta| \leq 1 } \sum_{|\gamma| \leq 2} \int_{ |y| \geq t-b} | \zeta ( \mathcal{L}_{Z^{\gamma}} (G) |^2+r^2|  \zeta ( \mathcal{L}_{\Omega^{\beta}} (\nabla_{\partial_r} G)) |^2 dy. \label{eq:zetabbb}
\end{align}
Since $\nabla_{\partial_r}$ commute with $\mathcal{L}_{\Omega}$ and the null decomposition (see Lemma \ref{randrotcom}), we get, using $2\partial_r= L-\underline{L}$,
\begin{eqnarray}
\nonumber \tau_-|  \theta ( \mathcal{L}_{\Omega^{\beta}} (\nabla_{\partial_r} G)) | & = & \tau_- |  \nabla_{\partial_r} \theta ( \mathcal{L}_{\Omega^{\beta}} (G)) | \hspace{2mm} \leq \hspace{2mm} \tau_-|  \nabla_L \theta ( \mathcal{L}_{\Omega^{\beta}} (G)) |+ \tau_-|  \nabla_{\underline{L}} \theta ( \mathcal{L}_{\Omega^{\beta}} (G)) |, \\
 r|  \zeta ( \mathcal{L}_{\Omega^{\beta}} (\nabla_{\partial_r} G)) | & = & r|  \nabla_{\partial_r} \zeta ( \mathcal{L}_{\Omega^{\beta}} (G)) | \hspace{2mm} \leq \hspace{2mm} r|  \nabla_L \zeta ( \mathcal{L}_{\Omega^{\beta}} (G))|+ r|  \nabla_{\underline{L}} \zeta ( \mathcal{L}_{\Omega^{\beta}} (G) )| . \label{eq:zetabbbb}
 \end{eqnarray}
We start by dealing with $\theta$. According to Lemma \ref{null},
$$ \tau_-|  \nabla_L \theta ( \mathcal{L}_{\Omega^{\beta}} (G)) |+ \tau_-|  \nabla_{\underline{L}} \theta ( \mathcal{L}_{\Omega^{\beta}} (G)) | \hspace{2mm} \lesssim \hspace{2mm} \sum_{ |\gamma| \leq |\beta|+1} | \theta ( \mathcal{L}_{Z^{\gamma}} (G) |,$$
so that, as $\tau_+ \lesssim r $ in $V_b(T)$,
$$\tau_+^2 \tau_- |\theta|^2 \hspace{2mm} \lesssim \hspace{2mm} \sum_{|\gamma| \leq 2} \int_{ |y| \geq t-b} | \theta ( \mathcal{L}_{Z^{\gamma}} (G) |^2 dx \hspace{2mm} \lesssim \hspace{2mm} \mathcal{E}^b_2[G](t).$$
In order to improve the decay rate near the light cone for the components $\alpha$ and $\sigma$, note that we obtain from the equations \eqref{nullsigma}-\eqref{nullalpha}, applied to $\mathcal{L}_{\Omega^{\beta}}(G)$,
$$
 r|  \nabla_{\underline{L}} \alpha ( \mathcal{L}_{\Omega^{\beta}} (G)) |+r|  \nabla_{\underline{L}} \sigma ( \mathcal{L}_{\Omega^{\beta}} (G)) |  \hspace{2mm} \lesssim \hspace{2mm} r|  \nabla_{L} \underline{\alpha} ( \mathcal{L}_{\Omega^{\beta}} (G) )|+r|  \slashed{\nabla} \underline{\alpha} ( \mathcal{L}_{\Omega^{\beta}} (G)) |+r|  \slashed{\nabla} \rho ( \mathcal{L}_{\Omega^{\beta}} (G)) |+|\mathcal{L}_{\Omega^{\beta}} (G) |.$$
Applying Lemma \ref{null}, we then get
\begin{equation}\label{nullalphasigma}
r|  \nabla_L \alpha ( \mathcal{L}_{\Omega^{\beta}} (G))|+ r|  \nabla_{\underline{L}} \alpha ( \mathcal{L}_{\Omega^{\beta}} (G) )|+ r|  \nabla_L \sigma ( \mathcal{L}_{\Omega^{\beta}} (G))|+ r|  \nabla_{\underline{L}} \sigma ( \mathcal{L}_{\Omega^{\beta}} (G) )| \hspace{2mm} \lesssim \hspace{2mm}  \sum_{ |\gamma| \leq |\beta|+1} | \mathcal{L}_{Z^{\gamma}} (G) |,
\end{equation}
which implies, using \eqref{eq:zetabbb}-\eqref{eq:zetabbbb} and since $\tau_+ \lesssim r $ in $V_b(T)$,
$$\tau_+^3 (|\alpha|^2+|\sigma|^2)  \hspace{2mm} \lesssim \hspace{2mm} \sum_{|\gamma| \leq 2} \int_{ |y| \geq t-b} |  \mathcal{L}_{Z^{\gamma}} (G) |^2 dx  \hspace{2mm} \lesssim \hspace{2mm} \mathcal{E}^b_2[G](t).$$
This concludes the proof.
\end{proof}
The optimal pointwise decay estimate on the component $\rho$ will be obtained by using the following result, which requires to control one derivative more of the electromagnetic field in $L^2$.
\begin{Pro}\label{decayrhoalter}
Let $C>0$ and assume that 
\begin{equation}\label{eq:imprdecay} \forall \hspace{0.5mm} (t,x) \in V_b(T), \hspace{1cm} \sum_{|\gamma| \leq 1} |\mathcal{L}_{Z^{\gamma}}(G)|(t,x) \hspace{2mm} \leq \hspace{2mm} \frac{C}{\tau_+ \tau_-^{\frac{1}{2}}}, \quad |J^L|(t,x) \hspace{2mm} \leq \hspace{2mm} \frac{C}{\tau_+^{\frac{7}{4}} \tau_-^{\frac{3}{4}} } .
\end{equation}
Then, we have
\begin{equation}\label{eq:imprdecay2} \forall \hspace{0.5mm} (t,x) \in V_b(T), \hspace{1cm} |\rho|(t,x) \hspace{2mm} \lesssim \hspace{2mm} \frac{C}{\tau_+^\frac{3}{2}}.
\end{equation}
\end{Pro}
\begin{proof}
Let $(t,x)=(t,r \omega) \in V_b(T)$ and recall that $\tau_+ \lesssim r$ in this region. Using first \eqref{nullrho} as well as Lemma \ref{null} and then \eqref{eq:imprdecay}, we obtain
\begin{equation}\label{nullrhorho}
|\nabla_{\underline{L}} \hspace{0.5mm} \rho|(t,x) \hspace{2mm} \lesssim \hspace{2mm} \frac{1}{\tau_+} \sum_{|\gamma| \leq 1} |\mathcal{L}_{Z^{\gamma}}(G)|(t,x)+|J^L|(t,x) \hspace{2mm} \lesssim \hspace{2mm} \frac{C}{\tau_+^{\frac{7}{4}} \tau_-^{\frac{3}{4}}}.
 \end{equation}
Let us now introduce the function $$\varphi ( \underline{u}, u) := \rho \left( \frac{\underline{u}+u}{2}, \frac{\underline{u}-u}{2} \omega \right), \quad \text{which satisfies}, \quad   |\nabla_{\underline{L}} \varphi |(\underline{u},u) \lesssim \frac{C}{(1+\underline{u})^{\frac{7}{4}}(1+|u|)^{\frac{3}{4}}}.$$
Hence, since $|\rho|(0,x) \leq C(1+|x|)^{-\frac{3}{2}}$ by assumption,
\begin{eqnarray}
\nonumber |\rho|(t,x) & = & |\varphi|(t+r,t-r) \hspace{2mm} \leq \hspace{2mm} \int_{u=-t-r}^{t-r} |\nabla_{\underline{L}} \varphi |(t+r,u) du + |\varphi|(t+r,-t-r) \\ \nonumber
& = & \int_{u=-t-r}^{t-r} |\nabla_{\underline{L}} \varphi |(t+r,u) du + |\rho|(0,(t+r)\omega) \\ \nonumber
& \lesssim &  \frac{ C}{(1+r+r)^{\frac{7}{4}}} \int_{u=-t-r}^{b} \frac{du}{(1+|u|)^{\frac{3}{4}}} + \frac{C}{(1+t+r)^{\frac{3}{2}}} \hspace{2mm} \lesssim \hspace{2mm} \frac{C}{(1+t+r)^{\frac{3}{2}}} .
\end{eqnarray}
\end{proof}
Improved decay estimates can be obtained when translations are applied to the electromagnetic field.
\begin{Pro}\label{extradecayderiv}
The null components of $\mathcal{L}_{\partial_{\mu}}(G)$, for $\mu \in \llbracket 0, 3 \rrbracket$, satisfy the following estimates on $ V_b(T)$,
$$\tau_+|\alpha (\mathcal{L}_{\partial_{\mu}} G)|+\tau_+|\sigma (\mathcal{L}_{\partial_{\mu}} G)|+\tau_-|\underline{\alpha} (\mathcal{L}_{\partial_{\mu}} G)| \lesssim  \sum_{|\gamma| \leq 1} |\mathcal{L}_{Z^{\gamma}}( G)|, \qquad \tau_+|\rho (\mathcal{L}_{\partial_{\mu}} G)| \lesssim \tau_+|J^L|+ \sum_{|\gamma| \leq 1} |\mathcal{L}_{Z^{\gamma}}( G)|.$$
\end{Pro}
\begin{proof}
During the first step of the proof, we will use that $\nabla_{0} L =-\nabla_{0} \underline{L}=0$ and $\nabla_i L=-\nabla_i \underline{L} = \frac{x^j}{r^3}\Omega_{ji}=C^B\frac{e_B}{r}$, where $C^B$ are bounded functions. Hence, as $\alpha_A = G_{AL}$ and $\mathcal{L}_{\partial_{\mu}}=\nabla_{\partial_{\mu}}$, we have
$$ |\alpha (\mathcal{L}_{\partial_{\mu}} G)_A|=|\alpha (\nabla_{\partial_{\mu}} G)_A|= |\nabla_{\mu} \left( \alpha (G) \right)_A -G (A,\nabla_{\mu} L)| \lesssim |\nabla_{\mu} \alpha (G) |+\frac{1}{r}|\sigma (G)|.$$
Since $\rho=\frac{1}{2}G_{L \underline{L}}$,
$$ |\rho (\mathcal{L}_{\partial_{\mu}} G)|=|\nabla_{\mu} \rho(G) -\frac{1}{2}G (\nabla_{\mu} L, \underline{L})   -\frac{1}{2}G (L,\nabla_{\mu} \underline{L})| \lesssim |\nabla_{\mu} \rho (G) |+\frac{1}{r}|G|.$$
One can prove similarly, using for the component $\sigma$ that $\nabla_{\mu} e_B = \frac{1}{r}(C^L_{\mu}L+C^{\underline{L}}_{\mu} \underline{L}+C^D_{\mu} e_D)$, where $C^{\nu}_{\mu}$ are bounded functions, that
$$ |\underline{\alpha} (\mathcal{L}_{\partial_{\mu}} G)| \lesssim |\nabla_{\mu} \underline{\alpha} (G) |+\frac{1}{r}|\sigma (G)|, \qquad |\sigma (\mathcal{L}_{\partial_{\mu}} G)| \lesssim |\nabla_{\mu} \sigma (G) |+\frac{1}{r}|G|.$$
As $\tau_-\leq \tau_+ \lesssim r$ on $V_b(T)$, it remains us to bound $|\nabla_{\mu} \zeta (G)|$ for any null component $\zeta \in \{ \alpha, \rho , \sigma , \underline{\alpha} \}$. The starting point is the inequality 
$$ |\nabla_{\mu} \zeta (G)| \lesssim |\nabla_{\underline{L}} \zeta (G)|+|\nabla_{L} \zeta (G)|+|\nabla_{e_1} \zeta (G)|+|\nabla_{e_2} \zeta (G)|.$$
Let $(dt,dr,de^1,de^2)$ be the dual basis of $(\partial_t,\partial_r,e_1,e_2)$. Note that if $\phi$ is a function and $\Phi=\Phi_D de^D$ a $1$-form tangential to the $2$-spheres,
$$ |\nabla_{e_A} \phi | \leq |\slashed{\nabla} \phi|, \qquad |\nabla_{e_A} \Phi| = |\slashed{\nabla}_{e_A} \Phi+\Gamma^D_{Ar}\Phi_D dr| \leq |\slashed{\nabla} \Phi|+\frac{1}{r}|\Phi|,$$
where $\Gamma$ are the Christofel symbols in Minkowski spacetime in the nonholonomic basis $(\partial_t, \partial_r, e_1,e_2)$, verifying in particular $|\Gamma^D_{Ar}| \lesssim r^{-1}$. Applying this to $\phi = \rho, \sigma$ as well as $\Phi=\alpha,\underline{\alpha}$ and using Lemma \ref{null}, we get
$$ \forall \hspace{0.5mm} (t,x) \in V_b(T), \qquad |\nabla_{L} \zeta (G)|(t,x)+|\nabla_{e_1} \zeta (G)|(t,x)+|\nabla_{e_2} \zeta (G)|(t,x) \lesssim \frac{1}{\tau_+}\sum_{|\gamma| \leq 1} |\zeta ( \mathcal{L}_{Z^{\gamma}}(G))|(t,x).$$
Finally, one can bound $|\nabla_{\underline{L}} \underline{\alpha} (G)|$ by applying Lemma \ref{null}, $|\nabla_{\underline{L}} \alpha (G)|+|\nabla_{\underline{L}} \sigma (G)|$ using \eqref{nullalphasigma} and $|\nabla_{\underline{L}} \rho (G)|$ using the first inequality of \eqref{nullrhorho}.
\end{proof}
\begin{Rq}\label{rqforintroo}
Following the proof of the previous Proposition and estimating $|\nabla_{\underline{L}} \alpha (G)|$ by using the relation of Lemma $D.3$ of \cite{massless} instead of \eqref{nullalphasigma}, one could improve the decay estimate on $\alpha( \mathcal{L}_{\partial_{\mu}} G)$. More precisely, there holds
$$\tau_+|\alpha (\mathcal{L}_{\partial_{\mu}} G)| \lesssim \tau_+|\slashed{J}|+ \sum_{|\gamma| \leq 1} |\alpha (\mathcal{L}_{Z^{\gamma}} (G))|+|\rho (\mathcal{L}_{Z^{\gamma}} (G))|+|\sigma (\mathcal{L}_{Z^{\gamma}} (G))|, \qquad \slashed{J}=(J^{e_1},J^{e_2}).$$
\end{Rq}
\subsection{Null properties of the Vlasov equation}\label{sec5}

In order to take advantage of the null structure of the commuted transport equation, we will expand quantities such as $\mathcal{L}_{Z^{\gamma}}(F) \left( v,\nabla_v g \right)$, with $g$ a regular function, in null coordinates. We will then use the following lemma.

\begin{Lem}\label{calculF}
Let $G$ be a sufficiently regular $2$-form, $(\alpha, \underline{\alpha}, \rho, \sigma)$ its null components and $g$ a sufficiently regular function. Then,
\begin{equation*}
\left| G \left( v, \nabla_v g \right) \right| \hspace{2mm} \lesssim \hspace{2mm} \left( |\alpha|+|\rho|+|\sigma|+\frac{v^{\underline{L}}+|v^A|}{v^0} |\underline{\alpha}| \right) \hspace{-1mm} \left( \tau_+\left| \nabla_{t,x} g\right|+ \sum_{\widehat{Z} \in \K} \left| \widehat{Z} g \right| \right) . 
\end{equation*}
\end{Lem}
\begin{proof}
Expanding $G(v, \nabla_v g )$ with null components, we get
\begin{eqnarray}
\nonumber G(v, \nabla_v g ) & = & 2 \rho\left( v^L \left( \nabla_v g \right)^{\underline{L}}-v^{\underline{L}} \left( \nabla_v g \right)^L \right)+v^B \varepsilon_{BA} \sigma \left( \nabla_v g \right)^A-v^L \alpha_A \left( \nabla_v g \right)^A+v^A \alpha_A \left( \nabla_v g \right)^L \\ \label{eq:calculF}
& & -v^{\underline{L}} \underline{\alpha}_A \left( \nabla_v g \right)^A+v^A \underline{\alpha}_A \left( \nabla_v g \right)^{\underline{L}}.
\end{eqnarray}
It then remains to bound the components of $\nabla_v g$ using $v^0 \partial_{v^i} = \widehat{\Omega}_{0i}-t \partial_i-x^i \partial_t$. Note that in order to prove the inequality \eqref{calculGintro} written in the introduction, one has to notice that the radial component $\left( \nabla_v g \right)^r=\frac{1}{2}\left( \nabla_v g \right)^L=-\frac{1}{2}\left( \nabla_v g \right)^{\underline{L}}$ has a better behavior than $\nabla_v g$ since
\begin{equation*}
v^0\left( \nabla_v g \right)^r = \frac{x^i}{r} v^0\partial_{v^i} g=\frac{x^i}{r} \widehat{\Omega}_{0i}g-Sg+(t-r) \underline{L} g.
\end{equation*}
\end{proof}

\section{Bootstrap assumptions and strategy of the proof}\label{sec6}

Let $N \geq 9$, $b \leq -1$ and $(\delta, \eta) \in \R_+^2$ be two constants such that $0 < 2 \delta \leq \eta < \frac{1}{4N}$. From now, we drop the dependance in $b$ of all the quantities defined previously (such as the energy norms $\E_b$ and $\E^{q,\eta}_{Q,b}$ defined in \eqref{defE} and \eqref{defE2} or $\mathcal{E}_N^{b}$). Let $(f_0,F_0)$ be an initial data set satisfying the assumptions of Theorem \ref{theorem}. Then, by a local well-posedness argument, there exists a unique maximal solution to the Vlasov-Maxwell system defined in $V_{b}(T^*)$, with $T^* \in \R_+^* \cup \{ + \infty \}$. Let $T \in ]0,T^*]$ be the largest time such that\footnote{Note that $T>0$ by continuity. Remark also that, considering if necessary $\epsilon_1 = C_1 \epsilon$, with $C_1$ a constant depending only on $N$, we can suppose without loss of generality that the energy norms are initially smaller than $\epsilon$. We refer to Appendix $B$ of \cite{massless} for the details of the computations for similar energy norms.}, for all $t \in [0,T[$,
\begin{eqnarray}\label{bootf1}
 \E^{N+13,\eta}_{N-3}[f](t) &  \leq & 4 \epsilon (1+t)^{\delta} , \\ \label{bootf2}
 \E^{N+9,\eta}_{N}[f](t) & \leq & 4 \epsilon (1+t)^{\delta}, \\ \label{bootF1}
 \mathcal{E}_{N}[F](t) & \leq & 4 \epsilon .
\end{eqnarray}
The remainder of this paper is devoted to the improvement of these inequalities which will prove that $T=T^*$ and then $T^* = + \infty$, implying Theorem \ref{theorem}. Let us expose the immediate consequences of the bootstrap assumptions. Using the Klainerman-Sobolev inequality of Proposition \ref{KS1} and \eqref{bootf1} (respectively \eqref{bootf2}), one has, since $4-3(1+2\eta) \geq 0$,
\begin{eqnarray}\label{decayf}
\forall \hspace{0.5mm} (t,x) \in V_b(T), \hspace{3mm} |\beta| \leq N-6, \hspace{15mm} \int_v  \sqrt{z}^{\hspace{0.5mm} N+9-(1+2\eta)\beta_P} \left|\widehat{Z}^{\beta} f \right| dv & \lesssim & \epsilon \frac{(1+t)^{\delta}}{\tau_+^{2} \tau_-}, \\ 
\forall \hspace{0.5mm} (t,x) \in V_b(T), \hspace{3mm} |\beta| \leq N-3, \hspace{15mm} \int_v  \sqrt{z}^{\hspace{0.5mm} N+5-(1+2\eta)\beta_P} \left| \widehat{Z}^{\beta} f \right| dv & \lesssim & \epsilon \frac{(1+t)^{\delta}}{\tau_+^{2} \tau_-}. \label{decayf2}
\end{eqnarray}
For the Maxwell field, we obtain from the bootstrap assumptions the following pointwise decay estimates.
\begin{Pro}\label{decayMax}
For all $(t,x) \in V_b(T)$ and $|\gamma| \leq N-3$,
\begin{equation*}
 \left| \alpha \left( \mathcal{L}_{Z^{\gamma}}(F) \right) \right|(t,x)+\left| \rho \left( \mathcal{L}_{Z^{\gamma}}(F) \right) \right|(t,x)+\left| \sigma \left( \mathcal{L}_{Z^{\gamma}}(F) \right) \right|(t,x) \hspace{0.5mm} \lesssim \hspace{0.5mm} \frac{\sqrt{\epsilon}}{\tau_+^{\frac{3}{2}} }, \qquad \left| \underline{\alpha} \left( \mathcal{L}_{Z^{\gamma}}(F) \right) \right|(t,x) \hspace{0.5mm} \lesssim \hspace{0.5mm} \frac{\sqrt{\epsilon}}{\tau_+ \tau_-^{\frac{1}{2}}}.
\end{equation*}
If $Z^{\gamma}$ contains at least one translation, i.e. if $\gamma_T \geq 1$, we have the following improved decay estimates. For all $(t,x) \in V_b(T)$ and $|\gamma| \leq N-3$,
\begin{align*}
 \left| \alpha \left( \mathcal{L}_{Z^{\gamma}}(F) \right) \right|(t,x)+\left| \sigma \left( \mathcal{L}_{Z^{\gamma}}(F) \right) \right|(t,x) \hspace{0.5mm} & \lesssim \hspace{0.5mm} \frac{\sqrt{\epsilon}}{\tau_+^{2}\tau_-^{\frac{1}{2}} }, \qquad  \qquad \left| \underline{\alpha} \left( \mathcal{L}_{Z^{\gamma}}(F) \right) \right|(t,x) \hspace{0.5mm} \lesssim \hspace{0.5mm} \frac{\sqrt{\epsilon}}{\tau_+ \tau_-^{\frac{3}{2}}} \\
 \left| \rho \left( \mathcal{L}_{Z^{\gamma}}(F) \right) \right|(t,x) \hspace{0.5mm} & \lesssim \hspace{0.5mm} \frac{\sqrt{\epsilon}}{\tau_+^{2-\delta} \tau_-^{\frac{1}{2}+\delta}} .
\end{align*}
\end{Pro}
\begin{proof}
According to Proposition \ref{decayMaxellext} and since \eqref{bootF1} holds, we have for all $(t,x) \in V_b(T)$ and $|\gamma| \leq N-2$,
\begin{equation}\label{tobeusedhere}
 \left| \alpha \left( \mathcal{L}_{Z^{\gamma}}(F) \right) \right|(t,x)+\left| \sigma \left( \mathcal{L}_{Z^{\gamma}}(F) \right) \right|(t,x) \hspace{0.5mm} \lesssim \hspace{0.5mm} \frac{\sqrt{\epsilon}}{\tau_+^{\frac{3}{2}} }, \qquad \left| \rho \left( \mathcal{L}_{Z^{\gamma}}(F) \right) \right|(t,x)+\left| \underline{\alpha} \left( \mathcal{L}_{Z^{\gamma}}(F) \right) \right|(t,x) \hspace{0.5mm} \lesssim \hspace{0.5mm} \frac{\sqrt{\epsilon}}{\tau_+ \tau_-^{\frac{1}{2}}}.
\end{equation}
This gives the expected estimates for the components $\alpha$, $\sigma$ and $\underline{\alpha}$. The improved decay estimates on $\rho$, for $|\gamma| \leq N-3$, can then be obtained using Proposition \ref{decayrhoalter}, \eqref{tobeusedhere} and \eqref{decayf2}.

If $\gamma_T \geq 1$, then, as $[Z,\partial_{\mu}] \in \{0 \} \cup \{ \pm \partial_{\lambda} \hspace{1mm} / \hspace{1mm} 0 \leq \lambda \leq 3 \}$ for $0 \leq \mu \leq 3$, we can assume without loss of generality that $Z^{\gamma} = \partial_{\mu} Z^{\gamma_0}$. The result then follows from Proposition \ref{extradecayderiv} applied to $\mathcal{L}_{Z^{\gamma_0}}(F)$, \eqref{decayf2} and the pointwise decay estimates proved previously.
\end{proof}
The proof is organized as follows:
\begin{itemize}
\item We start by improving the bootstrap assumptions \eqref{bootf1} and \eqref{bootf2} by several applications of the approximate conservation law of Proposition \ref{energyf}. Exploiting the null structure of the non linearity $\mathcal{L}_{Z^{\gamma}}(F)(v,\nabla_v \widehat{Z}^{\beta} f )$ is then fundamental in order to bound the spacetime integrals arising from the energy estimates.
\item Then, we improve the bound on the energy norm of the electromagnetic field \eqref{bootF1}. For this, we use the energy estimate of Proposition \ref{energyMax1} and we make crucial use of the null structure of the source terms of the Maxwell equations.
\item The last step consists in proving an estimate on $ \|\int_v | z\widehat{Z}^{\beta} f | dv \|_{L^2(\Si^b_t)}$ for $|\beta| \geq N-2$. We then rewrite all Vlasov equations as an inhomogeneous system of transport equations. We deal with the homogenous part by taking advantage of the smallness assumption on the $N+3$ derivatives of $f$ at $t=0$ as well as the pointwise decay estimates of Proposition\ref{decayMax}. We will decompose the inhomogeneous part as a product $KY$ where $|K|^2 Y \in L^1_v L^1(\Si^b_t)$ and $\int_v |Y| dv$ is a decaying function.
\end{itemize}
\begin{Rq}\label{rqdecayforintro}
We could control $\E^{N+13,\eta}_{N-2}[f]$ and save one derivative by using the full null structure of the system given by\footnote{Note that the strategy used in Section \ref{sec9} would then have to be adapted in order to keep of the null structure of the Vlasov equation. See for instance \cite{massless}.} \eqref{calculGintro} and the estimate $|\rho(\mathcal{L}_{Z^{\gamma}}(F))| \lesssim \sqrt{\epsilon} \min ( \tau_+^{-1} \tau_-^{-\frac{1}{2}}, \tau_+^{-\frac{3}{2}+\delta})$, for $|\gamma| \leq N-2$, which could be obtained from Proposition \ref{decayMaxellext} and Remark \ref{rhodecayagain}.

Let us also mention that a certain number of weights $z$ could be saved in a few steps of the proof (e.g. in \eqref{decayf}-\eqref{decayf2}, Subsection \ref{secFL2}, Section \ref{sec8}). More precisely, we could propagate the weaker energy norms $\E^{N+A_{\eta}, \eta}_N[f]$ and $\E^{N+B_{\eta}, \eta}_{N-3}[f]$, where $A_{\eta} = (N+5)\eta+\frac{5}{2}$ and $B_{\eta}=A_{\eta}+\frac{5}{4}+2\eta$, allowing us to save almost $\frac{15}{4}$ power of $x$ in the condition on the initial norm of the Vlasov field. For the readibility of the proof and since the initial decay on $f$ in $x$ would not be optimal either, we prefer to work with more convenient energy norms.
\end{Rq}
\section{Improvement of the energy bound on the particle density}\label{sec7}

The aim of this section is to prove that, for $\epsilon$ small enough, $\E^{N+9,\eta}_N[f] \leq 3 \epsilon (1+t)^{\delta}$ for all $t \in [0,T[$ (we will sketch the improvement of the estimate on $\E^{N+13,\eta}_{N-3}[f]$ as it is very similar). For this, recall that $\E_N^{N+9,\eta}[f] (0) \leq \epsilon$ and let us prove that
$$ \forall \hspace{0.5mm} |\kappa| \leq N, \hspace{2mm} \forall \hspace{0.5mm} t \in [0,T[, \hspace{1cm} \E[ \sqrt{z}^{\hspace{0.5mm} N+9-(1+2\eta) \kappa_P} \widehat{Z}^{\kappa} f ](t) - 2 \E[ \sqrt{z}^{\hspace{0.5mm} N+9-(1+2\eta) \kappa_P} \widehat{Z}^{\kappa} f ](0) \hspace{2mm} \lesssim \hspace{2mm} \epsilon^{\frac{3}{2}} (1+t)^{ \delta}.$$
We then fix $|\kappa| \leq N$ and we denote $\frac{1}{2}( N+9-(1+2\eta) \kappa_P)$ by $a$. Note, by Lemma \ref{weights}, that
\begin{equation}\label{eq:forestimates}
  \TT_F(z^a \widehat{Z}^{\kappa} f )=F(v,\nabla_v z^a )\widehat{Z}^{\kappa} f + z^a\TT_F( \widehat{Z}^{\kappa} f ).
\end{equation}
Thus, in view of the energy estimate of Proposition \ref{energyf} and commutation formula of Proposition \ref{Maxcom}, it suffices to prove that
\begin{equation}\label{weight}
\int_0^t \int_{\Si^b_s} \int_v \left| z^{a-1} F \left( v,\nabla_v z \right) \widehat{Z}^{\kappa} f  \right| \frac{dv}{v^0} dx ds \hspace{1mm} \lesssim \hspace{1mm} \epsilon^{\frac{3}{2}} (1+t)^{\delta}
\end{equation}
and that the following proposition holds, where $[\gamma] := \max(0,1-\gamma_T)$.
\begin{Pro}\label{estif}
Let $\gamma$ and $\beta$ be such that $|\gamma|+|\beta| \leq |\kappa|$, $|\beta| \leq |\kappa|-1$ and $\beta_P+[\gamma] \leq \kappa_P$. Then,
$$ \int_0^t \int_{\Si^b_s} \int_v \left| z^a \mathcal{L}_{Z^{\gamma}}(F) \left( v,\nabla_v \widehat{Z}^{\beta} f \right) \right| \frac{dv}{v^0} dx ds \hspace{1mm} \lesssim \hspace{1mm} \epsilon^{\frac{3}{2}} (1+t)^{\delta}.$$
\end{Pro}
The pointwise decay estimates given in Theorem \ref{theorem} on the Vlasov field can then be obtained from Proposition \ref{Cordecayopti}, \eqref{weight} and the previous Proposition.

The remainder of the section is divided in four parts. The first two ones are devoted to the proof of \eqref{weight} and Proposition \ref{estif}. Then, we explain briefly how to improve the bound on $\E^{N+13,\eta}_{N-3}[f]$. In the fourth subsection, we prove an $L^2$ estimate on $\int_v z |\widehat{Z}^{\beta} f |dv$ which will be useful for Section \ref{sec8}. Finally, we briefly explain how, under slightly stronger decay assumptions, we could prove a uniform bound on $\E^{N+9,\eta}_N[f]$.
\subsection{Proof of inequality \eqref{weight}}\label{secweight}
Note first that we have $\left| \nabla_{t,x} z \right| \leq 1$ and, using Lemma \ref{weights}, $ \left| \widehat{Z} (z) \right| \lesssim z$. Applying Lemma \ref{calculF} to $(G,g)=(F,z)$, we can then observe that it suffices to prove that
$$  I \hspace{1mm} := \hspace{1mm} \int_0^t \int_{\Si^b_s} \int_v (\tau_++z) \left( |\rho(F)|+ |\alpha(F)|+|\sigma(F)|+\frac{v^{\underline{L}}+|v^A|}{v^0} |\underline{\alpha}(F)| \right) \left| z^{a-1} \widehat{Z}^{\kappa} f \right| \frac{dv}{v^0}dx ds \hspace{1mm} \lesssim \hspace{1mm} \epsilon^{\frac{3}{2}} (1+t)^{\delta}.$$
According to the pointwise decay estimates of Proposition \ref{decayMax} as well as the inequalities $1 \lesssim \sqrt{v^0 v^{\underline{L}}}$ and $v^{\underline{L}} \lesssim \tau_+^{-1} v^0 z$, which come from Lemma \ref{weights1}, we have
$$ |\rho(F)|+ |\alpha(F)|+|\sigma(F)| \lesssim \frac{\sqrt{\epsilon}}{\tau_+^{\frac{3}{2}}}, \qquad 1 \lesssim \sqrt{v^0 v^{\underline{L}}} \lesssim  \frac{v^0 \sqrt{z}}{\sqrt{\tau_+}} ,$$
so that, since $1 \leq z$ (cf \eqref{boundz}),
$$ (\tau_++z)(|\rho(F)|+ |\alpha(F)|+|\sigma(F)|) \hspace{1mm} \lesssim \hspace{1mm} \sqrt{\epsilon} \frac{v^0}{\tau_+}z.$$
Now, using the pointwise decay estimate on $\underline{\alpha}(F)$ given by Proposition \ref{decayMax} and the inequality $v^{\underline{L}} +|v^A| \lesssim \tau_+^{-1} v^0 z$ (see Lemma \ref{weights1}),
$$ (\tau_++z)\frac{v^{\underline{L}}+|v^A|}{v^0} |\underline{\alpha}(F)| \hspace{1mm} \lesssim \hspace{1mm} \frac{\sqrt{\epsilon}}{\tau_+}z.$$
Consequently, using the bootstrap assumption \eqref{bootf2},
\begin{eqnarray}
\nonumber I & \lesssim  &  \int_0^t \int_{\Si^b_s} \frac{\sqrt{\epsilon}}{\tau_+} \int_v   \left| z^a \widehat{Z}^{\kappa} f \right| dv dx ds \hspace{2mm} \lesssim \hspace{2mm} \sqrt{\epsilon} \int_0^t \frac{\E[z^a \widehat{Z}^{\kappa} f](s)}{1+s}ds \\ \nonumber
& \lesssim & \sqrt{\epsilon} \int_0^t \frac{\E^{N+9,\eta}_N[ f](s)}{1+s}ds \hspace{2mm} \lesssim \hspace{2mm} \epsilon^{\frac{3}{2}} \int_0^t \frac{(1+s)^{\delta}}{1+s}ds  \hspace{2mm} \lesssim \hspace{2mm} \epsilon^{\frac{3}{2}} (1+t)^{\delta}.
\end{eqnarray}

\subsection{Proof of Proposition \ref{estif}}\label{secpointF}

Let $\gamma$ and $\beta$ satisfying $|\beta|+|\gamma| \leq |\kappa|$, $|\beta| \leq |\kappa|-1$ and $\beta_P+[\gamma] \leq \kappa_P$. We will estimate the electromagnetic field pointwise if $|\gamma| \leq N-3$ and in $L^2$ otherwise. For simplify the presentation, we will use $\zeta$ in order to denote one of the good null components $\alpha$, $\rho$ or $\sigma$. According to Lemma \ref{calculF}, we need to bound by $\epsilon^{\frac{3}{2}} (1+t)^{\delta}$, for all $\widehat{\Gamma} \in \K$, the following integrals,
\begin{eqnarray}
\nonumber J^{a}_{\zeta} & := & \int_0^t \int_{\Si^b_s}  \left|  \zeta \left( \mathcal{L}_{Z^{\gamma}}(F) \right) \right| \int_v  \left| z^a \widehat{\Gamma} \widehat{Z}^{\beta} f \right|  \frac{dv}{v^0} dx ds, \qquad \qquad \zeta \in \{ \alpha, \rho , \sigma \} \\
\nonumber I^{a}_{\zeta} & := & \int_0^t \int_{\Si^b_s} \tau_+\left| \zeta \left( \mathcal{L}_{Z^{\gamma}}(F) \right) \right| \int_v \left| z^a \nabla_{t,x} \widehat{Z}^{\beta} f \right|  \frac{dv}{v^0} dx ds, \qquad \zeta \in \{ \alpha, \rho , \sigma \} \\ \nonumber
 J^{a}_{\underline{\alpha}} & := & \int_0^t \int_{\Si^b_s}   \left| \underline{\alpha} \left( \mathcal{L}_{Z^{\gamma}}(F) \right) \right|  \int_v  \frac{|v^A|+v^{\underline{L}}}{v^0}\left| z^a \widehat{\Gamma} \widehat{Z}^{\beta} f \right|  \frac{dv}{v^0} dx ds, \\ \nonumber
 I^{a}_{\underline{\alpha}} & := & \int_0^t \int_{\Si^b_s} \tau_+  \left| \underline{\alpha} \left( \mathcal{L}_{Z^{\gamma}}(F) \right) \right|  \int_v  \frac{|v^A|+v^{\underline{L}}}{v^0}\left| z^a \nabla_{t,x} \widehat{Z}^{\beta} f \right|  \frac{dv}{v^0} dx ds.
\end{eqnarray}
In order to close the energy estimates, we will have to pay attention to the hierarchies discussed in Remark \ref{hierar}. Indeed, if, say, $|\gamma| \leq N-3$ the pointwise decay estimates on the electromagnetic field given by Proposition \ref{decayMax} are not strong enough to compensate the weight $\tau_+$ in $I^a_{\zeta}$ and $I^a_{\underline{\alpha}}$. Consequently, we use the condition $\beta_P+[\gamma] \leq \kappa_P$, so that
\begin{itemize}
\item either $\beta_P < \kappa_P$ and we can gain decay through the inequality $\tau_- \lesssim z$ since we control in that case $\E[z^{a+\frac{1}{2}+2\eta} \nabla_{t,x} \widehat{Z}^{\beta} f]$.
\item Or $\gamma_T \geq 1$ and the electromagnetic field decay much faster in that case (see Proposition \ref{decayMax}).
\end{itemize} We fix, for the remainder of this subsection, $\widehat{\Gamma} \in \K$. The proof is divided in two parts. First, we treat the case where the electromagnetic field can be estimated pointwise ($|\gamma| \leq N-3$). Otherwise we necessarily have $|\beta| \leq 2 \leq N-7$ and we can use the estimate \eqref{decayf} on the Vlasov field.
\subsubsection{If $|\gamma| \leq N-3$}\label{secFpoint}

Suppose first that $\beta_P < \kappa_P$, which implies $a+\frac{1}{2}+\eta \leq \frac{1}{2}(N+9-(1+2\eta)\beta_P)$. Note also that if $\xi$ and $\theta$ are such that $\widehat{Z}^{\xi}=\widehat{\Gamma} \widehat{Z}^{\beta}$ and $\widehat{Z}^{\theta}=\partial_{\mu} \widehat{Z}^{\beta}$, we have $\xi^P \leq \beta_P+1$ and $\theta_P=\beta_P$. Consequently, the bootstrap assumption \eqref{bootf2} and Remark \ref{Rq32} give
\begin{equation}\label{eq:two}
 \E[ z^a \widehat{\Gamma} \widehat{Z}^{\beta} f ](t) + \E[ z^{a+\frac{1}{2}+\eta} \nabla_{t,x} \widehat{Z}^{\beta} f ](t) \hspace{2mm} \lesssim \hspace{2mm} \E^{N+9,\eta}_N[f](t) \hspace{2mm} \lesssim \hspace{2mm}  \epsilon (1+t)^{\delta}.
\end{equation}
According to the pointwise decay estimates of Proposition \ref{decayMax}, we have $\left| \mathcal{L}_{Z^{\gamma}}(F)  \right| \lesssim \sqrt{\epsilon} \tau_+^{-1}$, so
\begin{align*}
J^{a}_{\alpha}+J^{a}_{\rho}+J^{a}_{\sigma}+J^{a}_{\underline{\alpha}} \hspace{1mm} & \lesssim \hspace{1mm} \int_0^t \frac{\sqrt{\epsilon}}{1+s} \int_{\Si^b_s} \int_v \left| z^a \widehat{\Gamma} \widehat{Z}^{\beta} f \right| dv dx ds \hspace{1mm} \lesssim \hspace{1mm} \sqrt{\epsilon} \int_0^t \frac{\E[z^a \widehat{\Gamma} \widehat{Z}^{\beta} f](s) }{1+s} ds \\
& \lesssim \hspace{1mm} \epsilon^{\frac{3}{2}} \int_0^t \frac{(1+s)^{\delta}}{1+s} ds \hspace{1mm} \lesssim \hspace{1mm} \epsilon^{\frac{3}{2}} (1+t)^{\delta} .
\end{align*}
For the remaining integrals, we need to use that we control the $L^1$ norm of $z^{a+\frac{1}{2}+\eta}\nabla_{t,x} \widehat{Z}^{\beta} f$ in order to gain decay through the weight $z^{\frac{1}{2}+\eta}$. Recall from Proposition \ref{decayMax} that, for $\zeta \in \{ \alpha , \rho , \sigma \}$ and since $|\gamma| \leq N-3$, $\left|\zeta \left( \mathcal{L}_{Z^{\gamma}}(F) \right) \right| \lesssim \sqrt{\epsilon} \tau_+^{-\frac{3}{2}}$. Hence, using $1 \lesssim \sqrt{v^0 v^{\underline{L}}}$ and $1 \lesssim \tau_-^{-\frac{1}{2}-\eta} z^{\frac{1}{2}+\eta}$ (see Lemma \ref{weights1}), we get
\begin{eqnarray}
\nonumber \tau_+\left| \zeta \left( \mathcal{L}_{Z^{\gamma}}(F) \right) \right| \int_v \left| z^a \nabla_{t,x} \widehat{Z}^{\beta} f \right|  \frac{dv}{v^0}  & \lesssim &  \int_v \sqrt{\frac{\epsilon v^{\underline{L}}}{\tau_+\tau_-^{1+2\eta}v^0}} \left| z^{a+\frac{1}{2}+\eta} \nabla_{t,x} \widehat{Z}^{\beta} f \right|  dv \\ \nonumber
& \lesssim &   \int_v \left( \frac{\sqrt{\epsilon}}{\tau_+}+\frac{\sqrt{\epsilon} v^{\underline{L}}}{\tau_-^{1+2\eta}v^0} \right) \left| z^{a+\frac{1}{2}+\eta} \nabla_{t,x} \widehat{Z}^{\beta} f \right|  dv, 
\end{eqnarray}
Now use again Lemma \ref{weights1} in order to get $|v^A| \lesssim v^0 v^{\underline{L}}$ and $1 \lesssim \tau_-^{-\frac{1}{2}-\eta} z^{\frac{1}{2}+\eta}$, so that, as $|\underline{\alpha} \left( \mathcal{L}_{Z^{\gamma}}(F) \right)| \lesssim \sqrt{\epsilon}\tau_+^{-1} \tau_-^{-\frac{1}{2}}$ by Proposition \ref{decayMax},
$$\tau_+ \left| \underline{\alpha} \left( \mathcal{L}_{Z^{\gamma}}(F) \right) \right|  \int_v  \frac{|v^A|+v^{\underline{L}}}{v^0}\left| z^a \nabla_{t,x} \widehat{Z}^{\beta} f \right|  \frac{dv}{v^0}  \hspace{2mm} \lesssim \hspace{2mm}  \int_v \frac{\sqrt{\epsilon} v^{\underline{L}}}{\tau_-^{1+\eta}v^0} \left| z^{a+\frac{1}{2}+\eta} \nabla_{t,x} \widehat{Z}^{\beta} f \right|  dv.$$
This leads to
$$ I^a_{\alpha}+I^a_{\rho}+I^a_{\sigma}+I^a_{ \underline{\alpha}} \hspace{1mm} \lesssim \hspace{1mm} \int_0^t \frac{\sqrt{\epsilon}}{1+s} \int_{\Si^b_s} \int_v \left| z^{a+\frac{1}{2}+\eta} \nabla_{t,x} \widehat{Z}^{\beta} f \right| dv dx ds+\int_0^t  \int_{\Si^b_s} \frac{\sqrt{\epsilon}}{\tau_-^{1+\eta}} \int_v \frac{v^{\underline{L}}}{v^0} \left| z^{a+\frac{1}{2}+\eta} \nabla_{t,x} \widehat{Z}^{\beta} f \right|dv dx ds .$$
The first integral can be bounded using \eqref{eq:two},
$$ \int_0^t \frac{\sqrt{\epsilon}}{1+s} \int_{\Si^b_s} \int_v \left| z^{a+\frac{1}{2}+\eta} \nabla_{t,x} \widehat{Z}^{\beta} f \right| dv dx ds \lesssim \sqrt{\epsilon}\int_0^t \frac{\E[ z^{a+\frac{1}{2}+\eta} \nabla_{t,x} \widehat{Z}^{\beta} f ](s)}{1+s} ds \lesssim \epsilon^{\frac{3}{2}} (1+t)^{\delta}.$$
To deal with the second one, we need to use the null foliation $(C_u(t))_{u < b}$ (c.f. Lemma \ref{foliationexpli}) and we recall that $\tau_-=\sqrt{1+u^2}$ is contant on $C_u(t)$. Hence, using \eqref{eq:two}, one has
\begin{eqnarray*}
\int_0^t  \int_{\Si^b_s} \frac{\sqrt{\epsilon}}{\tau_-^{1+\eta}} \int_v \frac{v^{\underline{L}}}{v^0} \left| z^{a+\frac{1}{2}+\eta} \nabla_{t,x} \widehat{Z}^{\beta} f \right|dv dx ds & = & \int_{u=-\infty}^b \int_{C_u(t)} \frac{\sqrt{\epsilon}}{\tau_-^{1+\eta}} \int_v \frac{v^{\underline{L}}}{v^0} \left| z^{a+\frac{1}{2}+\eta} \nabla_{t,x} \widehat{Z}^{\beta} f \right|dv dC_u(t) \frac{du}{\sqrt{2}} \\
& = &  \int_{u=-\infty}^b \frac{\sqrt{\epsilon}}{\tau_-^{1+\eta}} \int_{C_u(t)}  \int_v \frac{v^{\underline{L}}}{v^0} \left| z^{a+\frac{1}{2}+\eta} \nabla_{t,x} \widehat{Z}^{\beta} f \right|dv dC_u(t) \frac{du}{\sqrt{2}} \\ 
\\
& \leq & \sqrt{\epsilon} \int_{u=-\infty}^b \frac{1}{\tau_-^{1+\eta}} \E[ z^{a+\frac{1}{2}+\eta} \nabla_{t,x} \widehat{Z}^{\beta} f ](t) du \\
& \lesssim & \epsilon^{\frac{3}{2}} (1+t)^{\delta} \int_{u=-\infty}^b \frac{du}{(1+|u|)^{1+\eta}} \hspace{2mm} \lesssim \hspace{2mm} \epsilon^{\frac{3}{2}} (1+t)^{\delta}.
\end{eqnarray*}
We suppose now that $\beta_P = \kappa_P$, so that $\gamma_T \geq 1$. Note that \eqref{eq:two} does not hold in that case. The bootstrap assumption \eqref{bootf2} merely gives us, as $a= \frac{1}{2}(N+9-(1+2\eta)\beta_P)$,
\begin{equation}\label{eq:three}
 \E[ z^{a-\frac{1}{2}-\eta} \widehat{\Gamma} \widehat{Z}^{\beta} f ](t)+ \E[ z^a \nabla_{t,x} \widehat{Z}^{\beta} f ](t) \hspace{2mm} \lesssim \hspace{2mm} \E^{N+9,\eta}_N[f](t) \hspace{2mm} \lesssim \hspace{2mm} \epsilon (1+t)^{\delta}.
\end{equation}
Applying successively the inequality $|v^A| \lesssim  v^0 v^{\underline{L}}$, which comes from Lemma \ref{weights1}, and Proposition \ref{decayMax} ($\gamma_T \geq 1$), we have
$$ \left| \underline{\alpha} \left( \mathcal{L}_{Z^{\gamma}}(F) \right) \right| \frac{|v^A|+v^{\underline{L}}}{v^0} \hspace{2mm} \lesssim \hspace{2mm} \left| \underline{\alpha} \left( \mathcal{L}_{ Z^{\gamma}}(F) \right) \right| v^{\underline{L}} \hspace{2mm} \lesssim \hspace{2mm} \sqrt{\epsilon} \frac{v^{\underline{L}}}{\tau_+\tau_-^{\frac{3}{2}}}.$$
Recall that $\zeta \in \{ \alpha , \rho , \sigma \}$. Using successively Proposition \ref{decayMax} ($\gamma_T \geq 1$) and the inequality $1 \lesssim \sqrt{v^0 v^{\underline{L}}}$ of Lemma \ref{weights1}, we get
$$ \left| \zeta \left( \mathcal{L}_{Z^{\gamma}}(F) \right) \right|  \hspace{2mm} \lesssim \hspace{2mm}  \frac{\sqrt{\epsilon}}{\tau_+^{2-\delta}\tau_-^{\frac{1}{2}+\delta}} \hspace{2mm} \lesssim \hspace{2mm}  \sqrt{\epsilon} \frac{\sqrt{v^0 v^{\underline{L}}}}{\tau_+^{2-\delta}\tau_-^{\frac{1}{2}+\delta}} \hspace{2mm} \lesssim \hspace{2mm} \sqrt{\epsilon} \frac{v^0 }{\tau_+^{\frac{5}{2}}}+\sqrt{\epsilon} \frac{v^{\underline{L}} }{\tau_+ \tau_-^{\frac{3}{2}}}.$$
Consequently, we obtain
$$I^a_{\alpha}+I^a_{\rho}+I^a_{\sigma}+I^a_{\underline{\alpha}} \hspace{2mm} \lesssim \hspace{2mm} \int_0^t \int_{\Si^b_s} \frac{\sqrt{\epsilon}}{\tau_+^{\frac{3}{2}}} \int_v \left| z^a \nabla_{t,x} \widehat{Z}^{\beta} f \right|  dv dx ds+\int_0^t \int_{\Si^b_s} \frac{\sqrt{\epsilon}}{\tau_-^{\frac{3}{2}}} \int_v \frac{v^{\underline{L}}}{v^0} \left| z^a \nabla_{t,x} \widehat{Z}^{\beta} f \right|  dv dx ds$$
and, using $z^{\frac{1}{2}+\eta}\lesssim \tau_+^{\frac{1}{2}+\eta} \leq \tau_+ $ (see \eqref{boundz}),
\begin{align*}
J^a_{\alpha}+J^a_{\rho}+J^a_{\sigma}+J^a_{\underline{\alpha}} \hspace{2mm} & \lesssim \hspace{2mm} \int_0^t \int_{\Si^b_s} \frac{\sqrt{\epsilon}}{\tau_+^{\frac{5}{2}}} \int_v \left| z^a \widehat{\Gamma} \widehat{Z}^{\beta} f \right|  dv dx ds+\int_0^t \int_{\Si^b_s} \frac{\sqrt{\epsilon}}{\tau_+\tau_-^{\frac{3}{2}}} \int_v \frac{v^{\underline{L}}}{v^0} \left| z^a \widehat{\Gamma} \widehat{Z}^{\beta} f \right|  dv dx ds \\
& \lesssim \hspace{2mm} \int_0^t \int_{\Si^b_s} \frac{\sqrt{\epsilon}}{\tau_+^{\frac{3}{2}}} \int_v \left| z^{a-\frac{1}{2}-\eta} \widehat{\Gamma} \widehat{Z}^{\beta} f \right|  dv dx ds+\int_0^t \int_{\Si^b_s} \frac{\sqrt{\epsilon}}{\tau_-^{\frac{3}{2}}} \int_v \frac{v^{\underline{L}}}{v^0} \left| z^{a-\frac{1}{2}-\eta} \widehat{\Gamma} \widehat{Z}^{\beta} f \right|  dv dx ds.
\end{align*}
It then remains us to bound four integrals. We introduce $g \in \{ z^{a-\frac{1}{2}-\eta} \widehat{\Gamma} \widehat{Z}^{\beta} f, z^{a} \nabla_{t,x} \widehat{Z}^{\beta} f\}$ in order to unify their studies and we remark that $\E[g](t) \lesssim \epsilon (1+t)^{\delta}$ by \eqref{eq:three}. Note first that, since $\delta < \frac{1}{4}$,
$$\int_0^t \int_{\Si^b_s} \frac{\sqrt{\epsilon}}{\tau_+^{\frac{3}{2}}} \int_v \left| g \right|  dv dx ds \hspace{1mm}  \lesssim \hspace{1mm} \sqrt{\epsilon}\int_0^t \frac{\E[g ](s)}{(1+s)^{\frac{3}{2}}} ds \hspace{1mm} \lesssim \hspace{1mm} \epsilon^{\frac{3}{2}} \int_0^t \frac{ds}{(1+s)^{\frac{3}{2}-\delta}} ds \hspace{1mm} \lesssim \hspace{1mm} \epsilon^{\frac{3}{2}}.$$
 Applying Lemma \ref{foliationexpli} in order to take advantage of the null foliation $(C_u(t))_{u < b}$, we have
\begin{eqnarray*}
\int_0^t \int_{\Si^b_s} \frac{\sqrt{\epsilon}}{\tau_-^{\frac{3}{2}}} \int_v \frac{v^{\underline{L}}}{v^0} \left|g\right|  dv dx ds & \leq & \int_{u=-\infty}^b \int_{C_u(t)} \frac{\sqrt{\epsilon}}{\tau_-^{\frac{3}{2}}} \int_v \frac{v^{\underline{L}}}{v^0} |g| dv d C_u(t) du \\
& = & \int_{u=-\infty}^b  \frac{\sqrt{\epsilon}}{\tau_-^{\frac{3}{2}}}  \int_{C_u(t)} \frac{\sqrt{\epsilon}}{\tau_-^{\frac{3}{2}}} \int_v \frac{v^{\underline{L}}}{v^0} |g| dv d C_u(t) du \\ 
& \leq & \sqrt{\epsilon} \E[g ](t)  \int_{u=-\infty}^b  \frac{du}{\tau_-^{\frac{3}{2}}} \hspace{2mm} \lesssim \hspace{2mm} \epsilon^{\frac{3}{2}} (1+t)^{\delta}.
\end{eqnarray*}
This concludes the proof of Proposition \ref{estif} in the case $|\gamma| \leq N-3$.

\subsubsection{When $|\gamma| \geq N-2$}\label{secFL2}

In that case, we have $|\beta| \leq 2 \leq N-7$. The strategy consists in separate the Maxwell field to the Vlasov field through a Cauchy-Schwarz inequality in $(t,x)$ and then use the null foliation $(C_u(t))_{u < b}$ in order to control the particle density, which can be estimated pointwise. Additional features of the null structure of the system will also be used for the most problematic terms. 

According to Lemma \ref{weights1}, one has $1 \lesssim \frac{z}{\tau_-}$ and $\tau_+\frac{|v^A|+v^{\underline{L}}}{v^0} \lesssim z$, so that, by the Cauchy-Schwarz inequality in $(t,x)$,
\begin{align*}
I^a_{\zeta} \hspace{1mm} & \lesssim \hspace{1mm} \left| \int_0^t \int_{\Si^b_s} \frac{|\zeta(\mathcal{L}_{Z^{\gamma}}(F))|^2}{\tau_-^{\frac{3}{2}}} dx ds \int_0^t \int_{\Si^b_s} \frac{\tau_+^2}{\tau_-^{\frac{1}{2}}}  \left| \int_v \left| z^{a+1} \nabla_{t,x} \widehat{Z}^{\beta} f \right| \frac{dv}{v^0} \right|^2 dx ds \right|^{\frac{1}{2}}, \\
I^a_{\underline{\alpha}} \hspace{1mm} & \lesssim \hspace{1mm} \left| \int_0^t \int_{\Si^b_s} \frac{|\underline{\alpha}(\mathcal{L}_{Z^{\gamma}}(F))|^2}{(1+s)^{\frac{3}{2}}} dx ds \int_0^t \int_{\Si^b_s}(1+s)^{\frac{3}{2}}  \left| \int_v \left| z^{a+1} \nabla_{t,x} \widehat{Z}^{\beta} f \right| \frac{dv}{v^0} \right|^2 dx ds \right|^{\frac{1}{2}}.
\end{align*}
By the bootstrap assumption \eqref{bootF1} and since $\zeta \in \{ \alpha , \rho , \sigma \}$, we have 
$$\forall \hspace{0.5mm} t \in [0,T[, \qquad \sup_{u<b}\int_{C_u(t)} |\zeta(\mathcal{L}_{Z^{\gamma}}(F))|^2 d C_u(t) + \int_{\Si^b_t} |\underline{\alpha}(\mathcal{L}_{Z^{\gamma}}(F))|^2 dx \hspace{1mm} \lesssim \hspace{1mm} \epsilon.$$ This yields
\begin{align*}
\int_0^t \int_{\Si^b_s} \frac{|\zeta(\mathcal{L}_{Z^{\gamma}}(F))|^2}{\tau_-^{\frac{3}{2}}} dx ds \hspace{1mm} & = \hspace{1mm} \int_{u=-\infty}^b \frac{1}{\tau_-^{\frac{3}{2}}} \int_{C_u(t)} |\zeta(\mathcal{L}_{Z^{\gamma}}(F))|^2 d C_u(t) du \hspace{1mm} \lesssim \hspace{1mm} \epsilon  \int_{u=-\infty}^b \frac{1}{\tau_-^{\frac{3}{2}}} du \hspace{1mm} \lesssim \hspace{1mm} \epsilon, \\
\int_0^t \int_{\Si^b_s} \frac{|\underline{\alpha}(\mathcal{L}_{Z^{\gamma}}(F))|^2}{(1+s)^{\frac{3}{2}}} dx ds \hspace{1mm} & \lesssim \hspace{1mm} \int_0^t \frac{\epsilon}{(1+s)^{\frac{3}{2}}} ds \hspace{1mm} \lesssim \hspace{1mm} \epsilon.
\end{align*}
Hence, as $(1+s)^{\frac{3}{2}}  \leq \frac{\tau_+^2}{\tau_-^{\frac{1}{2}}}$, it remains us to bound sufficiently well
$$ \mathfrak{I} \hspace{1mm} := \hspace{1mm} \int_0^t \int_{\Si^b_s} \frac{\tau_+^2}{\tau_-^{\frac{1}{2}}}  \left| \int_v \left| z^{a+1} \nabla_{t,x} \widehat{Z}^{\beta} f \right| \frac{dv}{v^0} \right|^2 dx ds .$$
Using Lemma \ref{foliationexpli} in order to use a null foliation, the inequality $\frac{1}{v^0} \lesssim \sqrt{\frac{v^{\underline{L}}}{v^0}}$ coming from Lemma \ref{weights1} and the Cauchy-Schwarz inequality in the variable $v$, we obtain
$$ \mathfrak{I} \hspace{1mm} \lesssim \hspace{1mm} \int_{u=-\infty}^b \int_{C_u(t)} \frac{\tau_+^2}{\tau_-^{\frac{1}{2}}} \int_v \left| z^a \nabla_{t,x} \widehat{Z}^{\beta} f \right| dv \int_v \frac{v^{\underline{L}}}{v^0} \left| z^{a+2} \nabla_{t,x} \widehat{Z}^{\beta} f \right| dv d C_u(t) du. $$
Now notice that by assumption, $\beta_P \leq \kappa_P$, so that $a \leq \frac{1}{2}(N+9-(1+2\eta)\beta_P)$. Using Remark \ref{Rq32}, we get according to the bootstrap assumption \eqref{bootf1} as well as \eqref{decayf},
$$ \sup_{u<b} \int_{C_u(t)}  \int_v \frac{v^{\underline{L}}}{v^0} \left| z^{a+2} \nabla_{t,x} \widehat{Z}^{\beta} f \right| dv d C_u(t) \lesssim \E_{N-3}^{N+13,\eta}[f](t) \lesssim \epsilon (1+t)^{\delta}, \qquad \int_v \left| z^a \nabla_{t,x} \widehat{Z}^{\beta} f \right| dv \lesssim \epsilon \frac{(1+t)^{\delta}}{\tau_+^2\tau_-}.
$$ 
We then get
$$  \mathfrak{I} \hspace{1mm} \lesssim \hspace{1mm} \int_{u=-\infty}^b \left\| \frac{\tau_+^2}{\tau_-^{\frac{1}{2}}} \cdot \epsilon \frac{(1+t)^{\delta}}{\tau_+^2\tau_-} \right\|_{L^{\infty}(C_u(t)} \cdot \epsilon (1+t)^{\delta} du \hspace{1mm} \lesssim \hspace{1mm} \epsilon^2 (1+t)^{2\delta} \int_{u=-\infty}^b \frac{du}{\tau_-^{\frac{3}{2}}} \hspace{1mm} \lesssim \hspace{1mm} \epsilon^2 (1+t)^{2\delta},$$
which implies $I^a_{\alpha}+I^a_{\alpha}+I^a_{\alpha}+I^a_{\alpha} \lesssim \epsilon^{\frac{3}{2}} (1+t)^{\delta}$. For the remaining terms, we obtain similarly that
\begin{align*}
 J^a_{\alpha}+J^a_{\rho}+J^a_{\sigma}+J^a_{\underline{\alpha}} \hspace{1mm} & \lesssim \hspace{1mm} \left| \int_0^t \frac{\left\| \mathcal{L}_{Z^{\gamma}}(F) \right\|_{L^2(\Si^b_s)}^2}{(1+s)^{\frac{3}{2}}}ds \int_0^t \int_{\Si^b_s} (1+s)^{\frac{3}{2}}  \left| \int_v \left| z^a \widehat{\Gamma} \widehat{Z}^{\beta} f \right| \frac{dv}{v^0} \right|^2 dx ds \right|^{\frac{1}{2}} \\
 & \lesssim \hspace{1mm} \sqrt{\epsilon} \left| \int_{u=-\infty}^b \int_{C_u(t)} \frac{\tau_+^2}{\tau_-^{\frac{1}{2}}} \int_v \left| z^{a-1} \widehat{\Gamma} \widehat{Z}^{\beta} f \right| dv \int_v \frac{v^{\underline{L}}}{v^0} \left| z^{a+1} \widehat{\Gamma}  \widehat{Z}^{\beta} f \right| dv d C_u(t) du \right|^{\frac{1}{2}}.
 \end{align*}
Since $a-1 \leq \frac{1}{2}(N+9-(1+2\eta)(\beta_P+1))$, the bootstrap assumption \eqref{bootf1} and \eqref{decayf} give
$$ \sup_{u<b} \int_{C_u(t)}  \int_v \frac{v^{\underline{L}}}{v^0} \left| z^{a+1} \widehat{\Gamma} \widehat{Z}^{\beta} f \right| dv d C_u(t) \lesssim \E_{N-3}^{N+13,\eta}[f](t) \lesssim \epsilon (1+t)^{\delta}, \qquad \int_v \left| z^{a-1} \widehat{\Gamma} \widehat{Z}^{\beta} f \right| dv \lesssim \epsilon \frac{(1+t)^{\delta}}{\tau_+^2\tau_-},
$$ 
so that $J^a_{\alpha}+J^a_{\rho}+J^a_{\sigma}+J^a_{\underline{\alpha}}  \lesssim \epsilon^{\frac{3}{2}}(1+t)^{\delta}$.
\subsection{The remaining energy norm}

For the improvement of $\E^{N+13,\eta}_{N-3}[f] \leq 4 \epsilon (1+t)^{\delta}$ we have, in view of \eqref{eq:forestimates} as well as Propositions \ref{energyf} and \ref{Maxcom}, to prove similar estimates than \eqref{weight} and those of Proposition \ref{estif}. More precisely, $\E^{N+13,\eta}_{N-3}[f] \leq 3 \epsilon (1+t)^{\delta}$ on $[0,T[$ ensues, for $\epsilon$ small enough, from the following proposition.
\begin{Pro}
Let $|\kappa| \leq N-3$, $\gamma$ and $\beta$ be such that $|\gamma|+|\beta| \leq |\kappa|$, $|\beta| \leq |\kappa|-1$ and $\beta_P+[\gamma] \leq \kappa_P$. Then,
\begin{eqnarray}
\nonumber \int_0^t \int_{\Si^b_s} \int_v z^{\frac{N+13}{2}-(\frac{1}{2}+\eta)\kappa_P-1} \left| F \left( v,\nabla_v z \right) \widehat{Z}^{\kappa} f  \right| \frac{dv}{v^0} dx ds & \lesssim & \epsilon^{\frac{3}{2}} (1+t)^{\delta}. \\ \nonumber
 \int_0^t \int_{\Si^b_s} \int_v \left| z^{\frac{N+13}{2}-(\frac{1}{2}+\eta)\kappa_P} \mathcal{L}_{Z^{\gamma}}(F) \left( v,\nabla_v \widehat{Z}^{\beta} f \right) \right| \frac{dv}{v^0} dx ds & \lesssim & \epsilon^{\frac{3}{2}} (1+t)^{\delta}.
 \end{eqnarray}
\end{Pro}
\begin{proof}
One only has to follow Subsections \ref{secweight} and, as $|\gamma| \leq N-3$, \ref{secFpoint} and to use the bootstrap assumption \eqref{bootf1} instead of \eqref{bootf2}.
\end{proof}

\subsection{$L^2$ estimates on velocity averages}

The following result will allow us to improve our estimate on $\mathcal{E}_N[F]$.

\begin{Pro}\label{L2estimates}
We have, for all $t \in [0,T[$,
$$\sum_{|\beta| \leq N} \left\|  \int_v z |\widehat{Z}^{\beta} f| dv \right\|_{L^2(\Si^b_t)} \hspace{2mm} \lesssim \hspace{2mm} \frac{\epsilon}{(1+t)^{\frac{3}{4}}}.$$
\end{Pro}
\begin{proof}
We start by considering $|\beta| \leq N-3$. Using successively the Cauchy-Schwarz inequality in $v$, the pointwise decay estimate \eqref{decayf2} and the bootstrap assumption \eqref{bootf2}, we get
\begin{eqnarray}
\nonumber \left\|  \int_v z |\widehat{Z}^{\beta} f| dv \right\|^2_{L^2(\Si^b_t)} \hspace{-1.1mm} & \lesssim & \hspace{-0.5mm} \left\|  \int_v  |\widehat{Z}^{\beta} f| dv \int_v  z^2 |\widehat{Z}^{\beta} f| dv \right\|_{L^1(\Si^b_t)} \hspace{1.4mm} \lesssim \hspace{1.4mm} \left\|  \int_v |\widehat{Z}^{\beta} f| dv \right\|_{L^{\infty}(\Si^b_t)} \left\|  \int_v z^2 |\widehat{Z}^{\beta} f| dv \right\|_{L^1(\Si^b_t)} \\ \nonumber 
& \lesssim & \hspace{-0.5mm} \left\|  \frac{\epsilon}{\tau_+^{2-\delta}\tau_-} \right\|_{L^{\infty}(\Si^b_t)}  \E^{N+9,\eta}_{N}[f](t) \hspace{1.4mm} \lesssim \hspace{1.4mm} \frac{\epsilon^2}{(1+t)^{2-2\delta}} \hspace{1.4mm} \lesssim \hspace{1.4mm} \frac{\epsilon^2}{(1+t)^{\frac{3}{2}}} .
\end{eqnarray}
The cases $N-2 \leq |\beta| \leq N$ are the purpose of Section \ref{sec9}.
\end{proof}

\subsection{Improvements under stronger assumptions on the initial data}\label{subsecuniformbound}

The only terms preventing us to take $\delta=0$ are those which are estimated by $\int_0^t \frac{\E^{N+9,\eta}_{N}[f](s)}{1+s}ds$, i.e. \eqref{weight}, $I^a_{\zeta}$ when $|\gamma| \leq N-3$ and $J^a_{\zeta}+J^a_{\underline{\alpha}}$ when $\gamma_T =0$ and $|\gamma| \leq N-3$. One can check that $J^a_{\zeta}+J^a_{\underline{\alpha}}$ and the term containing $\underline{\alpha}$ in \eqref{weight} could be bounded by $\epsilon^{\frac{3}{2}}$. For the other ones, which only contain the null components\footnote{Note that we could deal with the terms containing $\rho$ or $\sigma$ as a factor by using \eqref{calculGintro} instead of Lemma \ref{calculF}.} $\alpha$, $\rho$ and $\sigma$, the problem comes from a small lack of decay in $t+r$. We present two different ways for solving this issue. Assuming one of the following two stronger decay hypotheses 
$$ \int_{\Si^b_0} r^{2\delta} |\mathcal{L}_{Z^{\gamma}}(F)|^2 dx \hspace{1mm} \leq \hspace{1mm} \epsilon \qquad \text{or} \qquad \E^{N+13+2\delta,\eta}_{N+3}[|v^0|^{4\delta}f](0) \hspace{1mm} \leq \hspace{1mm} \epsilon,$$
we could prove that $\E^{N+9,\eta}_{N}[f](t) \leq 3\epsilon$ for all $t \in [0,T[$. Indeed, in the first case we could\footnote{The proof would be similar to the one of Proposition \ref{decayrhoalter}.} prove using \eqref{nullrho}-\eqref{nullalpha} that $|\alpha|+|\rho|+|\sigma| \lesssim \sqrt{\epsilon}\tau_+^{-\frac{3}{2}-\delta}$. In the second case, we could obtain that $\E^{N+9+2\delta,\eta}_{N}[|v^0|^{4\delta}f](t) \lesssim \epsilon (1+t)^{\delta}$ and then use this bound combined with the inequality $1 \lesssim |v^0 v^{\underline{L}}|^{2 \delta} \lesssim |v^0|^{4 \delta}\tau_+^{-2\delta}z^{2\delta}$ in order to get $I^a_{\zeta} \lesssim \epsilon^{\frac{3}{2}}$.

\section{The energy bound on the electromagnetic field}\label{sec8}

According to the energy estimate of Proposition \ref{energyMax1}, the commutation formula of Proposition \ref{Maxcom} and since $\mathcal{E}_N[F](0) \leq \epsilon$, we would obtain $\mathcal{E}_N[F] \leq 3 \epsilon$ on $[0,T[$ for $\epsilon$ small enough if we could prove
$$\sum_{|\gamma| \leq N} \sum_{|\beta| \leq N} \int_0^t \int_{\Si^b_s} \left| \mathcal{L}_{Z^{\gamma}}(F)_{0 \nu} \int_v \frac{v^{\nu}}{v^0} \widehat{Z}^{\beta} f dv \right| dx ds  \hspace{2mm} \lesssim \hspace{2mm} \epsilon^{\frac{3}{2}} .$$
We then fix $|\beta| \leq N$, $|\gamma| \leq N$ and we denote by $(\alpha, \underline{\alpha}, \rho, \sigma)$ the null decomposition of $\mathcal{L}_{Z^{\gamma}}(F)$. Expanding $\mathcal{L}_{Z^{\gamma}}(F)_{0 \nu} \int_v \frac{v^{\nu}}{v^0} \widehat{Z}^{\beta} f $ in null coordinates, we can observe that it suffices to prove that
$$ I_{\rho} := \int_0^t \int_{\Si^b_s} |\rho| \int_v  \left| \widehat{Z}^{\beta} f \right| dv dx ds \lesssim \epsilon^{\frac{3}{2}} \hspace{8mm} \text{and} \hspace{0.8cm} I_{\alpha, \underline{\alpha}} := \int_0^t \int_{\Si^b_s} \left( |\alpha|+|\underline{\alpha}| \right) \int_v \frac{|v^{A}|}{v^0} \left| \widehat{Z}^{\beta} f \right| dv dx ds \lesssim \epsilon^{\frac{3}{2}}. $$
Using succesively the Cauchy-Schwarz inequality in $(s,x)$, the inequality $\tau_+ |v^A| \lesssim v^0 z$ which comes from Lemma \ref{weights1} and then the bootstrap assumption \eqref{bootF1} as well as Proposition \ref{L2estimates}, we have
\begin{eqnarray}
\nonumber \left| I_{\alpha,\underline{\alpha}} \right|^2 & \lesssim & \int_{0}^t  \frac{\left\| |\alpha|+ |\underline{\alpha}| \right\|^2_{L^2(\Si^b_s)}}{(1+s)^2} ds \int_0^t (1+s)^2 \left\|  \int_v \frac{|v^A|}{v^0} \left| \widehat{Z}^{\beta} f \right| dv \right\|_{L^2(\Si^b_s)}^2 ds \\ \nonumber
& \lesssim & \int_{0}^t \frac{\mathcal{E}_N[F](s)}{(1+s)^2} ds  \int_0^t  \left\|  \int_v z \left| \widehat{Z}^{\beta} f \right| dv \right\|_{L^2(\Si^b_s)}^2 ds \hspace{2mm} \lesssim \hspace{2mm} \int_0^{+ \infty} \frac{\epsilon}{(1+s)^2} ds \int_0^{+\infty} \frac{\epsilon^2}{(1+s)^{\frac{3}{2}}}ds \hspace{2mm} \lesssim \hspace{2mm} \epsilon^3.
\end{eqnarray}
Similarly, using $\tau_- \lesssim z$ instead of $\tau_+ |v^A| \lesssim v^0 z$ (see also Lemma \ref{weights1}) and Lemma \ref{foliationexpli}, one gets
\begin{eqnarray}
\nonumber \left| I_{\rho} \right|^2  & \lesssim & \int_{0}^t \int_{\Si^b_s}  \frac{|\rho|^2}{\tau_-^2} dx ds  \int_0^t \int_{\Si^b_s} \tau_-^2 \left|  \int_v  \left| \widehat{Z}^{\beta} f \right| dv \right|^2 dx ds \\ \nonumber
& \lesssim & \int_{u=-\infty}^b  \int_{C_u(t)} \frac{|\rho|^2}{\tau_-^2} d C_u(t) du \int_0^t \int_{\Si^b_s}  \left|  \int_v z \left| \widehat{Z}^{\beta} f \right| dv \right|^2 dx ds  \\ \nonumber
& \lesssim &  \int_{u=-\infty}^b \frac{\mathcal{E}_N[F](t)}{\tau_-^2} du \int_0^t \left\|  \int_v z \left| \widehat{Z}^{\beta} f \right| dv \right\|_{L^2(\Si^b_s)}^2 ds  \hspace{2mm} \lesssim \hspace{2mm} \int_{u=-\infty}^0 \frac{\epsilon}{\tau_-^2} du \int_0^{+\infty} \frac{\epsilon^2}{(1+s)^{\frac{3}{2}}}ds \hspace{2mm} \lesssim \hspace{2mm} \epsilon^3.
\end{eqnarray}
This concludes the improvement of the bootstrap assumption \eqref{bootF1}.

\section{$L^2$ estimates for the higher order derivatives of the Vlasov field}\label{sec9}

In this last section, we complete the proof of Proposition \ref{L2estimates}. For this purpose, we follow the strategy used in Section $4.5.7$ of \cite{FJS}. The first step of the proof consists in rewriting all transport equations as a hierarchised system. Let $I$ and $\I$ be the following two ordered sets,
\begin{eqnarray}
\nonumber I & := & \{ \beta \hspace{2mm} \text{multi-index} \hspace{1mm} / \hspace{1mm} N-2 \leq |\beta| \leq N \}  \hspace{2mm} = \hspace{2mm} \{ \beta^{1},...,\beta^{|I|} \},  \\ \nonumber
 \mathfrak{I} & := & \{ \xi \hspace{2mm} \text{multi-index} \hspace{1mm} / \hspace{1mm}  |\xi| \leq N-3 \} \hspace{2mm} = \hspace{2mm} \{ \xi^1,...,\xi^{|\I|} \} .
\end{eqnarray}

We also consider two vector valued fields $R$ and $W$ of respective length $|I|$ and $|\I|$ such that
$$ R_i= \widehat{Z}^{\beta^i}f \hspace{10mm} \text{and} \hspace{10mm} W_i = \widehat{Z}^{\xi^i}f.$$
We denote by $\Vv$ the vector space of functions $\{ h \hspace{1mm} / \hspace{1mm} h : V_b(T) \times \R^3_v \rightarrow \R \}$ and we recall that $[\gamma]:= \max(0,1- \gamma_T)$. Let us now rewrite the Vlasov equation satisfied by the vector $R$.
\begin{Lem}\label{L2bilan}
There exists three matrices valued functions $A : V_b(T) \times \R^3_v \rightarrow \mathfrak M_{|I|}(\Vv)$, $B : V_b(T)  \times \R^3_v \rightarrow  \mathfrak M_{|I|,|\I|}(\Vv)$ and $D : V_b(T) \times \R^3_v \rightarrow \mathfrak M_{|\I|}(\Vv)$ such that
$$\TT_F(R)+AR=B W, \qquad \qquad \TT_F(W)=DW  .$$
The matrices $A$ and $B$ are such that $\TT_F(R_i)$, for $1 \leq i \leq |I|$, is a linear combination of  the following terms, where $(\theta, \nu) \in \llbracket 0,3 \rrbracket^2$,
\begin{flalign*}
& \hspace{1.2cm} \frac{v^{\mu}}{v^0}\mathcal{L}_{Z^{\gamma}}(F)_{\mu \nu}  R_j , \hspace{8.5mm} \text{with} \hspace{5mm} |\gamma| \leq N-6 \hspace{8mm} \text{and} \hspace{10.5mm} \beta^j_P + [\gamma]  \leq  \beta^i_P+1, & \\
& \hspace{1.2cm} x^{\theta} \frac{v^{\mu}}{v^0}\mathcal{L}_{Z^{\gamma}}(F)_{\mu \nu}  R_j , \hspace{5.5mm} \text{with} \hspace{5mm} |\gamma| \leq N-6 \hspace{8mm} \text{and} \hspace{10.5mm} \beta^j_P + [\gamma]  \leq  \beta^i_P, & \\
& \hspace{1.2cm} \frac{v^{\mu}}{v^0}\mathcal{L}_{Z^{\gamma}}(F)_{\mu \nu}  W_q , \hspace{8.5mm} \text{with} \hspace{5mm} |\gamma| \leq N \hspace{8mm} \text{and} \hspace{10.5mm} \xi^q_P + [\gamma]  \leq  \beta^i_P+1, & \\
& \hspace{1.2cm} x^{\theta} \frac{v^{\mu}}{v^0}\mathcal{L}_{Z^{\gamma}}(F)_{\mu \nu}  W_q , \hspace{5.5mm} \text{with} \hspace{5mm} |\gamma| \leq N \hspace{8mm} \text{and} \hspace{10.5mm} \xi^q_P + [\gamma]  \leq  \beta^i_P. & \\
\end{flalign*}
Similarly, $D$ is such that $\TT_F(W_i)$, for $1 \leq i \leq |\I|$, can be written as a linear combination of  the following terms, where $(\theta, \nu) \in \llbracket 0,3 \rrbracket^2$,
\begin{flalign*}
& \hspace{1.2cm} \frac{v^{\mu}}{v^0}\mathcal{L}_{Z^{\gamma}}(F)_{\mu \nu}  W_q , \hspace{8.5mm} \text{with} \hspace{5mm} |\gamma| \leq N-3 \hspace{8mm} \text{and} \hspace{10.5mm} \xi^q_P + [\gamma]  \leq  \beta^i_P+1, & \\
& \hspace{1.2cm} x^{\theta} \frac{v^{\mu}}{v^0}\mathcal{L}_{Z^{\gamma}}(F)_{\mu \nu}  W_q , \hspace{5.5mm} \text{with} \hspace{5mm} |\gamma| \leq N-3 \hspace{8mm} \text{and} \hspace{10.5mm} \xi^q_P + [\gamma]  \leq  \beta^i_P. & \\
\end{flalign*}
\end{Lem}
\begin{Rq}
Note that pointwise decay estimates can be obtained on the derivatives of the matrix $A$ up to order $3$. This directly follows from the description of the components of $A$ and from Proposition \ref{decayMax}.
\end{Rq}
\begin{proof}
Applying the commutation formula of Proposition \ref{Maxcom} to $R_i=\widehat{Z}^{\beta^i} f$, we know that $\TT_F(R_i)$ can be written as a linear combination of terms of the form
$$ \mathcal{L}_{Z^{\gamma}}(F)(v,\nabla_v \widehat{Z}^{\beta} f), \qquad |\gamma| + |\beta| \leq |\beta^i|, \quad |\beta| \leq |\beta^i|-1, \quad \beta_P + [\gamma]  \leq  \beta^i_P.$$
Writing $v^0\partial_{v^i}=\widehat{\Omega}_{0i}-t\partial_i-x^i\partial_t$, we obtain
\begin{equation}\label{L2decomp}
\frac{v^{\mu}}{v^0}{\mathcal{L}_{Z^{\gamma}}(F)_{\mu}}^i \widehat{\Omega}_{0i}\widehat{Z}^{\beta} f- t\frac{v^{\mu}}{v^0}{\mathcal{L}_{Z^{\gamma}}(F)_{\mu}}^i \partial_i \widehat{Z}^{\beta} f-x_i\frac{v^{\mu}}{v^0}{\mathcal{L}_{Z^{\gamma}}(F)_{\mu}}^i \partial_t\widehat{Z}^{\beta} f.
\end{equation}
In order to lighten the notations, we fix now $1 \leq i \leq 3$ and $0 \leq \lambda \leq 3$.
\begin{itemize}
\item If $N-3 \leq |\beta| \leq N-1$, there exists $\beta^j, \beta^k \in I$ such that $\widehat{\Omega}_{0i}\widehat{Z}^{\beta} f = R_j$ and $\partial_{\lambda}\widehat{Z}^{\beta} f = R_k$. Since $\beta^j_P=\beta_P+1$ and $\beta^k_P=\beta_P$, we obtain $\beta^j_P + [\gamma]  \leq  \beta^i_P+1$ and $\beta^k_P + [\gamma]  \leq  \beta^i_P$. Moreover, $|\gamma| + |\beta| \leq |\beta^i| \leq N$ implies in that case that $|\gamma| \leq 3 \leq N-6$, so that all the terms in \eqref{L2decomp} have the requested form.
\item Otherwise, $|\beta| \leq N-4$ and there exists $\xi^q, \xi^k \in I$ such that $\widehat{\Omega}_{0i}\widehat{Z}^{\beta} f = W_q$ and $\partial_{\lambda}\widehat{Z}^{\beta} f = W_k$. As before, we obtain $\xi^q_P + [\gamma]  \leq  \beta^i_P+1$ and $\xi^k_P + [\gamma]  \leq  \beta^i_P$. Since $|\gamma|  \leq |\beta^i| \leq N$, all the terms in \eqref{L2decomp} have also the requested form in that case.
\end{itemize} 
The construction of the matrix $D$ is easier and follows from an application of Proposition \ref{Maxcom} to $W_i=\widehat{Z}^{\xi^i}f$. Since $|\xi^i| \leq N-3$, all the error terms obtained contained derivatives of the electromagnetic field of order $|\gamma| \leq N-3$.
\end{proof}

In order to establish an $L^2$ estimate on the velocity average of $R$, we split it in $R:=H+G$, where
$$\left\{
    \begin{array}{ll}
         \TT_F(H)+AH=0 \hspace{1mm}, \hspace{6mm} H(0,\cdot,\cdot)=R(0,\cdot,\cdot),\\
        \TT_F(G)+AG=BW \hspace{1mm}, \hspace{2mm} G(0,\cdot,\cdot)=0
    \end{array}
\right.$$
and then prove $L^2$ estimates on $\int_v |H|dv$ and $\int_v |G| dv$. For the homogeneous part $H$, we will commute the transport equation and take advantage of the decaying properties of the matrix $A$ in order to obtain boundedness on a certain $L^1$ norm as for $f$ in Section \ref{sec7}. The $L^2$ estimate will then follow from a Klainerman-Sobolev inequality and the bound obtained on $\E[H]$. The inhomogeneous part will be schematically decomposed as $G=KW$, with $K$ a matrix such that $\E[|K|^2 |W|](t) \leq \epsilon (1+t)^{\frac{1}{4}}$. The expected decay rate on $\| \int_v |G| dv \|_{L^2_x}$ will then be obtained using the pointwise decay estimates satisfied by the components of $W$.

\subsection{The homogeneous system}\label{subsecH}

With the aim of obtaining an $L^{\infty}$ estimate on $\int_v |H| dv$, we will have to commute at least three times the transport equation satisfied by each component of $H$. In order to take advantage of similar hierarchies as those considered in Subsection \ref{sec7}, we introduce the following energy norm
$$ \E_H(t) := \sum_{1 \leq i \leq |I|}  \sum_{ |\beta| \leq 3}  \hspace{0.5mm} \E \left[ \sqrt{z}^{N+5-(1+2\eta)(\beta_P+\beta^i_P)}\widehat{Z}^{\beta} H_i \right] (t).$$
We start with a technical result similar to Lemma \ref{comumax1}.
\begin{Lem}\label{petitrsult}
Let $G$ be a $2$-form defined on $V_b(T)$, $g:V_b(T) \times \R^3_v \rightarrow \R$ be a sufficiently regular function, $(\theta, \nu) \in \llbracket 0, 3 \rrbracket$ and $Z \in \mathbb{K} $. Then\footnote{Recall that if $Z$ is a Killing vector field, then $\widehat{Z}$ denotes its complete lift and if $Z=S$ then $\widehat{Z}=S$.}, $\widehat{Z} \left( \frac{v^{\mu}}{v^0} G_{\mu \nu} g\right)$ can be written as a linear combination of 
$$ \frac{v^{\mu}}{v^0} \mathcal{L}_Z(G)_{\mu \nu} g, \qquad \frac{v^{\mu}}{v^0} G_{\mu \nu} \widehat{Z}(g), \qquad \frac{v^{\kappa}}{v^0}\frac{v^{\mu}}{v^0} G_{\mu \xi} g, \quad 0 \leq \kappa, \xi \leq 3.$$
Similarly, $\widehat{Z} \left(x^{\theta} \frac{v^{\mu}}{v^0} G_{\mu \nu} g\right)$ can be written as a linear combination of
$$ x^{\theta}\frac{v^{\mu}}{v^0} \mathcal{L}_Z(G)_{\mu \nu} g, \qquad x^{\theta}\frac{v^{\mu}}{v^0} G_{\mu \nu} \widehat{Z}(g), \qquad \frac{v^{\mu}}{v^0} G_{\mu \nu} g, \qquad x^{\lambda} \frac{v^{\kappa}}{v^0}\frac{v^{\mu}}{v^0} G_{\mu \xi} g, \quad 0 \leq \kappa, \lambda, \xi \leq 3.$$
\end{Lem}
\begin{proof}
The second relation can be obtained from the first one by Leibniz formula as $\widehat{Z}(x^{\theta})=0$ or $\widehat{Z}(x^{\theta})=x^{\lambda}$ if $\widehat{Z}$ is an homogeneous vector field and $\partial_{\mu}(x^{\theta})=\delta_{\mu}^{\theta}$. Note now that by Leibniz formula,
$$ \widehat{Z} \left( \frac{v^{\mu}}{v^0} G_{\mu \nu} g\right) = \widehat{Z} \left( \frac{1}{v^0} \right) v^{\mu} G_{\mu \nu} g + \frac{v^{\mu}}{v^0} G_{\mu \nu} \widehat{Z}(g)+\frac{1}{v^0} \widehat{Z} \left( v^{\mu} G_{\mu \nu} \right) g .$$
As $\widehat{Z} \left( \frac{1}{v^0}\right)=-\frac{v^i}{|v^0|^2}$ if $\widehat{Z}=\widehat{\Omega}_{0i}$  and $\widehat{Z} \left( \frac{1}{v^0} \right) =0$ otherwise, the first two terms on the right hand side of the previous equality have the requested form. In order to study the last one, we introduce $Z_v:=\widehat{Z}-Z$ and we remark that $v^{\mu} G_{\mu \nu}=G(v,\partial_{\nu})$. Hence,
$$ \widehat{Z} \left( v^{\mu} G_{\mu \nu} \right) = \mathcal{L}_Z(G)(v,\partial_{\nu})+G(v,[Z,\partial_{\nu}])+G([Z,v],\partial_{\nu})+G(Z_v(v),\partial_{\nu}).$$
To conclude the proof, it remains to notice that $[Z,\partial_{\nu}] \in \{0\} \cup \{ \pm \partial_{\mu} \hspace{1mm} / \hspace{1mm} 0 \leq \mu \leq 3 \}$ and $[Z,v]+Z_v(v)=-\delta_{Z}^Sv$.
\end{proof}
We have the following commutation formula.
\begin{Lem}\label{comL2hom}
Let $|\beta| \leq 3$ and $\beta^i \in I$. Then, $\TT_F( \widehat{Z}^{\beta} H_i)$ can be written as a linear combination with polynomial coefficients in $\{ \frac{v^{\lambda}}{v^0} \hspace{1mm} / \hspace{1mm} 0 \leq \lambda \leq 3 \}$ of terms of the following two families, where $(\theta, \nu) \in \llbracket 0, 3 \rrbracket$ and $|\gamma| \leq N-3$.
\begin{align*}
& \frac{v^{\mu}}{v^0}\mathcal{L}_{Z^{\gamma}}(F)_{\mu \nu}  \widehat{Z}^{\kappa} H_j, \qquad \beta^j_P+\kappa_P +[\gamma] \leq \beta^i_P+\beta_P+1, \\
&  x^{\theta}\frac{v^{\mu}}{v^0}\mathcal{L}_{Z^{\gamma}}(F)_{\mu \nu}  \widehat{Z}^{\kappa} H_j, \qquad \beta^j_P+\kappa_P+ [\gamma] \leq \beta^i_P+\beta_P .
\end{align*}
$\TT_F( \widehat{Z}^{\beta} H_i)$ can then be bounded by a linear combination of terms of the form
$$  \left( |\alpha (\mathcal{L}_{Z^{\gamma}}(F))|+|\rho (\mathcal{L}_{Z^{\gamma}}(F))|+|\sigma (\mathcal{L}_{Z^{\gamma}}(F))|+\frac{|v^A|+v^{\underline{L}}}{v^0}|\underline{\alpha} (\mathcal{L}_{Z^{\gamma}}(F))| \right) |  \widehat{Z}^{\kappa} H_j|,  \qquad |\gamma| \leq N-3, $$
where
\begin{itemize}
\item either $\beta^j_P+\kappa_P \leq \beta^i_P+\beta_P$ 
\item or $\beta^j_P+\kappa_P = \beta^i_P+\beta_P+1$ and $\gamma_T \geq 1$.
\end{itemize}
$$  \tau_+\left( |\alpha (\mathcal{L}_{Z^{\gamma}}(F))|+|\rho (\mathcal{L}_{Z^{\gamma}}(F))|+|\sigma (\mathcal{L}_{Z^{\gamma}}(F))|+\frac{|v^A|+v^{\underline{L}}}{v^0}|\underline{\alpha} (\mathcal{L}_{Z^{\gamma}}(F))| \right) |  \widehat{Z}^{\kappa} H_j|,  \qquad |\gamma| \leq N-3, $$
where
\begin{itemize}
\item either $\beta^j_P+\kappa_P < \beta^i_P+\beta_P$ 
\item or $\beta^j_P+\kappa_P = \beta^i_P+\beta_P$ and $\gamma_T \geq 1$.
\end{itemize}
\end{Lem}
\begin{proof}
The second part of the Lemma follows directly from the first part and
\begin{equation}\label{eq:calculmachin}
\left| \frac{v^{\mu}}{v^0}\mathcal{L}_{Z^{\gamma}}(F)_{\mu \nu} \right| \lesssim \left( |\alpha (\mathcal{L}_{Z^{\gamma}}(F))|+|\rho (\mathcal{L}_{Z^{\gamma}}(F))|+|\sigma (\mathcal{L}_{Z^{\gamma}}(F))|+\frac{|v^A|+v^{\underline{L}}}{v^0}|\underline{\alpha} (\mathcal{L}_{Z^{\gamma}}(F))| \right) ,
\end{equation}
which ensues from
$$\frac{v^{\mu}}{v^0}\mathcal{L}_{Z^{\gamma}}(F)_{\mu \nu} =  \frac{v^{L}}{v^0}\mathcal{L}_{Z^{\gamma}}(F)_{L \nu}+\frac{v^{\underline{L}}}{v^0}\mathcal{L}_{Z^{\gamma}}(F)_{\underline{L} \nu}+\frac{v^{A}}{v^0}\mathcal{L}_{Z^{\gamma}}(F)_{A \nu}, \quad \left| \mathcal{L}_{Z^{\gamma}}(F)_{L \nu} \right| \lesssim |\alpha ( \mathcal{L}_{Z^{\gamma}}(F))|+|\rho ( \mathcal{L}_{Z^{\gamma}}(F))| .$$
It then remains us to prove the first assertion. The starting point is the relation $$ \TT( \widehat{Z}^{\beta} H_i) \hspace{1mm} = \hspace{1mm} [\TT_F,\widehat{Z}^{\beta}](H_i)+\widehat{Z}^{\beta} \left( \TT_F \left(H_i \right) \right).$$ 
Applying the commutation formula of Proposition \ref{Maxcom}, $[\TT_F,\widehat{Z}^{\beta}](H_i)$ gives, using $v^0\partial_{v^i}=\widehat{\Omega}_{0i}-t\partial_i-x^i\partial_t$ as in \eqref{L2decomp}, terms described in this lemma with $j=i$ and $|\gamma| \leq 3 \leq N-3$. The other ones arise from $\widehat{Z}^{\beta} \left( \TT_F \left(H_i \right) \right)$. Using Lemma \ref{L2bilan}, which describes the source terms of $\TT_F \left(H_i \right)$, and applying $|\beta|$ times Lemma \ref{petitrsult}, we obtain terms described in the Lemma with $\kappa_P \leq \beta_P$ and $|\gamma| \leq N-6+|\beta| \leq N-3$.
\end{proof}
Hence, as $R(0,.,.)=H(0,.,.)$, there exists $C_0>0$ such that $\E_H(0) \leq C_0 \epsilon$. Following the proof of \eqref{weight} and Proposition \ref{estif} (for the cases where $|\gamma| \leq N-3$), one can prove, if $\epsilon$ small enough, that $\E_H (t) \leq 3C_0 \epsilon (1+t)^{\delta}$ for all $t \in [0,T[$. By Proposition \ref{KS1}, we then obtain
\begin{equation}\label{decayH}
\hspace{-1mm} \forall \hspace{0.5mm} (t,x) \in V_b(T), \hspace{3mm} 1 \leq j \leq |I|,  \hspace{9mm} \int_v | H_j| dv \hspace{1mm} \lesssim \hspace{1mm} \int_v \sqrt{z}^{N+5-(1+2\eta)(\beta^j_P+3)}| H_j| dv \hspace{1mm} \lesssim \hspace{1mm} \epsilon \frac{(1+t)^{\delta}}{\tau_+^2 \tau_-}.
 \end{equation}

\subsection{The inhomogenous system}\label{subsecG}

The purpose of this subsection is to prove an $L^2$ estimate on $\int_v |G| dv$. We cannot proceed by commuting $\TT_F(G)+AG=BW$ since $B$ contains top order derivatives of $F$ and we then follow the strategy exposed earlier in this section. It will be convenient, in order to take advantage of the hierarchies used in Section \ref{sec7}, to work with a slightly different vector than $G$
\begin{Def}
Let $L$ and $Y$ be two vector valued fields of respective length $|I|$ and $|\mathfrak{I}|$ such that, for $i \in \llbracket 1, |I| \rrbracket$ and $k \in \llbracket 1, |\mathfrak{I}| \rrbracket$,
$$L_i = \sqrt{z}^{N+1-(1+2\eta)\beta^i_P} G_i , \qquad Y_k = \sqrt{z}^{N+3-(1+2\eta)\xi^k_P} W_k.$$
\end{Def}
The aim of the next lemma is to describe in details the transport equation satisfied by $L$.
\begin{Lem}\label{bilaninho}
There exist three matrices valued functions $\overline{A} : V_b(T) \times \R^3_v  \rightarrow \mathfrak M_{|I|}(\mathbb{V})$, $\overline{B} : V_b(T) \times \R^3_v  \rightarrow \mathfrak M_{|I|,|\mathfrak{I}|}(\mathbb{V})$ and $\overline{D} : V_b(T) \times \R^3_v \rightarrow \mathfrak M_{|\mathfrak{I}|}(\mathbb{V})$ such that
$$\TT_F(L)+\overline{A}L= \overline{B} Y, \hspace{10mm} \TT_F(Y)= \overline{D} Y \hspace{8mm} \text{and} \hspace{8mm} \forall \hspace{0.5mm} (t,x) \in V_b(T), \hspace{5mm}  \int_v  z^2|Y|(t,x,v) dv \lesssim \epsilon \frac{(1+t)^{\delta}}{\tau_+^2 \tau_-}.$$
The matrices $\overline{A}$ and $\overline{B}$ are such that $\TT_F(L_i)$ can be bounded, for $1 \leq i \leq |I|$, by a linear combination of the following terms,
\begin{eqnarray}
\nonumber & \bullet &  \frac{\tau_+}{\tau_-}\left( \left| \alpha \left( \mathcal{L}_{Z^{\gamma}}(F) \right) \right|+ \left| \rho \left( \mathcal{L}_{Z^{\gamma}}(F) \right) \right|+\left| \sigma \left( \mathcal{L}_{Z^{\gamma}}(F) \right) \right| \right)|Y|, \qquad  \left| \underline{\alpha} \left( \mathcal{L}_{Z^{\gamma}}(F) \right) \right|  |Y|, \hspace{1cm} \text{with} \hspace{3mm} |\gamma| \leq N. \\ \nonumber
& \bullet & \sqrt{\epsilon}\frac{v^0}{\tau_+}|L|, \hspace{1cm} \sqrt{\epsilon}\frac{v^{\underline{L}}}{\tau_-^{1+\eta}}|L| .
\end{eqnarray}
Similarly, $\overline{D}$ is such that
$$\forall \hspace{0.5mm} 1 \leq i \leq |\mathfrak{I}|, \qquad \left| \TT_F(Y_i) \right| \hspace{1mm} \lesssim \hspace{1mm} \sqrt{\epsilon}\frac{v^0}{\tau_+}|Y|+ \sqrt{\epsilon}\frac{v^{\underline{L}}}{\tau_-^{1+\eta}}|Y| . $$
\end{Lem}
\begin{proof}
Remark first that since $Y_j=\sqrt{z}^{N+3-(1+2\eta)\xi^j_P}\widehat{Z}^{\xi^j}f$ and $|\xi^j| \leq N-3$, the pointwise decay estimates on $\int_v z^2 |Y| dv$ are given by \eqref{decayf2}. Fix $i \in \llbracket 1 , |I| \rrbracket$ and note that
$$ \TT_F(L_i) \hspace{1mm} = \hspace{1mm} \TT_F(\sqrt{z}^{N+1-(1+2\eta)\beta^i_P}G_i) \hspace{1mm} = \hspace{1mm} \sqrt{z}^{N-1-(1+2\eta)\beta^i_P-2} F(v, \nabla_v z)G_i+\sqrt{z}^{N+1-(1+2\eta)\beta^i_P} \TT_F(G_i).$$
Following the computations of Subsection \ref{secweight}, we have
$$ \left|  F(v, \nabla_v z)  \sqrt{z}^{N-1-(1+2\eta)\beta^i_P} G_i \right| \hspace{1mm} \lesssim \hspace{1mm} \sqrt{\epsilon} \frac{v^0}{\tau_+} z  \sqrt{z}^{N-1-(1+2\eta)\beta^i_P} | G_i| \hspace{1mm} \lesssim \hspace{1mm} \sqrt{\epsilon} \frac{v^0}{\tau_+} |L_i| .$$ 
By Lemma \ref{L2bilan}, $ \sqrt{z}^{N+1-(1+2\eta)\beta^i_P} \TT_F(G_i)$ can be written as a linear combination of the following terms. \newline
$\bullet$ Those coming from $BW$,
\begin{align*}
 & \sqrt{z}^{N+1-(1+2\eta)\beta^i_P} \frac{v^{\mu}}{v^0} \mathcal{L}_{Z^{\gamma}}(F)_{\mu \nu}  W_j , \hspace{5mm} \text{with} \hspace{5mm}  \xi^j_P  \leq \beta^i_P+1, \\
 & \sqrt{z}^{N+1-(1+2\eta)\beta^i_P} x^{\theta}\frac{v^{\mu}}{v^0} \mathcal{L}_{Z^{\gamma}}(F)_{\mu \nu}  W_k , \hspace{5mm} \text{with} \hspace{5mm} \xi^k_P  \leq \beta^i_P
 \end{align*}
and where $|\gamma| \leq N$. Let $(\alpha, \underline{\alpha}, \rho, \sigma)$ be the null decomposition of $\mathcal{L}_{Z^{\gamma}}(F)$ and recall that for all $1 \leq q \leq |\mathfrak{I}|$, $Y_q=\sqrt{z}^{N+3-(1+2\eta)\xi^q_P} W_q$. Then, as $2 \geq 1+2\eta$ and $\xi^j_P  \leq \beta^i_P+1$,
$$ \left| \sqrt{z}^{N+1-(1+2\eta)\beta^i_P} \frac{v^{\mu}}{v^0} \mathcal{L}_{Z^{\gamma}}(F)_{\mu \nu}  W_j \right| \hspace{1mm} \lesssim \hspace{1mm} \left| \mathcal{L}_{Z^{\gamma}}(F) \right| |\sqrt{z}^{N+3-(1+2\eta)\xi^j_P}  W_j| \hspace{1mm} \lesssim \hspace{1mm} \left( |\alpha|+|\rho|+|\sigma|+|\underline{\alpha}| \right) |Y|.$$
Using first $|x^{\theta}| \leq \tau_+$, \eqref{eq:calculmachin} as well as $\xi^k_P \leq \beta^i_P$ and then $\tau_-+\tau_+\frac{|v^A|+v^{\underline{L}}}{v^0} \lesssim z$ (see Lemma \ref{weights1}), we get
\begin{align*}
 \left| \sqrt{z}^{N+1-(1+2\eta)\beta^i_P} x^{\theta}\frac{v^{\mu}}{v^0} \mathcal{L}_{Z^{\gamma}}(F)_{\mu \nu}  W_k  \right| \hspace{1mm} & \lesssim \hspace{1mm} \tau_+ \left( |\alpha|+|\rho|+|\sigma|+\frac{|v^A|+v^{\underline{L}}}{v^0}|\underline{\alpha}| \right)  \frac{\sqrt{z}^{N+3-(1+2\eta)\xi^k_P}}{z} |W_k| \\
 & \lesssim \hspace{1mm} \frac{\tau_+}{\tau_-} \left( |\alpha|+|\rho|+|\sigma| \right) |Y_k|+|\underline{\alpha}| |Y_k|.
\end{align*}
This concludes the construction of the matrix $\overline{B}$.

$\bullet$ Those coming from $AG$, 
\begin{align*}
 & \mathfrak{Q}_1 \hspace{1mm} := \hspace{1mm} \sqrt{z}^{N+1-(1+2\eta)\beta^i_P} \frac{v^{\mu}}{v^0} \mathcal{L}_{Z^{\gamma}}(F)_{\mu \nu}  G_j , \hspace{5mm} \text{with} \hspace{5mm}  \beta^j_P+[\gamma]  \leq \beta^i_P+1, \\
 & \mathfrak{Q}_2 \hspace{1mm} := \hspace{1mm} \sqrt{z}^{N+1-(1+2\eta)\beta^i_P} x^{\theta}\frac{v^{\mu}}{v^0} \mathcal{L}_{Z^{\gamma}}(F)_{\mu \nu}  G_k , \hspace{5mm} \text{with} \hspace{5mm} \beta^k_P+[\gamma]  \leq \beta^i_P
 \end{align*}
and $|\gamma| \leq N-6$, so that the electromagnetic field can be estimated pointwise. For simplicity, let us denote again the null decomposition of $\mathcal{L}_{Z^{\gamma}}(F)$ by $(\alpha, \underline{\alpha}, \rho, \sigma)$. Recall from Proposition \ref{decayMax} that, for all $(t,x) \in V_b(T)$,
\begin{align}
|\alpha|(t,x)+|\rho|(t,x)+|\sigma|(t,x) & \lesssim \frac{\sqrt{\epsilon}}{\tau_+^{\frac{3}{2}}}, \qquad \qquad |\underline{\alpha}|(t,x) \lesssim \frac{\sqrt{\epsilon}}{\tau_+\tau_-^{\frac{1}{2}}}, \label{decaysta} \\  |\alpha|(t,x)+|\rho|(t,x)+|\sigma|(t,x) & \lesssim \frac{\sqrt{\epsilon}}{\tau_+^{2-\delta} \tau_-^{\frac{1}{2}}}, \qquad |\underline{\alpha}|(t,x) \lesssim \frac{\sqrt{\epsilon}}{\tau_+\tau_-^{\frac{3}{2}}}, \qquad \text{if $\gamma_T \geq 1$}. \label{decayextra}
\end{align} 
We also recall that for any $1 \leq q \leq |I|$, $L_q = \sqrt{z}^{N+1-(1+2\eta)\beta^q_P} G_q$. We start by treating $\mathfrak{Q_1}$. If $\beta^j_P \leq \beta^i_P$, then, using \eqref{decaysta},
$$|\mathfrak{Q}_1| \hspace{1mm} \lesssim \hspace{1mm} \left| \mathcal{L}_{Z^{\gamma}}(F) \right| \sqrt{z}^{N+1-(1+2\eta)\beta^j_P} |G_j| \hspace{1mm} \lesssim \hspace{1mm} \frac{\sqrt{\epsilon}}{\tau_+} |L_j|.$$
Otherwise, $\beta^j_P = \beta^i_P+1$ so that $\gamma_T \geq 1$. Consequently, using \eqref{eq:calculmachin} and $z \leq \tau_+$, we get
$$ |\mathfrak{Q}_1| \hspace{1mm} \lesssim \hspace{1mm} \left( |\alpha|+|\rho|+|\sigma|+\frac{|v^A|+v^{\underline{L}}}{v^0}|\underline{\alpha}| \right) \tau_+^{\frac{1}{2}+\eta} \sqrt{z}^{N+1-(1+2\eta)\beta^j_P} |G_j|.$$
Then, as $\delta+\eta \leq \frac{1}{2}$, the pointwise decay estimates \eqref{decayextra} and the inequality $|v^A|  \lesssim v^0 v^{\underline{L}}$, coming from Lemma \ref{weights1}, give
$$ |\mathfrak{Q}_1| \hspace{1mm} \lesssim \hspace{1mm} \frac{\sqrt{\epsilon}}{\tau_+^{\frac{3}{2}-\delta-\eta}} |L_j|+\sqrt{\epsilon}\frac{v^{\underline{L}}}{\tau_+^{\frac{1}{2}-\eta}\tau_-^{\frac{3}{2}}} |L_j| \hspace{1mm} \lesssim \hspace{1mm} \frac{\sqrt{\epsilon}}{\tau_+} |L_j|+\sqrt{\epsilon}\frac{v^{\underline{L}}}{\tau_-^{1+\eta}}|L_j|.$$
We now turn on $\mathfrak{Q}_2$. Remark first, using $|x^{\theta} | \leq \tau_+$ and \eqref{eq:calculmachin} that
$$ |\mathfrak{Q}_2| \hspace{1mm} \lesssim \hspace{1mm} \tau_+\left( |\alpha|+|\rho|+|\sigma|+\frac{|v^A|+v^{\underline{L}}}{v^0}|\underline{\alpha}| \right) \sqrt{z}^{N+1-(1+2\eta)\beta^i_P} |G_k|.$$
If $\beta^k_P \leq \beta^i_P-1$, the inequalities $1 \lesssim \sqrt{v^0 v^{\underline{L}}}$, $|v^A| \lesssim v^0 v^{\underline{L}}$ and $\tau_-^{\frac{1}{2}+\eta} \lesssim z^{\frac{1}{2}+\eta}$ lead, using also \eqref{decaysta}, to
\begin{align*}
|\mathfrak{Q}_2| \hspace{1mm} & \lesssim \hspace{1mm} \tau_+\sqrt{\epsilon} \left( \frac{\sqrt{v^0v^{\underline{L}}}}{\tau_+^{\frac{3}{2}}}+\frac{v^{\underline{L}}}{\tau_+\tau_-^{\frac{1}{2}}} \right) \frac{\sqrt{z}^{N+1-(1+2\eta)\beta^k_P}}{\tau_-^{\frac{1}{2}+\eta}} |G_k| \\
 & \lesssim \hspace{1mm} \sqrt{\epsilon} \left( \frac{\sqrt{v^0v^{\underline{L}}}}{\tau_+^{\frac{1}{2}}\tau_-^{\frac{1}{2}+\eta}}+\frac{v^{\underline{L}}}{\tau_-^{1+\eta}} \right)  |L_k| \hspace{1mm} \lesssim \hspace{1mm} \sqrt{\epsilon} \frac{v^0}{\tau_+} |L_k|+\sqrt{\epsilon}\frac{v^{\underline{L}}}{\tau_-^{1+\eta}}|L_k| .
 \end{align*}
 Finally, if $\beta^k_P=\beta^i_P$, we have $\gamma_T \geq 1$ and we obtain, using \eqref{decayextra} instead of \eqref{decaysta},
 \begin{align*}
|\mathfrak{Q}_2| \hspace{1mm} & \lesssim \hspace{1mm} \tau_+\sqrt{\epsilon} \left( \frac{\sqrt{v^0v^{\underline{L}}}}{\tau_+^{2-\delta} \tau_-^{\frac{1}{2}}}+\frac{v^{\underline{L}}}{\tau_+\tau_-^{\frac{3}{2}}} \right) \sqrt{z}^{N+1-(1+2\eta)\beta^k_P} |G_k| \\
 & \lesssim \hspace{1mm} \sqrt{\epsilon} \left( \frac{\sqrt{v^0v^{\underline{L}}}}{\tau_+^{1-\delta}\tau_-^{\frac{1}{2}}}+\frac{v^{\underline{L}}}{\tau_-^{\frac{3}{2}}} \right)  |L_k| \hspace{1mm} \lesssim \hspace{1mm} \sqrt{\epsilon} \frac{v^0}{\tau_+^{\frac{3}{2}-2\delta}} |L_k|+\sqrt{\epsilon}\frac{v^{\underline{L}}}{\tau_-^{\frac{3}{2}}}|L_k| .
 \end{align*}
It remains to notice that $\delta \leq \frac{1}{4}$ and $\eta \leq \frac{1}{2}$. The matrix $\overline{D}$ can be constructed from $D$ exactly as $\overline{A}$ has been constructed from $A$. Indeed, the components of the matrix $D$ contain derivatives of the electromagnetic field up to order $N-3$ and can then be estimated pointwise (see Lemma \ref{L2bilan}). This concludes the proof.
\end{proof}
Let $K$ be the solution of $\TT_F(K)+\overline{A}K+K \overline{D}=\overline{B}$ satisfying $K(0,\cdot,\cdot)=0$. Note that $KY$ satisfies $\TT_F(KY)+\overline{A}KY=\overline{B}Y$ and initially vanishes, so that $KY=L$. The goal now is to prove a sufficiently good estimate on the energy norm
$$ \E_G(t) := \sum_{i=0}^{|I|} \sum_{j=0}^{|\mathfrak{I}|} \sum_{q=0}^{|\mathfrak{I}|}   \E \left[ \left| K_i^j \right|^2 Y_q \right](t).$$
In order to apply Proposition \ref{energyf}, remark that
\begin{equation}\label{sourceinho}
\TT_F\left( |K^j_i|^2 Y_q\right) = |K^j_i |^2 \overline{D}^r_q Y_r-2\left(\overline{A}^r_i K^j_r +K^r_i \overline{D}^j_r \right) K^j_i Y_q+2 \overline{B}^j_iK^j_iY_q.
\end{equation}

\begin{Pro}\label{inhoL1}
If $\epsilon$ is small enough, we have $\E_G(t) \lesssim \epsilon (1+t)^{\delta}$ for all $t \in [0,T[$.
\end{Pro}
\begin{proof}
We use again the continuity method. Let $T_0 \in ]0,T]$ be the largest time such that $\E_G(t) \lesssim \epsilon (1+t)^{\delta}$ for all $t \in [0,T_0[$. Fix $i \in \llbracket 1, |I| \rrbracket$ and $(j,q) \in \llbracket 1, |\mathfrak{I}| \rrbracket^2$. According to the energy estimate of Proposition \ref{energyf} and \eqref{sourceinho}, we would improve the bootstrap assumption, for $\epsilon$ small enough, if we could prove that
\begin{eqnarray}\label{Aterms}
I_{\overline{A}} & := & \int_0^t \int_{\Si^b_s} \int_v \left| |K^j_i |^2 \overline{D}^r_q Y_r-2\left(\overline{A}^k_i K^j_k +K^r_i \overline{D}^j_r \right) K^j_i Y_q \right| \frac{dv}{v^0}dx ds \hspace{2mm} \lesssim \hspace{2mm} \epsilon^{\frac{3}{2}} (1+t)^{\delta}, \\
I_{\overline{B}} & := & \int_0^t \int_{\Si^b_s} \int_v \left| \overline{B}^j_iK^j_iY_q \right| \frac{dv}{v^0}dx ds \hspace{2mm} \lesssim \hspace{2mm} \epsilon^{\frac{3}{2}} (1+t)^{\delta}. \label{Bterms}
\end{eqnarray}
Let us start with \eqref{Aterms}. In view of the description of the matrices $\overline{A}$ and $\overline{D}$ (see Lemma \ref{bilaninho}), we have, applying Lemma \ref{foliationexpli},
\begin{align*}
I_{\overline{A}} \hspace{1mm} & \lesssim \hspace{1mm} \int_0^t \int_{\Si^b_s} \frac{\sqrt{\epsilon}}{\tau_+}\int_v |K|^2 |Y| dvdx ds+\int_{u=-\infty}^b \int_{C_u(t)} \frac{\sqrt{\epsilon}}{\tau_-^{1+\eta}}\int_v \frac{v^{\underline{L}}}{v^0} |K|^2 |Y| dvdC_u(t) du \\
& \lesssim \hspace{1mm} \sqrt{\epsilon} \int_0^t \frac{\E [|K|^2 |Y|](s)}{1+s}ds + \sqrt{\epsilon} \int_{u=-\infty}^b  \frac{du}{\tau_-^{1+\eta}} \sup_{u<b}\int_{C_u(t)}\int_v \frac{v^{\underline{L}}}{v^0} |K|^2|Y| dvdC_u(t) .
\end{align*} 
The estimate \eqref{Aterms} then follows from $\E [|K|^2 |Y|](s)+\sup_{u<b}\int_{C_u(s)}\int_v \frac{v^{\underline{L}}}{v^0} |K|^2|Y| dvdC_u(s) \leq \E_G(s)$ and the bootstrap assumption. Let us focus now on \eqref{Bterms}. According to Lemma \ref{bilaninho}, we have
$$\left| \overline{B}_i^j \right| \hspace{1mm} \lesssim \hspace{1mm}  \frac{\tau_+}{\tau_-}\left( \left| \alpha \left( \mathcal{L}_{Z^{\gamma}}(F) \right) \right| + \left| \rho \left( \mathcal{L}_{Z^{\gamma}}(F) \right) \right|+\left| \sigma \left( \mathcal{L}_{Z^{\gamma}}(F) \right) \right|\right) +  \left| \underline{\alpha} \left( \mathcal{L}_{Z^{\gamma}}(F) \right) \right|   , \hspace{0.8cm} \text{where} \hspace{8mm} |\gamma| \leq N.$$ 
By the bootstrap assumption \eqref{bootF1} and dropping the dependence of $\alpha$, $\underline{\alpha}$, $\rho$ and $\sigma$ in $\mathcal{L}_{Z^{\gamma}}(F)$, we have
$$ \forall \hspace{0.5mm} s \in [0,T[, \hspace{2mm} \forall \hspace{0.5mm} u < b, \hspace{1.5cm} \| |\underline{\alpha}| \|^2_{L^2(\Si^b_s)}+\int_{C_u(s)}( |\alpha|^2+|\rho|^2+|\sigma|^2) dC_u(s) \hspace{1mm} \leq \hspace{1mm} \mathcal{E}_N[F](s) \hspace{1mm} \leq \hspace{1mm} 4\epsilon.$$
Consequently, using the Cauchy-Schwarz inequality in $v$, $\int_v |Y| dv \lesssim \epsilon (1+s)^{\delta} \tau_+^{-2} \tau_-^{-1}$ and $1 \lesssim v^0 v^{\underline{L}}$,
\begin{eqnarray}
\nonumber I_{\overline{B}} & \lesssim &  \int_0^t \int_{\Si^b_s} \left(\frac{\tau_+}{\tau_-} |\alpha|+ \frac{\tau_+}{\tau_-}|\rho|+\frac{\tau_+}{\tau_-}|\sigma|+|\underline{\alpha}| \right) \left| \int_v  |Y_q|  dv \int_v \left| K^j_i \right|^2 |Y_q| \frac{dv}{(v^0)^2}  \right|^{\frac{1}{2}} dx ds \\ 
& \lesssim & (1+t)^{\frac{\delta}{2}} \int_0^t \int_{\Si^b_s} \left(\frac{\tau_+}{\tau_-} |\alpha|+ \frac{\tau_+}{\tau_-}|\rho|+\frac{\tau_+}{\tau_-}|\sigma|+|\underline{\alpha}| \right) \frac{\sqrt{\epsilon}}{\tau_+ \tau_-^{\frac{1}{2}}}\left|  \int_v \frac{v^{\underline{L}}}{v^0} \left| K^j_i \right|^2 |Y_q| dv \right|^{\frac{1}{2}} dx ds. \label{eq:inho1}
\end{eqnarray}
By the Cauchy-Schwarz inequality in $x$ and $\E_G(s) \lesssim \epsilon (1+t)^{\delta}$, we get
\begin{eqnarray}
\nonumber \int_0^t  \int_{\Si^b_s}|\underline{\alpha}| \frac{\sqrt{\epsilon}}{\tau_+ \tau_-^{\frac{1}{2}}}\left|  \int_v \frac{v^{\underline{L}}}{v^0} | K^j_i |^2 |Y_q| dv \right|^{\frac{1}{2}}  dx ds  & \lesssim &  \int_0^t \frac{\sqrt{\epsilon}}{1+s}\| |\underline{\alpha}| \|_{L^2(\Si^b_s)} \left| \E \left[| K_i^j |^2 Y_q \right] \hspace{-0.3mm} (s) \right|^{\frac{1}{2}}  ds  \\ 
& \lesssim &  \epsilon \int_0^t \frac{\left| \E_G(s) \right|^{\frac{1}{2}}}{1+s}ds \hspace{2mm} \lesssim \hspace{2mm} \epsilon^{\frac{3}{2}} (1+t)^{\frac{\delta}{2}}. \label{eq:inho2}
\end{eqnarray}
Using the null foliation $(C_u(t))_{u < b}$ of $V_b(t)$ and the Cauchy-Schwarz inequality in $(\underline{u},\omega)$, one has, for any $\zeta \in \{ \alpha , \rho , \sigma \}$,
\begin{eqnarray}
\nonumber \int_0^t  \int_{\Si^b_s}|\zeta| \frac{\sqrt{\epsilon}}{ \tau_-^{\frac{3}{2}}}\left|  \int_v \frac{v^{\underline{L}}}{v^0} | K^j_i |^2 |Y_q| dv \right|^{\frac{1}{2}}  dx ds  & \lesssim &   \int_{- \infty}^b \frac{\sqrt{\epsilon}}{\tau_-^{\frac{3}{2}}} \left| \int_{C_u(t)} |\zeta|^2 d C_u(t) \int_{C_u(t)} \int_v \frac{v^{\underline{L}}}{v^0} | K^j_i |^2 |Y_q| dv d C_u(t) \right|^{\frac{1}{2}} du  \\ \nonumber
& \lesssim &  \sqrt{\epsilon} \left| \mathcal{E}_N[F](t) \E \left[| K_i^j |^2 Y_q \right](t) \right|^{\frac{1}{2}} \int_{u=-\infty}^0  \frac{du}{\tau_-^{\frac{3}{2}}} \\ 
& \lesssim &  \epsilon^{\frac{3}{2}}  (1+t)^{\frac{\delta}{2} }.  \label{eq:inho3}
\end{eqnarray}
From \eqref{eq:inho1}, \eqref{eq:inho2} and \eqref{eq:inho3}, we finally obtain $I_{\overline{B}} \lesssim \epsilon^{\frac{3}{2}}(1+t)^{\delta}$. This concludes the improvement of the bootstrap assumption and then the proof.
\end{proof}
\subsection{End of the proof of Proposition \ref{L2estimates}}

Let $i \in I$. Using the Cauchy-Schwarz inequality in $v$, $\E_H(t) \lesssim \epsilon (1+t)^{\delta}$ and the pointwise decay estimates \eqref{decayH}, we have
$$ \left\|  \int_v z |H_i| dv \right\|^2_{L^2(\Si^b_t)} \hspace{1mm} \lesssim \hspace{1mm} \left\|  \int_v |H_i| dv \right\|_{L^{\infty}(\Si^b_t)} \left\|  \int_v z^2 |H_i| dv \right\|_{L^1(\Si^b_t)}
\hspace{1mm} \lesssim \hspace{1mm} \left\|  \frac{\epsilon}{\tau_+^{2-\delta}\tau_-} \right\|_{L^{\infty}(\Si^b_t)}  \hspace{0.5mm} \E_H(t) \hspace{1mm} \lesssim \hspace{1mm} \frac{\epsilon^2}{(1+t)^{2-2\delta}}  .$$
As $L_i =K_i^j Y_j$, the Cauchy-Schwarz inequality in $v$, $\E_G(t) \lesssim \epsilon (1+t)^{\delta}$ and $\int_v z^2 |Y| dv \lesssim \epsilon \tau_+^{-2+\delta}$ gives
\begin{eqnarray}
\nonumber \left\|  \int_v z|L_i| dv \right\|^2_{L^2(\Si^b_t)} & \lesssim & \left\|  \int_v z^2|Y| dv   \int_v \left| K_i^j \right|^2 |Y_j| dv \right\|_{L^1(\Si^b_t)}  \hspace{2mm} \lesssim \hspace{2mm} \left\|  \int_v z^2|Y| dv \right\|_{L^{\infty}(\Si^b_t)} \left\| \int_v \left| K_i^j \right|^2 |Y_j| dv \right\|_{L^1(\Si^b_t)} \\ \nonumber 
& \lesssim & \left\|  \frac{\epsilon}{\tau_+^{2-\delta}} \right\|_{L^{\infty}(\Si^b_t)} \E_G(t) \hspace{2mm} \lesssim \hspace{2mm} \frac{\epsilon^2}{(1+t)^{2-2\delta}} .
\end{eqnarray}
To conclude the proof of Proposition \ref{L2estimates}, use $2\delta \leq \frac{1}{2}$ and notice that for all $N-2 \leq |\beta| \leq N$, there exists $i \in I$ verifying $ \widehat{Z}^{\beta} f = H_i+G_i$ and that $|G_i| \leq |L_i|$.

\renewcommand{\refname}{References}
\bibliographystyle{abbrv}
\bibliography{biblio}

\end{document}